\documentclass[reqno,10pt,letterpaper]{amsart}

\usepackage{lipsum}
\usepackage{amsmath}
\usepackage{amssymb}
\usepackage{amsthm}
\usepackage{mathrsfs}
\usepackage{accents}
\usepackage{calc}
\usepackage{arydshln}
\usepackage{upgreek}
\usepackage{slashed}
\usepackage{xifthen}
\usepackage{graphicx}
\usepackage[inline]{enumitem}
\usepackage{stmaryrd}

\usepackage{mathtools}
\usepackage{derivative}
\usepackage{tikz}
\usepackage{pgfplots}
\usetikzlibrary{calc, arrows.meta}
\usepgfplotslibrary{fillbetween}
\usepackage{subcaption}

\usepackage{xcolor}
\definecolor{winered}{rgb}{0.6,0,0}
\definecolor{lessblue}{rgb}{0,0,0.7}

\usepackage[pdftex,colorlinks=true,linkcolor=winered,citecolor=lessblue,urlcolor=lessblue,breaklinks=true,bookmarksopen=true]{hyperref}

\makeatletter
\newcommand{\myitem}[2]{\item[\rm(#2)]\def\@currentlabel{#2}\label{#1}}
\makeatother

\usepackage{titletoc}

\makeatletter

\def\@tocline#1#2#3#4#5#6#7{
\begingroup
  \par
    \parindent\z@ \leftskip#3 \relax \advance\leftskip\@tempdima\relax
                  \rightskip\@pnumwidth plus 4em \parfillskip-\@pnumwidth
    \ifcase #1 
       \vskip 0.6em \hskip 0em 
       \or
       \or \hskip 0em 
       \or \hskip 1em 
    \fi%
    #6
    \nobreak\relax{\leavevmode\leaders\hbox{\,.}\hfill}
    \hbox to\@pnumwidth {\@tocpagenum{#7}}
  \par
\endgroup
}

 \def\l@section{\@tocline{0}{0pt}{0pc}{}{}}

\renewcommand{\tocsection}[3]{%
  \indentlabel{\@ifnotempty{#2}{ 
    \ignorespaces\bfseries{#2. #3}}}
  \indentlabel{\@ifempty{#2}{\ignorespaces\bfseries{#3}}{}} 
    \vspace{1.5pt}}

\renewcommand{\tocsubsection}[3]{%
  \indentlabel{\@ifnotempty{#2}{
    \ignorespaces#2. #3}}
  \indentlabel{\@ifempty{#2}{\ignorespaces #3}{}}
    \vspace{1.5pt}}

\renewcommand{\tocsubsubsection}[3]{%
  \indentlabel{\@ifnotempty{#2}{
    \ignorespaces#2. #3}}
  \indentlabel{\@ifempty{#2}{\ignorespaces #3}{}}
    \vspace{1.5pt}}

\makeatother

\makeatletter
\def\@nomenstarted{0}
\newlength{\@nomenoldtabcolsep}

\newcommand{\nomenstart}
  {%
    \def\@nomenstarted{1}%
    \setlength{\@nomenoldtabcolsep}{\tabcolsep}%
    \setlength{\tabcolsep}{3.5pt}%
    \begin{longtable}{p{0.11\textwidth} p{0.86\textwidth}}
  }

\newcommand{\nomenitem}[2]{%
    \ifcase\@nomenstarted%
      \or 
      \or \\ 
    \fi%
    #1\,{\leavevmode\leaders\hbox{\,.}\hfill} & #2%
    \def\@nomenstarted{2}%
  }%
\newcommand{\nomenend}
  {\\%
      \end{longtable}%
      \setlength{\tabcolsep}{\@nomenoldtabcolsep}%
      \def\@nomenstarted{0}%
  }
\makeatother

\makeatletter
\newcommand{\BIG}{\bBigg@{3.5}}
\newcommand{\vast}{\bBigg@{4}}
\newcommand{\Vast}{\bBigg@{5}}
\newcommand{\VAST}[1]{\bBigg@{#1}}
\makeatother

\allowdisplaybreaks

\numberwithin{equation}{section}
\numberwithin{figure}{section}
\newtheorem{thm}{Theorem}[section]

\newtheorem{prop}[thm]{Proposition}
\newtheorem{lemma}[thm]{Lemma}
\newtheorem{cor}[thm]{Corollary}

\newtheorem*{qu}{Question}

\newtheorem*{thm*}{Theorem}
\newtheorem*{prop*}{Proposition}
\newtheorem*{cor*}{Corollary}
\newtheorem*{conj*}{Conjecture}
\newtheorem*{qu*}{Question}

\theoremstyle{definition}
\newtheorem{definition}[thm]{Definition}

\theoremstyle{remark}
\newtheorem{rmk}[thm]{Remark}

\makeatletter
\newcommand{\fakephantomsection}{%
  \Hy@MakeCurrentHref{\@currenvir.\the\Hy@linkcounter}
  \Hy@raisedlink{\hyper@anchorstart{\@currentHref}\hyper@anchorend}%
  \Hy@GlobalStepCount\Hy@linkcounter%
}
\makeatother

\newcommand{\mc}{\mathcal}

\newcommand{\cC}{\mc C}
\newcommand{\cD}{\mc D}
\newcommand{\cE}{\mc E}
\newcommand{\cF}{\mc F}

\newcommand{\cL}{\mc L}

\newcommand{\cO}{\mc O}

\newcommand{\cS}{\mc S}

\newcommand{\cX}{\mc X}
\newcommand{\cY}{\mc Y}

\newcommand{\ms}{\mathscr}

\newcommand{\sS}{\ms S}

\newcommand{\C}{\mathbb{C}}
\newcommand{\N}{\mathbb{N}}
\newcommand{\R}{\mathbb{R}}
\newcommand{\Z}{\mathbb{Z}}

\newcommand{\ran}{\operatorname{ran}}

\renewcommand{\Re}{\operatorname{Re}}

\newcommand{\mathspan}{\operatorname{span}}
\newcommand{\supp}{\operatorname{supp}}

\newcommand{\aug}{{\rm aug}}

\makeatletter
\@ifundefined{laplace}{\newcommand{\laplace}{\Delta}}{}
\makeatother

\newcommand{\eps}{\epsilon}

\newcommand{\hra}{\hookrightarrow}
\newcommand{\la}{\langle}

\newcommand{\ol}{\overline}
\newcommand{\pa}{\partial}
\newcommand{\dd}{{\mathrm d}}
\newcommand{\ra}{\rangle}
\newcommand{\spec}{\operatorname{spec}}

\newcommand{\wt}{\widetilde}
\newcommand{\xra}{\xrightarrow}

\DeclarePairedDelimiter{\abs}{|}{|}
\DeclarePairedDelimiter{\norm}{\|}{\|}
\DeclarePairedDelimiter{\jbr}{\langle}{\rangle}
\DeclarePairedDelimiter{\iprod}{\langle}{\rangle}
\DeclarePairedDelimiter{\set}{\{}{\}}
\newcommand*{\qms}{\mathcal{X}}
\newcommand*{\es}{\mathcal{Y}}

\newcommand*{\chfun}{\mathbf{1}}
\newcommand*{\dirlap}[1]{\laplace_{#1}^{\rm D}}

\newcommand{\pfstep}[1]{\(\bullet\)\ \underline{\textit{#1}}}

\newcommand{\ff}{\mathrm{ff}}

\newcommand{\cp}{{\mathrm{c}}}

\newcommand{\CI}{\cC^\infty}

\newcommand{\CIc}{\cC^\infty_\cp}

\newcommand{\openbigpmatrix}[1]
  {%
    \def\@bigpmatrixsize{#1}%
    \addtolength{\arraycolsep}{-#1}%
    \begin{pmatrix}%
  }
\newcommand{\closebigpmatrix}
  {%
    \end{pmatrix}%
    \addtolength{\arraycolsep}{\@bigpmatrixsize}%
  }

\newlength{\enummargin}

\newcommand{\usref}[1]{{\upshape\ref{#1}}}

\DeclareGraphicsExtensions{.mps}

\makeatletter
\newcommand*{\fwbw}[1]{\expandafter\@fwbw\csname c@#1\endcsname}
\newcommand*{\@fwbw}[1]{\ifcase #1 \or {\rm fw}\or {\rm bw}\fi}
\AddEnumerateCounter{\fwbw}{\@fwbw}
\makeatother

\begin{document}

\title{The Asymptotic Behavior of Simple Eigenvalues of Particle-in-Well Systems}

\begin{abstract}
    The particle in a well in dimension one is a classical problem in quantum mechanics. We study higher-dimensional analogues of the problem, where the well is a smooth domain in $\mathbb{R}^d$. We show that simple eigenvalues and eigenfunctions of the corresponding Schr\"odinger operator depend smoothly on the square root $h$ of the inverse depth of the well and provide an explicit first-order expansion of the eigenvalues at $h=0$.

    Our proof consists of two steps. In the first step, we construct $\mathcal{O}(h^\infty)$ quasimodes (approximate eigenfunctions) on a resolution of $[0,1)_h\times\mathbb{R}^d$ which allows us to capture fine structure near the boundary of the well. The second step corrects these quasimodes to true eigenfunctions via a fixed point argument.
\end{abstract}

\date{\today}


\author{Peter Hintz}
\address{Department of Mathematics, ETH Z\"urich, R\"amistrasse 101, 8092 Z\"urich, Switzerland}
\address{Department of Mathematics, Pennsylvania State University, 54 McAllister St, State
College,\newline PA 16801, United States}
\email{peter.hintz@math.ethz.ch}
\email{phintz@psu.edu}

\author{Aaron Moser}
\address{Department of Mathematics, ETH Z\"urich, R\"amistrasse 101, 8092 Z\"urich, Switzerland}
\address{Department of Mathematics, Massachusetts Institute of Technology, 77 Massachusetts
Avenue,\newline Cambridge, MA 02139-4307, United States}
\email{aaron.moser@math.ethz.ch}
\email{maaron@mit.edu}

\maketitle


\section{Introduction}
\label{SI}

\subsection{Setup and main results}

Fix a bounded smooth domain \(\Omega\subset \R^d\). For \(h > 0\), we define the
\emph{particle-in-well operator} on \(\Omega\) to be
\[
  P_h\coloneqq -\laplace + h^{-2}\chfun_{\Omega^\complement},
\]
where \(\laplace = \sum_{i=1}^d \partial_{x^i}^2\) is the
negative semidefinite Laplace operator on \(\R^d\) and \(\chfun_{\Omega^\complement}\) is multiplication
with the characteristic function of the complement \(\Omega^\complement\) of $\Omega$. This operator is an unbounded self-adjoint operator on $L^2(\R^d)$ with domain $H^2(\R^d)$ (and quadratic form domain $H^1(\R^d)$). The spectrum of \(P_h\) in the interval \([0,h^{-2})\) is discrete (Lemma~\ref{LemmaSpec}), so $\spec(P_h)\cap[0,\frac12 h^{-2}]$ consists of \(N_h\in\N\) eigenvalues.
\[
    \lambda_1^h \le\lambda_2^h\le \cdots \le \lambda_{N_h}^h.
\]
Write \(-\dirlap{\Omega}\) for the Dirichlet Laplacian on \(\Omega\), that is, the
self-adjoint extension of the operator \(-\laplace\) on \(\CIc(\Omega)\) to \(H^2(\Omega)\cap H^1_0(\Omega)\), and denote its eigenvalues by
\[
    0 < \lambda_1^{\rm D}\le\lambda_2^{\rm D}\le\lambda_3^{\rm D}\le \cdots.
\]
The corresponding eigenfunctions satisfy $u_n\in\CI_0(\bar\Omega):=\{u\in\CI(\bar\Omega)\colon u|_{\pa\Omega}=0\}$. It is easy to show that $N_h\to\infty$ as $h\to 0$, and that, moreover, for all $n\in\N$,
\begin{equation}
\label{EqIConv}
  \lambda_n^h \to \lambda_n^{\rm D},\quad h\to 0.
\end{equation}
(See Proposition~\ref{PropEVConv}.) Our main result shows that eigenvalues $\lambda_n^h$ converging to a simple Dirichlet eigenvalue are, in fact, smooth down to $h=0$:

\begin{thm}[Smooth dependence of simple eigenvalues on \(h\)]\label{thm:smooth_evs}
    Let \(n\in\N\) be such that $\lambda_n^{\rm D}$ is a simple eigenvalue of $-\dirlap{\Omega}$. Let \(h_n\) be such that \(N_{h_n}\ge n\). Then there exist $h_n'\in(0,h_n]$ and a smooth function $\lambda_n\colon[0,h'_n]\to(0,\infty)$ such that \(\lambda_n(0) = \lambda_n^{\rm D}\) and \(\lambda_n(h)=\lambda_n^h\in \spec(P_h)\).
\end{thm}

If \(d = 1\), then operators such as $P_h$ are well-known from the ``particle in a well'' model in standard undergraduate quantum mechanics. In this model, it is well-known that for any fixed $n\in\N$ the $n$-th eigenvalue \(\lambda_n^h\) of $P_h$ exists when \(h\) is small enough, and furthermore $\lambda_n^h\to\lambda_n^{\rm D}$ as \(h\to 0\). One can show a similar result for \(d\)-dimensional balls using explicit computations involving Bessel functions. See Appendices \ref{SAppIntervals}, \ref{SAppDisks}, and \ref{SAppBalls} for more details. Our Theorem~\ref{thm:smooth_evs} generalizes these special cases, for simple eigenvalues, to arbitrary bounded smooth domains.

We can compute the first order Taylor expansion of $\lambda_n(h)$ explicitly:

\begin{prop}[First order expansion of the eigenvalue]\label{PropIFirstOrderExp}
    Let \(n\in\N\) be such that \(\lambda_n^{\rm D}\) is a simple eigenvalue of \(-\dirlap{\Omega}\) with
    corresponding \(L^2\)-normalized eigenfunction \(u_n\in \CI_0(\bar\Omega)\). Then $\lambda_n^h$ is simple for all sufficiently small $h$, and
    \begin{equation}
    \label{EqIFirstOrderExp}
        \lambda_n^h = \lambda_n^{\rm D} - h\norm{\partial_{\nu} u_n}_{L^2(\pa\Omega)}^2 +
        \cO(h^2)\quad\text{as }h\to 0.
    \end{equation}
\end{prop}

If \(\Omega=B_a(0)\subset\R^d\) is a ball of radius \(a > 0\), Proposition~\ref{PropIFirstOrderExp} has an explicit form. Let \(n\in\N\) and suppose \(\lambda_n^{\rm D}\) is an eigenvalue of \(-\dirlap{\Omega}\) with corresponding eigenfunction \(u_n\in\CI_0(\bar\Omega)\). Rellich's identity~\cite{Rellich1940} for eigenvalues on balls states that
\begin{equation}\label{EqBallFirstOrder}
    \lambda_n^{\rm D} = \frac{1}{4}\frac{\int_{\pa B_a(0)}(\partial_\nu
    u_n)^2\partial_\nu(\abs{x}^2)\odif{x}}{\norm{u_n}_{L^2(B_a(0))}^2}.
\end{equation}
Since \(\partial_\nu(\abs{x}^2) = 2a\), Proposition~\ref{PropIFirstOrderExp} implies that on balls each simple eigenvalue \(\lambda_n^{\rm D}\) spawns an eigenvalue $\lambda_n^h$ of $P_h$ with
\begin{equation}
\label{EqBallCorr}
    \lambda_n^h = \lambda_n^{\rm D} - \frac{2\lambda_n^{\rm D}}{a}h + \cO(h^2) \quad\text{as }h\to 0.
\end{equation}
Figure~\ref{FigEigvalConv} shows the first nine eigenvalues \(\lambda_j^h\), \(j=1,\dots, 9\), of $P_h$ in comparison to the corresponding \(\lambda_j^{\rm D}\), \(j=1,\dots, 9\), for \(\Omega=(-2, 2)\) and \(\Omega=B_2(0)\subset \R^2\). Figure~\ref{FigEigvalDiff} shows the difference $\lambda_j^D-\lambda_j^h$ and compares it to the $\cO(h)$ correction term in~\eqref{EqBallCorr}. Note here that while in the one-dimensional case all eigenvalues have multiplicity \(1\), this is not the case for the two-dimensional ball. However, due to the symmetry of the domain, the particle-in-well operator and the Laplace operator share the same multiplicity structure of their eigenvalues.

The monotonicity of the eigenvalues $\lambda_j^h$ of $P_h$ as we let $h\to 0$ observed in Figure~\ref{FigEigvalConv} holds in general (Lemma~\ref{LemmaEVMono}).

Our second main result concerns the family of $L^2$-normalized eigenfunctions of $P_h$ with eigenvalue $\lambda_n^h$. The notation for blow-ups and lifts used in the following result is recalled in~\S\ref{SsBlowup}.

\begin{thm}[Behavior of the eigenfunctions as $h\to 0$]
\label{ThmIEigenfn}
  Define $\tilde M=[\R_x^d\times[0,1)_h;\pa\Omega\times\{0\}]$, and denote the lifts of $\Omega\times[0,1)$ and $\Omega^\complement\times[0,1)$ by $I$ and $E$, respectively; denote moreover the lifts of $\Omega\times\{0\}$ and $\Omega^\complement\times\{0\}$ by $\tilde\Omega\cong\Omega$ and $\tilde\Omega^\complement\cong\Omega^\complement$, respectively. (See Figure~\usref{fig:blowup}.) Let \(n\in\N\) be such that \(\lambda_n^{\rm D}\) is a simple eigenvalue of \(-\dirlap{\Omega}\). Then there exists a continuous function $u\colon\tilde M\to\R$ with the following properties:
  \begin{enumerate}
  \item\label{ItIEigenfnTrue} for all sufficiently small $h>0$, the restriction of $u$ to $\R^d\times\{h\}\cong\R^d$ is an $L^2$-normalized eigenfunction of $P_h$ (which by elliptic regularity lies in $H^2(\R^d)$) corresponding to the eigenvalue $\lambda_n^h$;
  \item $u|_{I^\circ}$ extends smoothly to $I$, and $u|_{E^\circ}$ extends smoothly to $E$;
  \item $u|_{\tilde\Omega}=u_n$ is the $n$-th Dirichlet eigenfunction;
  \item $u$ vanishes to infinite order at $\tilde\Omega^\complement$ and vanishes rapidly as $|x|\to\infty$, uniformly as $h\to 0$.
  \end{enumerate}
\end{thm}

Parts of Theorem~\ref{ThmIEigenfn} can easily be phrased without blow-ups, in particular: one can select $L^2$-normalized eigenfunctions $u_n^h$ in such a way that, for sufficiently small $h_0>0$, the function $(0,h_0)\times\Omega\ni(h,x)\mapsto u_n^h(x)$ extends to a smooth function on $[0,h_0)\times\Omega$ which restricts to $h=0$ as the $n$-th Dirichlet eigenfunction; and similarly $u_n^h|_{\Omega^\complement}$ extends to a smooth function on $[0,h_0)\times\Omega^\complement$ which vanishes to infinite order at $h=0$.

\begin{figure}
    \centering
    \begin{subfigure}{0.75\textwidth}
        \centering
        \includegraphics[width=\textwidth]{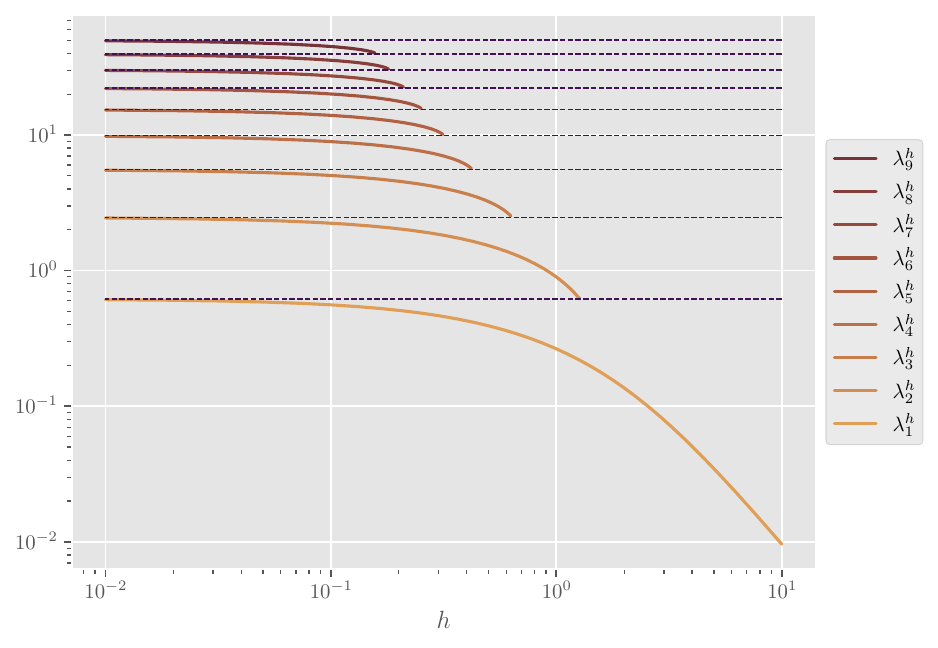}
        \caption{\(\Omega = (-2, 2)\subset \R\).}\label{FigEigvalConv1D}
    \end{subfigure}
    \hfill
    \begin{subfigure}{0.75\textwidth}
        \centering
        \includegraphics[width=\textwidth]{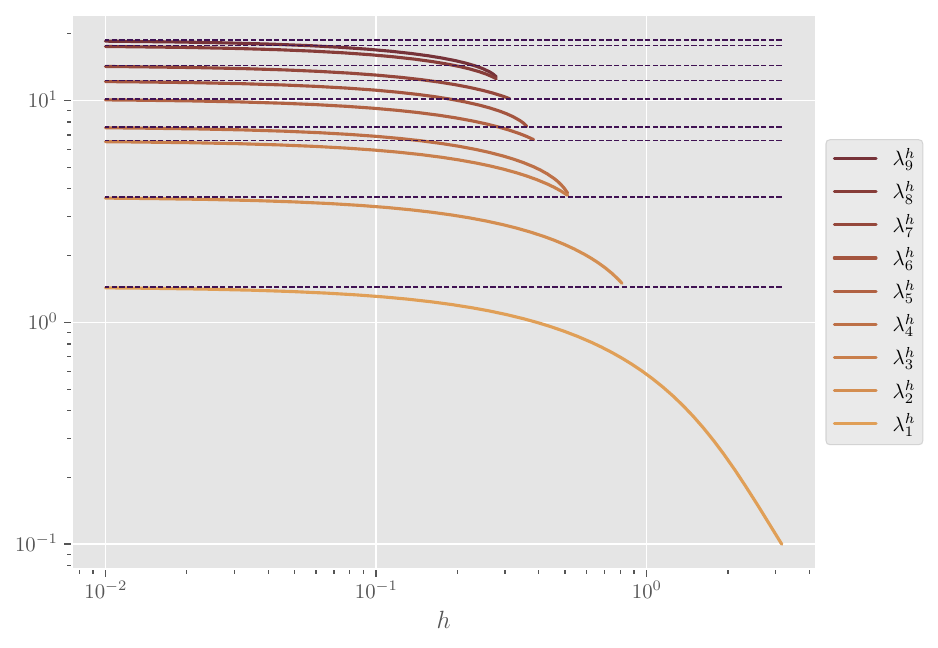}
        \caption{\(\Omega = B_2(0)\subset \R^2\).}\label{FigEigvalConv2D}
    \end{subfigure}
    \caption{
        Plots of the functions \(h\mapsto \lambda_j^h\), \(j=1,\dots, 9\), for balls of radius 2 in
        dimensions 1 and 2. The dashed lines represent the corresponding Dirichlet eigenvalues
        \(\lambda_j^{\rm D}\), \(j=1,\dots, 9\), on \(\Omega\). The eigenvalues are determined as
        the zeros of secular functions, see Appendices~\ref{SAppIntervals} and \ref{SAppDisks}.
        Decreasing $h$ increases the potential barrier and thus allows for more and more eigenvalues
        to appear. Intriguingly, every eigenvalue \(\lambda_j^h\) with \(j > 1\) starts its
        existence (at the maximal value of $h$) as an eigenvalue \(\lambda_i^{\rm D}\) with \(i <
        j\). In the one-dimensional case we always have \(j=i+1\) but in the two-dimensional case we
        have \(j = i + 1\) for \(j = 2, 3\), \(j=i+2\) for \(j=4, 5, 6, 7, 8\), and $j=i+3$ for $j =
        9$. Incidentally, the indices \(j=4, 9\) are the only indices among the ones considered for
        which the eigenvalue \(\lambda_j^{\rm D}\) has multiplicity \(1\), all others having
        multiplicity \(2\).
    }
\label{FigEigvalConv}
\end{figure}

\begin{figure}[!ht]
    \centering
    \begin{subfigure}{0.75\textwidth}
        \centering
        \includegraphics[width=\textwidth]{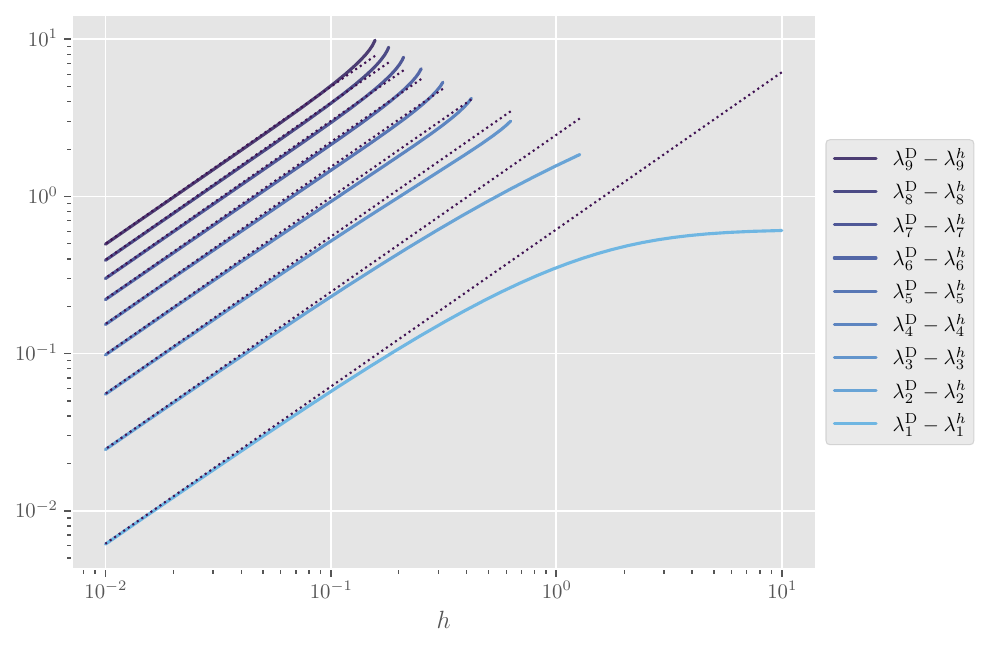}
        \caption{\(\Omega = (-2, 2)\subset \R\).}\label{FigEigvalDiff1D}
    \end{subfigure}
    \hfill
    \begin{subfigure}{0.75\textwidth}
        \centering
        \includegraphics[width=\textwidth]{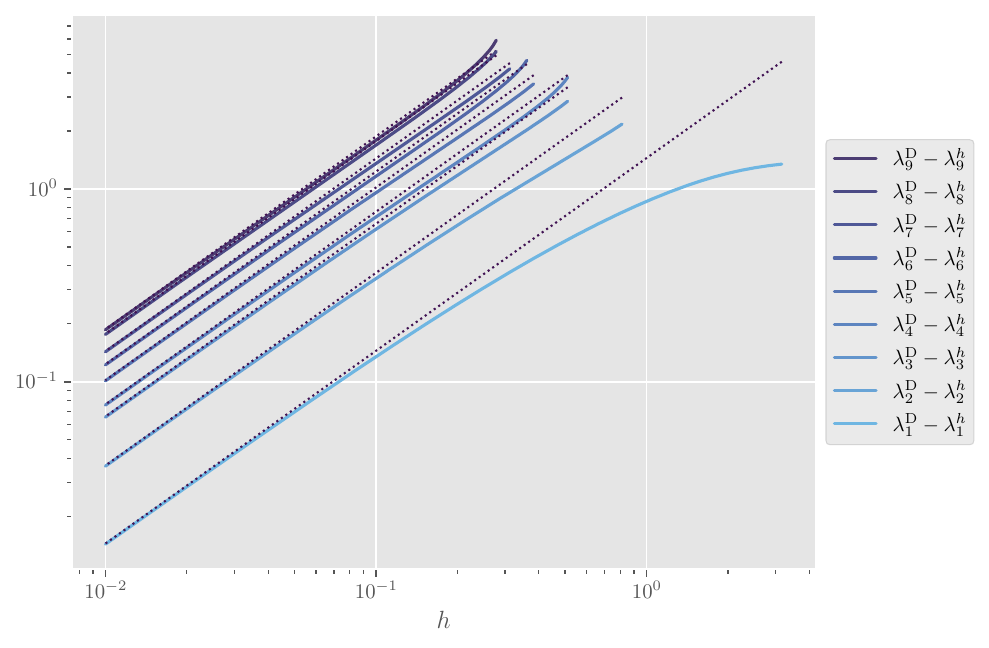}
        \caption{\(\Omega = B_2(0)\subset \R^2\).}\label{FigEigvalDiff2D}
    \end{subfigure}
    \caption{
        Plots of the functions \(h\mapsto \lambda_j^{\rm D} - \lambda_j^h\), \(j=1,\dots, 9\), for
        balls of radius 2 in dimensions 1 and 2. The dotted lines represent the corresponding expected
        first order term given by \(h\mapsto \lambda_j^{\rm D}h\), \(j=1,\dots, 9\), see
        \eqref{EqBallCorr}.}\label{FigEigvalDiff}
\end{figure}
\newpage

\begin{rmk}[Smooth potentials]
\label{RmkIPot}
  If $V\in\CI(\bar\Omega)$ is a real-valued potential, our methods apply, with purely notational changes, also to the operator $-\Delta+V\chfun_\Omega+h^{-2}\chfun_{\Omega^\complement}$, and yield Theorems~\ref{thm:smooth_evs} and \ref{ThmIEigenfn}, \emph{mutatis mutandis}.
\end{rmk}

Our proof of Theorem~\ref{ThmIEigenfn} consists of two steps. In the first step (\S\ref{SConstructionQuasimodes}), we present an algorithmic and (in principle) explicit procedure for the construction of increasingly accurate \emph{quasimodes}:

\begin{thm}[Family of quasimodes]\label{thm:exquasimodes}
    Let \(n\in\N\) and suppose \(\lambda_n^{\rm D}\) is a simple eigenvalue of \(-\dirlap{\Omega}\). Then there exist a continuous function \(\tilde{u}_n\) on $(0,1)_h\times\R^d_x$ and a function \(\tilde\lambda_n\in \CI([0,1))\) with the following properties:
    \begin{enumerate}
        \item\label{ItIQM1} the function \((\tilde u_n)_h:=\tilde{u}_n(h,\cdot)\) depends smoothly on \(h\);
        \item\label{ItIQM2} for each fixed \(h\), the function \((\tilde{u}_n)_h = \tilde{u}_n(\cdot, h)\) is
            continuous and compactly supported in \(x\), and it is smooth on \(\Omega\) and
            \((\Omega^\complement)^\circ\);
        \item\label{ItIQM3} we have \(\norm{(\tilde{u}_n)_h}_{L^2(\R^d)} = 1 +\cO(h)\) and $\tilde\lambda_n(0)=\lambda_n^{\rm D}$;
        \item\label{ItIQM4} the pair \((\tilde{u}_n, \tilde\lambda_n)\) is an $\cO(h^\infty)$-quasimode (relative to $H^{-1}(\R^d)$), i.e.,
        \[
            \norm{(P_h - \tilde\lambda_n(h)) \tilde{u}_n}_{H^{-1}(\R^d)} = \cO(h^{\infty}).
        \]
    \end{enumerate}
\end{thm}

In fact, $\tilde u_n$ will have the same properties as $u$ in Theorem~\ref{ThmIEigenfn} except of course for property~\eqref{ItIEigenfnTrue} in Theorem~\ref{ThmIEigenfn}, but it is in addition compactly supported in $x$. Essentially by the self-adjointness of $P_h$, the existence of such quasimodes implies the existence of eigenvalues $\cO(h^\infty)$-close to \(\tilde\lambda_n(h)\) (see Lemma~\ref{LemmaQM}); since $\lambda_n^h=\lambda_n^{\rm D}+o(1)$ is the unique eigenvalue of $P_h$ close to $\lambda_n^{\rm D}$ for sufficiently small $h>0$ (due to the simplicity of $\lambda_n^{\rm D}$ and the convergence~\eqref{EqIConv} of all eigenvalues), this implies $\lambda_n^h=\tilde\lambda_n(h)+\cO(h^\infty)$.

The construction of $\tilde u_n$ proceeds via the construction of its Taylor expansion on the manifold with corners $\tilde M$ of Theorem~\ref{ThmIEigenfn}; the equations one needs to solve at each step involve $-\dirlap{\Omega}-\lambda_n^{\rm D}$ (on $\Omega$) as well as the model operator $-\partial_{\hat\rho}^2+H(\hat\rho)$ (on $\R_{\hat\rho}\times\pa\Omega$) that describes the fine structure of $P_h$ near $\pa\Omega$. The explicit computation of the first term in the small-$h$ expansion of $\tilde u_n$ yields Proposition~\ref{PropIFirstOrderExp}.

The technique used in this article to construct quasimodes closely resembles the methods explained
in the very instructive lecture notes by Daniel Grieser~\cite{Grieser2017}. We refer the reader
to these notes for a broader overview of the technique, its applications, and the
surrounding literature.

At this point, we can already deduce that $P_h$ has an eigenvalue $\lambda_n^h$ that differs from a polynomial of degree $N$ in $h$, with first two terms given by~\eqref{EqIFirstOrderExp}, by an error of size $\cO(h^{N+1})$. But since functions $[0,1)\to\R$ of pointwise size $\cO(h^\infty)$ need not even be continuously differentiable at $h=0$, this is insufficient for obtaining the smoothness of $\lambda_n(h)=\lambda_n^h$ asserted in Theorem~\ref{thm:smooth_evs}. The second step (\S\ref{STCorrectingQuasimodes}) of the proof of Theorem~\ref{thm:smooth_evs} (and also of Theorem~\ref{ThmIEigenfn}) resolves this issue: starting with the quasimode from Theorem~\ref{thm:exquasimodes}, we solve a coupled (nonlinear) system of equations, of elliptic character, for the $\cO(h^\infty)$ corrections to the approximate eigenfunctions and approximate eigenvalues. (See~\eqref{EqTEq2} for a first version of this system.) Since this system becomes singular in $\cO(h)$-neighborhoods of $\pa\Omega$, considerable care is required in the setup of the appropriate $h$-dependent function spaces for uniform elliptic estimates to hold, and for the contraction mapping principle to become applicable.

\subsection{Related literature}

The asymptotic analysis of eigenvalues of Schr\"odinger operators has a long tradition in
mathematics and physics. However, the mathematical literature on the particular problem treated in
this article appears to be fairly sparse. If however \(h^{-2}\chfun_{\Omega^\complement}\)
is replaced by a smooth potential with a well, there is a large body of literature on a wide range
of semiclassical asymptotics. Particularly noteworthy in this context is the work of
Helffer--Sj\"ostrand \cite{HelfferSjostrandWells1, HelfferSjostrandWells2, HelfferSjostrandWells3}
which develops an extensive theory of the semiclassical spectral behavior of Schr\"odinger operators
with smooth potential wells. The relevant techniques are typically based on microlocal analysis and
cannot be easily extended to singular potentials. For linear Schrödinger operators with non-smooth
potentials, we mention the results for
\(\cC^1\)-potentials in dimension \(1\) given in \cite{Zou2024}, where the potentials are assumed
to vanish on an interval and grow quadratically at infinity. Even more singular potentials in the
form of $\delta$-like potentials have been considered as models for leaky quantum graphs; see the
review paper \cite{Exner2008} for more details on their spectral properties and references. The
scattering pendant for Schrödinger operators with $\delta$-potentials has been studied by Galkowski
in \cite{Galkowski2019b}.

There is a considerable body of work on the scattering theory of finite well potentials for
intervals and the two-dimensional ball. In particular, connecting to the spectral theory, the
articles \cite{Nussenzveig1959} and \cite{MeetenDochertyMorozov2019} provide some interesting
numerical studies on the behavior of poles of the scattering matrix as well as the scattering
resonances, and the transition of eigenvalues to resonances as the well becomes more shallow. The
mathematical treatment thereof has been subject of a more recent paper
\cite{ChristiansenDatchevGriffin2024} in which they provide fine information on the transition of
eigenvalues to resonances for a well in the shape of a disk in dimension two.

We remark that the eigenvalue problem for \(P_h\) can be restated as something akin to a
transmission problem for an inclusion of a shape \(\Omega\): the problem is to find \(\lambda\)
(depending on $h$ and close to a Dirichlet eigenvalue $\lambda_n^{\rm D}$ of $\Omega$) such that
\begin{equation*}
    \begin{cases}
        (-\laplace - \lambda)u_1 = 0 & \text{in }\Omega,\\
        (-\laplace + h^{-2} - \lambda)u_2 = 0 & \text{in }\Omega^\complement,\\
        u_1 = u_2 & \text{on }\pa\Omega,\\
        \pa_\nu u_1 = \pa_\nu u_2 & \text{on }\pa\Omega,\\
        u_2(x) \to 0 & \text{as }\abs{x}\to \infty.
    \end{cases}
\end{equation*}
In this context, we would like to highlight another work of Galkowski on the distribution of
scattering resonances for transmission problems \cite{Galkowski2019}. Note, however, that in our
case the speed of light on $\Omega^\complement$ depends on the value of $\lambda$, adding an
additional nonlinearity that is not present in usual transmission problems. Another related
spectral problem is presented in the article \cite{GodinKoralovVainberg2025}, in which the authors
investigate the spectral properties of elliptic operators with inhomogeneous coefficients with a
view towards understanding wave propagation in high-contrast media.

In physical applications, the particle-in-well system is a useful simplified model to understand
quantum dots and semiconductors. Since the body of literature on both of these topics is rather
large, we shall give here just a few examples of publications that use particle-in-well systems
directly. In the context of quantum dots, some studies on the validity and predictive power of
finite potential wells have been given in \cite{Nosaka1991, PellegriniMatteiMazzoldi2005,
JeongShin2018}. On the side of semiconductors, the book \cite{Shik1998} and the
article \cite{PresillaSjostrand1996} model a semiconductor heterostructures using
finite well potentials.

There are numerous other interesting mathematical questions that one can ask about higher
dimensional particle-in-well systems, one example is the question of the critical energy of such
systems in dimension $d\ge 3$, see for instance~\cite[Example~5]{Fefferman1983}, or whether one can
show an analogue of a quantitative Faber--Krahn inequality for \(\lambda_1^h(\Omega)\) in the spirit
of the one proved in \cite{BrascoDePhilippisVelichkov2015} for \(\lambda_1^{\rm
D}(\Omega)\).\footnote{Indeed, a qualitative Faber--Krahn inequality for \(\lambda_1^h(\Omega)\) can
readily be proved by the usual rearrangement arguments.}

\subsection{Outline}
We begin in \S{\ref{SSpectralTheory}} by proving some basic spectral properties of \(P_h\). In \S{\ref{SResolvingBoundary}}, we recall the concept of blow-ups, used here to resolve the behavior of quasimodes and eigenfunctions in $h$-neighborhoods of $\pa\Omega$. We then use the resolved space $\tilde M$ introduced in Theorem~\ref{ThmIEigenfn} to construct
quasimodes in \S{\ref{SConstructionQuasimodes}}. By solving a nonlinear problem we can then correct
quasimodes and approximate eigenvalues to true eigenfunctions and eigenvalues in \S{\ref{STCorrectingQuasimodes}}. In
\S{\ref{SRemarksMultiplicity}}, we collect some remarks on the issues arising in attempts to generalize Theorem~\ref{thm:smooth_evs} to the case that the starting point $\lambda_n^{\rm D}$ has multiplicity $\geq 2$. The remaining Appendices~\ref{SAppIntervals}--\ref{SAppBalls} concern explicit computations for intervals, disks, and balls.

\subsection{Notation}
\begin{itemize}
\item We write $A\lesssim B$ if there exists a constant \(C\), independent of \(A,B\), such that $A\leq C B$. If $A,B$ depend on \(h\), we require in addition that \(C\) is \emph{independent of \(h\)}.
\item We write $A\sim B$ if both $A\lesssim B$ and $B\lesssim A$ hold.
\item For two normed vector spaces $V,W$, we write $\cL(V,W)$ for the space of bounded linear operators from $V$ to $W$. We moreover write $\cL(V):=\cL(V,V)$.
\end{itemize}

\section{Spectral theory}
\label{SSpectralTheory}

Throughout this section, it suffices to assume that $\Omega$ is open and bounded; no smoothness is needed unless explicitly specified. We recall that $P_h = -\Delta + h^{-2}\chfun_{\Omega^\complement}$ on $\R^d$ for $h>0$.

\begin{lemma}[Invertibility]
\label{LemmaInvH1}
  \(P_h\colon H^1(\R^d)\to H^{-1}(\R^d)\) is invertible, and so is \(P_h\colon H^2(\R^d)\to L^2(\R^d)\).
\end{lemma}
\begin{proof}
  Observe first that
  \begin{equation}
  \label{EqSpecPoin}
    \norm{u}_{H^1(\R^d)}^2 \lesssim \norm{\nabla u}_{L^2}^2 + \norm{\chfun_{\Omega^\complement}u}_{L^2}^2,\quad u\in H^1(\R^d).
  \end{equation}
  Indeed, if \(\chi\in\CIc(\R^d)\) equals \(1\) near \(\bar\Omega\), then the Poincar\'e inequality on a ball containing \(\supp\chi\) implies
  \[
    \norm{\chi u}_{L^2} \lesssim \norm{\nabla(\chi u)}_{L^2} \leq \norm{\nabla u}_{L^2} + \norm{[\nabla,\chi]u}_{L^2} \lesssim \norm{\nabla u}_{L^2} + \norm{\chfun_{\Omega^\complement}u}_{L^2}
  \]
  since \(\supp[\nabla,\chi]\subset\Omega^\complement\). Since \(\norm{(1-\chi)u}_{L^2}\lesssim\norm{\chfun_{\Omega^\complement}u}_{L^2}\),~\eqref{EqSpecPoin} follows. Therefore,
  \[
      \norm{u}_{H^1(\R^d)}^2 \lesssim \iprod{\nabla u,\nabla u}_{L^2} +
      \iprod{u,\chfun_{\Omega^\complement}u}_{L^2} = \iprod{P_h u,u}_{L^2} \leq \norm{P_h u}_{H^{-1}(\R^d)}\norm{u}_{H^1(\R^d)},
  \]
  so \(\norm{u}_{H^1(\R^d)}\lesssim\norm{P_h u}_{H^{-1}(\R^d)}\). This implies the invertibility of \(P_h\colon H^1(\R^d)\to H^{-1}(\R^d)\).

  Finally, given \(f\in L^2(\R^d)\), let \(u=P_h^{-1}f\in H^1(\R^d)\); then \(-\Delta u=f-h^{-2}\chfun_{\Omega^\complement}u\in L^2(\R^d)\) gives \(u\in H^2(\R^d)\) by elliptic regularity.
\end{proof}

\begin{lemma}[Self-adjointness and pure point part of spectrum]
\label{LemmaSpec}
  The operator \(P_h\) is self-adjoint on \(L^2(\R^d)\) with domain \(H^2(\R^d)\). Its spectrum is
  contained in \([0,\infty)\). Moreover, the spectrum in \([0,h^{-2})\) is discrete.
\end{lemma}
\begin{proof}
  We need to show that \(P_h\pm i\colon H^2(\R^d)\to L^2(\R^d)\) is invertible, or (by Lemma~\ref{LemmaInvH1}) equivalently that
  \begin{equation}
  \label{EqSpecInv}
    (P_h\pm i)\circ P_h^{-1}=I\pm i P_h^{-1}\colon L^2(\R^d)\to L^2(\R^d)\ \text{is invertible.}
  \end{equation}
  Now, \(P_h^{-1}\) is symmetric on \(L^2(\R^d)\): for \(f,g\in L^2(\R^d)\), set \(u=P_h^{-1}f\), \(v=P_h^{-1}g\in H^2(\R^d)\), then
  \[
      \iprod{P_h^{-1}f,g}_{L^2} = \iprod{u,P_h v}_{L^2} = \iprod{P_h u,v}_{L^2} =
      \iprod{f,P_h^{-1}g}_{L^2}.
  \]
  Therefore, \((I\pm i P_h^{-1})f=0\) implies \(0=\iprod{(I\pm i P_h^{-1})f,f}_{L^2}=\norm{f}_{L^2}^2\pm i\iprod{ P_h^{-1}f,f}_{L^2}\). The first term is real, the second term is imaginary, so both vanish, and hence \(f=0\). Since \(I\pm i P_h^{-1}\) has closed range and its adjoint is \(I\mp i P_h^{-1}\),~\eqref{EqSpecInv} follows.

  Next, let \(\chi\in\CIc(\R^d)\) be equal to \(1\) near \(\bar\Omega\). Consider \(\lambda\in\C\) with \(\Re\lambda<h^{-2}\). Then
  \[
    P_h - \lambda(1-\chi) \colon H^2(\R^d)\to L^2(\R^d)\ \text{is invertible,}
  \]
  as follows by similar arguments as in the proof of Lemma~\ref{LemmaInvH1}. Therefore,
  \[
    (P_h-\lambda)\circ(P_h-\lambda(1-\chi))^{-1} = I - R_1,\quad
    (P_h-\lambda(1-\chi))^{-1}\circ(P_h-\lambda) = I - R_2,
  \]
  where
  \[
    R_1 = \lambda\chi(P_h-\lambda(1-\chi))^{-1},\quad
    R_2 = (P_h-\lambda(1-\chi))^{-1}\lambda\chi.
  \]
  By Rellich, \(R_1\) is compact on \(L^2(\R^d)\) and \(R_2\) is compact on \(H^2(\R^d)\). Therefore,
  \(P_h-\lambda\colon H^2(\R^d)\to L^2(\R^d)\) is Fredholm. Since it is invertible for \(\lambda=0\),
  its index is \(0\). The analytic Fredholm theorem finishes the proof.
\end{proof}

We shall pinpoint bounded eigenvalues \(\lambda_h\) of \(P_h\) by constructing quasimodes. However, our
quasimodes \(u_h\) will \emph{not} lie in the domain \(H^2(\R^d)\) of \(P_h\); they will, however, satisfy
\(\norm{u_h}_{L^2(\R^d)}=1+\cO(h)\) and \(\norm{(P_h-\lambda_h)u_h}_{H^{-1}(\R^d)}\leq C_N h^N\). We
proceed to show that this suffices to guarantee an actual eigenvalue nearby:

\begin{lemma}[Quasimodes and spectrum]
\label{LemmaQM}
  Let \(B>0\). There exist $h_0>0$ and \(A>0\) such that the following holds for all \(h\in(0,h_0]\) and
  \(\lambda\in\R\) with \(|\lambda|\leq B\). If \(u_0\in H^1(\R^d)\), \(\norm{u_0}_{L^2}=1\), and
  \(\norm{(P_h-\lambda)u_0}_{H^{-1}(\R^d)}\leq c\in[0,(2 A)^{-1})\), then there exists an eigenvalue
  \(\lambda+\mu\) of \(P_h\) with \(|\mu|\leq 2 c A\).
\end{lemma}

We shall apply this with \(c=C_N h^N\) (for all \(N\), in fact), yielding an eigenvalue
\(\cO(h^N)\)-close to \(\lambda\).

\begin{proof}[Proof of Lemma~\usref{LemmaQM}]
  Write \(b\coloneqq 2 c A\). Suppose the conclusion is false. By self-adjointness, this implies
  \(\norm{(P_h-\lambda)^{-1}}_{L^2\to L^2}\leq b^{-1}\). Now for all \(u\in H^2(\R^d)\),
  \begin{align*}
    \norm{u}_{H^2(\R^d)} &= \norm{(P_h+i)^{-1}(P_h+i)u}_{H^2(\R^d)} \\
      &\leq C_1\norm{(P_h+i)u}_{L^2(\R^d)} \\
      &\leq C_1\norm{P_h u}_{L^2(\R^d)} + C_1\norm{u}_{L^2(\R^d)} \\
      &\leq A\norm{(P_h-\lambda)u}_{L^2(\R^d)} + A\norm{u}_{L^2(\R^d)} \\
      &\leq (A+A b^{-1})\norm{(P_h-\lambda)u}_{L^2(\R^d)}.
  \end{align*}
  (Here \(A\) depends on \(B\).) We conclude that \(\norm{(P_h-\lambda)^{-1}}_{L^2\to H^2}\leq A+A b^{-1}\), and therefore
  \[
    \norm{(P_h-\lambda)^{-1}}_{L^2\to H^1} \leq A+A b^{-1}.
  \]
  We claim that also
  \[
    \norm{(P_h-\lambda)^{-1}}_{H^{-1}\to L^2} \leq A+A b^{-1}.
  \]
  This can be seen as follows: let \(f\in H^{-1}(\R^d)\), then
  \begin{align*}
    \norm{(P_h-\lambda)^{-1}f}_{L^2(\R^d)} &= \sup_{\norm{g}_{L^2(\R^d)}=1} |\iprod{(P_h-\lambda)^{-1}f,g}_{L^2}| \\
      &= \sup_{\norm{g}_{L^2(\R^d)}=1} |\iprod{f,(P_h-\lambda)^{-1}g}_{L^2}| \\
      &\leq \norm{f}_{H^{-1}(\R^d)} \sup_{\norm{g}_{L^2(\R^d)}=1}\norm{(P_h-\lambda)^{-1}g}_{H^1(\R^d)} \\
      &\leq (A+A b^{-1})\norm{f}_{H^{-1}(\R^d)}.
  \end{align*}
  But for \(f\coloneqq (P_h-\lambda)u_0\), this reads \(1\leq(A+A b^{-1})c=\frac12+c A<1\), a contradiction.
\end{proof}

This spectral consideration is useful if one is content with locating the spectrum just using the quasimode construction. We complement this abstract result with a direct proof, given in \S\ref{STCorrectingQuasimodes}, where we construct, directly, the actual eigenfunctions and eigenvalues for the particular operator \(P_h\).

\subsection{Basic properties of the eigenvalues of \texorpdfstring{$P_h$}{the particle-in-well operator}}

Let $n\in\N$. The $n$-th Dirichlet eigenvalue of $\Omega$ can then be described using the min-max formula
\begin{equation}
\label{EqEVVar}
  \lambda_n^{\rm D} = \min_V \max_{\substack{u\in V\\ \norm{u}_{L^2}=1}} \int_\Omega |\nabla u|^2\,\dd x,
\end{equation}
where the minimum is taken over all $n$-dimensional subspaces $V\subset H^1_0(\Omega)$. Suppose $n$ and $h$ are such that $\lambda_n^{\rm D}\leq\frac12 h^{-2}$. Note now that the expression
\begin{equation}
\label{EqEVMinMax}
  \min_V \max_{\substack{u\in V\\ \norm{u}_{L^2}=1}} \int_{\R^d} |\nabla u|^2 + h^{-2}\chfun_{\Omega^\complement}|u|^2\,\dd x,
\end{equation}
where $V\subset H^1(\R^d)$ is $n$-dimensional, is smaller than $\lambda_n^{\rm D}\leq\frac12 h^{-2}$ since $n$-dimensional subspaces of $H^1_0(\Omega)\subset H^1(\R^d)$ are allowed competitors. Therefore,~\eqref{EqEVMinMax} computes the $n$-th eigenvalue of $P_h$.

\begin{lemma}[Monotonicity]
\label{LemmaEVMono}
  Let $n\in\N$. Then there exists $h_n>0$ such that the operator $P_h$ has at least $n$ eigenvalues $\leq\frac12 h^{-2}$ for all $h\in(0,h_n]$. Moreover, for $0<h_-<h_+\leq h_n$, we have
  \begin{equation}
  \label{EqEVMono}
    \lambda_n^{h_+} \leq \lambda_n^{h_-} \leq \lambda_n^{\rm D}.
  \end{equation}
\end{lemma}
\begin{proof}
  Only the monotonicity in $h$ needs to be proved still. Let thus $V$ be an $n$-dimensional subspace of $H^1(\R^d)$ which minimizes $\max_{u\in V,\ \norm{u}_{L^2}=1} \int_{\R^d}|\nabla u|^2+h_+^{-2}\chfun_{\Omega^\complement}|u|^2\,\dd x$. Since the integral increases when $h_+$ is replaced by $h_-$,~\eqref{EqEVMono} follows.
\end{proof}

\begin{prop}[Convergence]
\label{PropEVConv}
  Assume that $\Omega$ is $\cC^1$. Let $n\in\N$. Then
  \[
    \lim_{h\to 0} \lambda_n^h = \lambda_n^{\rm D}.
  \]
\end{prop}
\begin{proof}
  Lemma~\ref{LemmaEVMono} implies that $\lim_{h\to 0}\lambda_n^h\leq\lambda_n^{\rm D}$. Let $\eps>0$; we need to prove that
  \begin{equation}
  \label{EqEVConv}
    \exists\,h>0 \colon \lambda_n^{\rm D}(\Omega)\leq\lambda_n^h+\eps,
  \end{equation}
  where, for clarity, we write $\lambda_n^{\rm D}(\Omega)$ for the $n$-th Dirichlet eigenvalue of $\Omega$.

  For small $\eta>0$, denote by $\Omega_\eta\subset\R^d$ the domain containing $\Omega$ which contains all points at distance $<\eta$ from $\Omega$. If $\eta$ is sufficiently small, then
  \begin{equation}
  \label{EqEVConv2}
    \lambda_n^{\rm D}(\Omega) \leq \lambda_n^{\rm D}(\Omega_\eta) + \frac{\eps}{2}.
  \end{equation}
  Note that the inequality $\lambda_n^{\rm D}(\Omega_\eta)\leq\lambda_n^{\rm D}(\Omega)$ follows from the variational characterization~\eqref{EqEVVar} since $H^1_0(\Omega)\subset H^1_0(\Omega_\eta)$. For the converse, one utilizes a map $\Psi_\eta\colon H^1_0(\Omega_\eta)\to H^1_0(\Omega)$ defined by localizing elements of $H^1_0(\Omega_\eta)$ to $\cC^1$ coordinate charts near points of $\pa\Omega$ which straighten out $\pa\Omega$ and translating by an amount $\eta$ in the inward normal direction, while in coordinate charts in the interior of $\Omega$ one does nothing. (If $\Omega$ is star-shaped around $0$, one can define a map $H^1_0(\Omega_\eta)\to H^1_0(\Omega)$ via pullback by scaling.) As $\eta\to 0$, we have $\int_\Omega|\nabla(\Psi_\eta u)|^2\,\dd x\to\int_{\Omega_\eta}|\nabla u|^2\,\dd x$ for $u\in H^1_0(\Omega_\eta)$. Applying this to $u$ which form a basis of a minimizing subspace $V\subset H^1_0(\Omega_\eta)$ for $\lambda_n^{\rm D}(\Omega_\eta)$ gives~\eqref{EqEVConv2}.

  Next, let $\chi_\eta\in\cC^1(\R^d;[0,1])$ be a function which equals $1$ on $\Omega$ and $0$ outside of $\Omega_\eta$. Write $C:=\norm{\nabla\chi_\eta}_{L^\infty}$. Let $V\subset H^1(\R^d)$ be an $n$-dimensional subspace achieving the minimum in~\eqref{EqEVMinMax}. (That is, $V$ is spanned by the first $n$ eigenfunctions of $P_h$.) Then for $u\in V$ with $\norm{u}_{L^2}=1$, we have
  \[
    \int_{\Omega^\complement} |u|^2\,\dd x\leq h^2\lambda_n^h.
  \]
  Consider then $v:=\chi_\eta u\in H^1_0(\Omega_\eta)$. Clearly $\norm{v}_{L^2}\leq\norm{u}_{L^2}\leq 1$, but also
  \[
    \int_{\R^d} |v|^2\,\dd x = 1 - \int_{\R^d} (1-\chi_\eta^2)|u|^2\,\dd x \geq 1 - \int_{\Omega^\complement} |u|^2\,\dd x\geq 1-h^2\lambda_n^h,
  \]
  so $v$ is almost $L^2$-normalized. Furthermore, we have $\nabla v=\chi_\eta\nabla u+u\nabla\chi_\eta$, so $|\nabla v|^2\leq(1+\delta)|\nabla u|^2+(1+\delta^{-1})|u|^2|\nabla\chi_\eta|^2$ for every $\delta>0$; in view of $\supp\nabla\chi_\eta\subset\Omega^\complement$, this gives
  \[
    \int_{\Omega_\eta} |\nabla v|^2\,\dd x \leq (1+\delta)\int_{\R^d} |\nabla u|^2\,\dd x + (1+\delta^{-1})C^2\int_{\Omega^\complement} |u|^2\,\dd x \leq (1+\delta)\lambda_n^h + (1+\delta^{-1})C^2 h^2\lambda_n^h.
  \]
  For the $L^2$-normalized function $v_0:=v/\|v\|_{L^2}$, we thus have
  \[
    \int_{\Omega_\eta} |\nabla v_0|^2\,\dd x\leq\lambda_n^h \frac{(1+\delta)+(1+\delta^{-1})C^2 h^2}{1-h^2\lambda_n^h}.
  \]
  Fixing first $\delta>0$ sufficiently small and then $h>0$ sufficiently small, this is bounded by $\lambda_n^h+\frac{\eps}{2}$. But the space $\{\chi_\eta u\colon u\in V\}\subset H^1_0(\Omega_\eta)$ is an $n$-dimensional competitor for~\eqref{EqEVVar} on $\Omega_\eta$; therefore,
  \[
    \lambda_n^{\rm D}(\Omega_\eta) \leq \lambda_n^h+\frac{\eps}{2}.
  \]
  Combining this with~\eqref{EqEVConv2} proves~\eqref{EqEVConv}.
\end{proof}

\section{Resolving the boundary}
\label{SResolvingBoundary}

\subsection{Normal coordinates}\label{ssec:fermi}
Let \(\delta > 0\) and denote by \(N_\delta \pa\Omega\) a \(\delta\)-neighborhood of \(\pa\Omega\), that is,
\[
    N_\delta\pa\Omega = \set{x\in\R^d\mid \operatorname{dist}(x, \pa\Omega) < \delta}.
\]
Let \(\nu\colon\pa\Omega\to \R^d\) denote the outward-pointing unit normal at \(\pa\Omega\).
We define the map
\[
    \begin{alignedat}{3}
        \Phi_\delta\colon (-2\delta, 2\delta)&{}\times{}&&\pa\Omega&&\to \R^d,\\
        (\rho&, &&y) &&\mapsto y + \rho\nu(y).
    \end{alignedat}
\]
Because we assume that \(\Omega\) has smooth boundary, \(\Phi_\delta\) is a diffeomorphism when $\delta>0$ is small enough (see, e.g., \cite{Gray2003}); we fix such a small
\[
  \delta>0
\]
for the remainder of the article. This induces \emph{normal coordinates} $\rho,y$ in a $\delta$-neighborhood near \(\pa\Omega\). In $N_\delta\pa\Omega$, \(\rho\) is the signed distance from \(\pa\Omega\), negative in \(\Omega\), positive in
\(\Omega^\complement\). We modify $\rho$ outside of $N_\delta\pa\Omega$ such that $\rho$ is smoothed out outside of $N_\delta\pa\Omega$ (where $\rho$ is the signed distance) to a globally smooth function on \(\R^d\) which is \(\leq-\delta\), resp.\ \(\geq\delta\), outside a \(\delta\)-neighborhood of \(\partial\Omega\) in \(\Omega\), resp.\
\(\Omega^\complement\).

On \((-\delta, \delta)_\rho \times \pa\Omega\) and in local coordinates \(y\) on \(\pa\Omega\), the dual of the Euclidean metric has the warped product form $\pa_\rho\otimes\pa_\rho+g^{j k}(\rho,y)\pa_{y^j}\otimes\pa_{y^k}$ by the Gauss lemma. Therefore, using the Einstein summation convention,
\[
\laplace = \frac{1}{g(\rho,y)}\partial_\rho g(\rho,y)\partial_\rho + \frac{1}{g(\rho,y)}\partial_{y^j}\bigl(g(\rho,y)g^{j k}(\rho,y)\partial_{y^k}\bigr),\quad
\laplace_{\partial\Omega}=\frac{1}{g(0,y)}\partial_{y^j}\bigl(g(0,y)g^{j k}(0,y)\partial_{y^k}\bigr);
\]
here \((g^{j k})\) is a smooth positive definite matrix and \(g(\rho,y)=(\det(g^{j k}))^{-1/2}\). We rewrite this as
  \begin{equation}
  \label{EqTApproxDelta}
    \laplace = \partial_\rho^2 + a(\rho,y)\partial_\rho + g^{j k}(\rho,y)\partial_{y^j}\partial_{y^k} + b^l(\rho,y)\partial_{y^l},\quad
    \laplace_{\partial\Omega} = g^{j k}(0,y)\partial_{y^j}\partial_{y^k} + b^l(0,y)\partial_{y^l}.
  \end{equation}
for smooth \(a,b^l\).

\subsection{Blow-up of the boundary}
\label{SsBlowup}

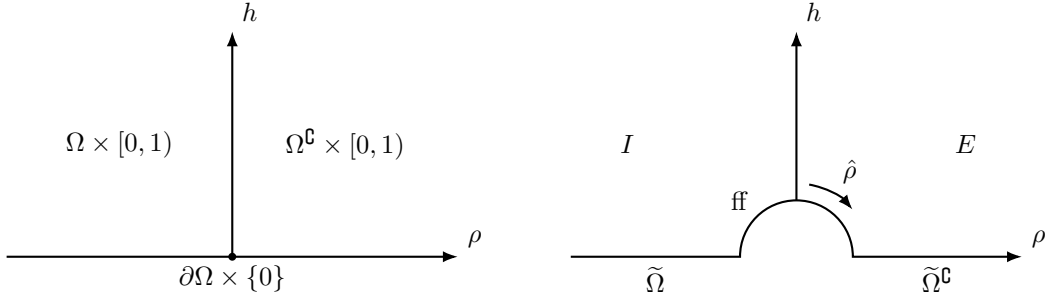
\begin{figure}
    \centering
    \begin{tikzpicture}[scale=1.5]
        \draw[thick, -{Latex[length=2mm]}] (-2.0, 0.0) -- (2.0, 0.0) node[anchor=south west] {\(\rho\)};
        \draw[thick, -{Latex[length=2mm]}] (0.0, 0.0) -- (0.0, 2.0) node[anchor=south west] {\(h\)};
        \fill [black, inner sep=3pt] (0, 0) circle (1pt) node[anchor=north] {\(\pa\Omega\times\set{0}\)};
        \node at (-1.0, 1.0) {\(\Omega\times[0,1)\)};
        \node at (1.0, 1.0) {\(\Omega^\complement\times[0,1)\)};

        \begin{scope}[shift={(5.0, 0.0)}]
            \draw[thick, -{Latex[length=2mm]}] (-2.0, 0.0) -- node[below, midway] {\(\wt\Omega\)} (-0.5, 0.0) arc
                [start angle=180, end angle=0, radius=0.5] --  node[below, midway]
                {\(\wt\Omega^\complement\)} (2.0, 0.0) node[anchor=south west] {\(\rho\)};
            \draw[thick, -{Latex[length=2mm]}] (0.0, 0.5) -- (0.0, 2.0) node[anchor=south west] {\(h\)};
            \node at (-0.5, 0.5) {\(\ff\)};
            \node at (-1.5, 1.0) {\(I\)};
            \node at (1.5, 1.0) {\(E\)};
            \draw[thick, -{Latex[length=2mm]}] ({0.65 * cos(80)}, {0.65 * sin(80)}) arc (80:40:0.65) node[anchor=south
                west, midway] {\(\hat{\rho}\)};
        \end{scope}
    \end{tikzpicture}
    \caption{We blow up the boundary \(\pa\Omega\times\set{0}\) in $\R^d\times[0,1)_h$ to pass from the  total space on the
        left to the blown-up space \(\tilde{M}\) on the right.}\label{fig:blowup}
\end{figure}

We will capture the expected analytically singular behavior of quasimodes and eigenfunctions by working with analytically simpler functions on a geometrically more complex space; we accomplish this using the method of real blow-up. We shall here introduce this method for our context. (For a detailed discussion, see \cite{MelroseDiffOnMwc}.)

We begin with the total space \(M\coloneqq \R^d_x\times [0,1)_h\). This is a manifold with boundary
\(\R^d\times \set{0}\). As we expect singular behavior on \(\pa\Omega\times\set{0}\), we consider
the blow-up
\[
    \tilde{M} \coloneqq [\R^d_x\times [0,1)_h;\pa\Omega\times \set{0}]
\]
together with the blow-down map
\[
    \beta \colon \tilde{M}\to M.
\]
Roughly speaking, $\tilde M$ is obtained from $M$ by replacing $\pa\Omega\times\{0\}$ by the collection $\ff$ of (non-strictly) inward pointing unit vectors, and $\beta$ is the identity on $(\R^d\times[0,1))\setminus(\pa\Omega\times\{0\})$ and the base projection on $\ff$. Local coordinates on $\tilde{M}$ near the interior of this \emph{front face} $\ff$ are $h\geq 0$, $\hat\rho=\frac{\rho}{h}\in\R$, and $y\in\pa\Omega$. In this manner, we can, on $\tilde{M}$, resolve the behavior of functions in $\cO(h)$-neighborhoods of $\pa\Omega$.

In more detail, let us describe the manifold with corners $\tilde M$ (see Figure~\ref{fig:blowup}) explicitly by providing a few overlapping coordinate charts that cover it:
\begin{enumerate}
\item $\R^d\times(0,1)$, with coordinates $(x,h)$;
\item $\Omega\times[0,1)$, again with coordinates $(x,h)$. This covers the interior of the set labeled $I$ in Figure~\ref{fig:blowup};
\item $\Omega^\complement\times[0,1)$, again with coordinates $(x,h)$. This covers the interior of the set labeled $E$ in Figure~\ref{fig:blowup};
\item $(-2,2)\times\pa\Omega\times[0,\frac{\delta}{2})$, with coordinates $(\hat\rho,y,h)$ corresponding to a point $(x,h)=(\Phi_\delta(h\hat\rho,y),h)$ in the other coordinate systems. This covers a neighborhood of the central part of $\ff$;
\item $[0,1)\times[0,\delta)\times\pa\Omega$ (twice), with coordinates $(\rho_0,\rho_\ff,y)$ corresponding to a point $(x,h)=(\Phi_\delta(\pm\rho_\ff,y),\rho_0\rho_\ff)$. Using the notation in Figure~\ref{fig:blowup}, this covers a neighborhood of the corner between $\ff$ and $\tilde\Omega$ (`$-$' sign), resp.\ $\tilde\Omega^\complement$ (`$+$' sign).
\end{enumerate}

In each of the coordinate systems, the blow-down map $\beta$ simply outputs the $(x,h)$-coordinates of the input point. The front face of $\tilde M$ is now rigorously defined as \(\ff \coloneqq \beta^{-1}(\pa\Omega\times \set{0})\). We may identify \(\ff\) with the space \(\ol{\R_{\hat{\rho}}}\times \pa\Omega\)
with \(\hat{\rho} = \rho / h\), where \(\rho\) is as in \S\ref{ssec:fermi}, and
\[
  \ol\R := [-\infty,\infty]
\]
is the radial compactification of $\R$: this can be defined as the closed interval $[-1,1]$, with a point $z\in(-1,1)$ identified with the real number $\tan(\frac{\pi z}{2})$. (Thus, on $\ol{\R_{\hat\rho}}\setminus\{0\}$, the function $1/\hat\rho$ is smooth and vanishes simply at $\pm\infty\in\ol\R$.)

We denote the map which restricts continuous functions on $\tilde{M}$ to $\ff$ by
\[
    R_{\ff}\colon \cC^0(\tilde{M})\to \cC^0(\ff),\quad u \mapsto u|_{\ff}.
\]
Fix a smooth cutoff function \(\chi\in \CIc(\R)\) with \(\supp \chi \subset (-\delta, \delta)\) and
\(\chi\equiv 1\) on \([-\frac{\delta}{2}, \frac{\delta}{2}]\). We then define a right-inverse
\(F_\ff^\chi\colon \cC^0(\ff)\to \cC^0(\tilde{M})\) of \(R_{\ff}\) by
\begin{equation}
\label{EqExtMap}
    F_\ff^\chi f (y + \rho \nu(y), h) = f\left(\frac{\rho}{h}, y\right)\chi(\rho),\quad f\in
    \CIc(\ff).
\end{equation}

For a set \(S\subset \R^d\times [0,1)\) we define its \emph{lift} by
\[
    \beta^* S\coloneqq \begin{cases}
        \beta^{-1}(S) & \text{if }(\pa\Omega\times \set{0})\cap S = \emptyset,\\
        {\rm cl}\bigl(\beta^{-1}(S\setminus (\pa\Omega\times \set{0}))\bigr) & \text{else,}
    \end{cases}
\]
where ${\rm cl}$ denotes the closure of a set in $\tilde M$. We define the \emph{lift \(\beta^*f\) of a function} \(f\in \cC^0(\R^d\times [0,
1)\setminus(\pa\Omega\times\set{0}))\) as follows: For \(m\in \tilde{M}\), set
\[
    \beta^*f(m) = \begin{cases}
        f(x, h) & \text{if }(x, h)\notin \pa\Omega\times\set{0}\text{ and } \beta(m) = (x,h),\\
        \lim_{h\to 0}f(y + h\hat{\rho}\nu(y), h) & \text{if }m = (y, \hat{\rho})\in\ff,
    \end{cases}
\]
provided the limit exists in the latter case. Thus, $\beta^*f$ is the continuous extension (if it exists) of $f|_{(\R^d\times[0,1))\setminus(\pa\Omega\times\{0\})}$ to $\tilde M$.

We label the lifted inner, exterior, and boundary parts of \(\Omega\times[0,1)\) as follows:
\[
    I \coloneqq \beta^*(\bar\Omega\times [0,1)), \quad E \coloneqq \beta^*(\Omega^\complement\times [0,1)),\quad
    B \coloneqq \beta^*(\pa\Omega\times [0,1)).
\]
We furthermore write
\[
    \widetilde\Omega \coloneqq \beta^*(\set{0}\times \Omega),\quad 
    \widetilde\Omega^\complement \coloneqq \beta^*(\set{0}\times \Omega^\complement).
\]

We moreover use the notation \(\rho_\Omega\), \(\rho_{\ff}\), \(\rho_{\Omega^\complement}\) for defining functions of \(\widetilde\Omega\), \({\ff}\), \(\widetilde\Omega^\complement\), respectively. (We recall here that a defining function of an embedded boundary hypersurface $H$ of a manifold with corners is a smooth function $\rho\geq 0$ such that $H=\rho^{-1}(0)$ and $\dd\rho\neq 0$ on $H$.) For definiteness,
we moreover demand that for \(h<\delta/8\),
\begin{equation}
\label{EqIDefFn}
  \rho_\Omega=1\ \text{for}\ \rho\geq-h,\qquad
  \rho_{\ff}=1\ \text{for}\ |\rho|\geq\delta,\qquad
  \rho_{\Omega^\complement}=1\ \text{for}\ \rho\leq h.
\end{equation}

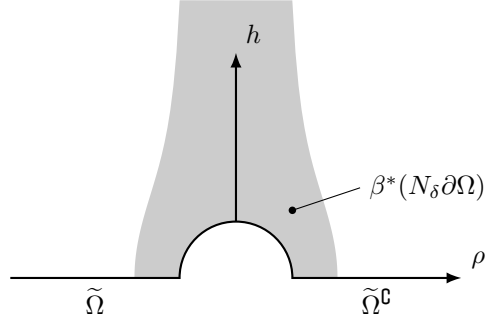
\begin{figure}
    \centering
    \begin{tikzpicture}[scale=1.5]
        \draw[thick, -{Latex[length=2mm]}] (-2.0, 0.0) -- node[below, midway] {\(\wt\Omega\)} (-0.5, 0.0) arc
            [start angle=180, end angle=0, radius=0.5] --  node[below, midway]
            {\(\wt\Omega^\complement\)} (2.0, 0.0) node[anchor=south west] {\(\rho\)};
        \draw[thick, -{Latex[length=2mm]}] (0.0, 0.5) -- (0.0, 2.0) node[anchor=south west] {\(h\)};
        \fill[opacity=0.2]
            plot [domain=0:0.5*pi-0.2,samples=25] ({ 0.5 * cos(\x r) + 0.4}, {0.5 * sin(\x r) + 0.4 * tan(\x r)})
         -- plot [domain=0.5*pi-0.2:0,samples=25] ({-0.5 * cos(\x r) - 0.4}, {0.5 * sin(\x r) + 0.4 * tan(\x r)})
         -- (-0.5, 0) -- (-0.5, 0.0) arc [start angle=180, end angle=0, radius=0.5] -- cycle;
     \draw (0.5, 0.6) node[circle, fill, inner sep=1pt] {} -- (1.1, 0.8) node[right] {\(\beta^*(N_\delta\pa\Omega)\)};
    \end{tikzpicture}
    \caption{The allowed region for functions supported near the boundary is shaded.}\label{fig:suppbndry}
\end{figure}

\subsection{The Schr\"odinger operator on the resolved space}

The operator family \((0,1)\ni h\mapsto P_h = \laplace + h^{-2}\chfun_{\Omega^\complement}\) that we
are interested in can be viewed as a single differential operator \(P\) acting on smooth functions
\(u\in \CI(\R^d\times (0,1))\) as
\[
    (Pu)(\cdot, h) = P_h(u(\cdot, h)).
\]
While the operator family \(P_h\) degenerates as \(h\to 0\), its structure in the limit $h\to 0$ becomes clearer if we consider the lift of $P$ to $\tilde M$. More precisely, we will compute the lift of $h^2 P$ and its restriction to $\ff$ in the coordinates \((\hat{\rho}, y,h)\).
Thus, for points \((y + \rho \nu(y), h)\in \tilde{M}\) with \(y\in\pa\Omega\) and \(\rho\) as well as
\(h\) sufficiently small, we can use \eqref{EqTApproxDelta} and change coordinates
\((\rho,y)\mapsto (\hat{\rho}, y)=(\rho/h,y)\) to find
\begin{equation}
\label{EqSchrResolved}
    \begin{aligned}
        h^2P &= h^2\bigl( -h^{-2}\pa_{\hat\rho}^2 - a(\rho,y)h^{-1}\pa_{\hat\rho} - g^{j
        k}(\rho,y)\partial_{y^j}\partial_{y^k} - b^l(\rho,y)\partial_{y^l} +
        h^{-2}\chfun_{\set{\hat\rho \ge 0}}(\hat\rho, y)\bigr)\\
        &= -\partial_{\hat{\rho}}^2 + H(\hat{\rho}) - ha(\rho,y)\pa_{\hat\rho} - h^2g^{j k}(\rho,y)\partial_{y^j}\partial_{y^k} - h^2b^l(\rho,y)\partial_{y^l},
    \end{aligned}
\end{equation}
where \(H\) is the Heaviside step function defined by \(H=1\) if its argument is \(\ge 0\) and \(H=0\) else.
Formally restricting this to \(h = 0\) we obtain the differential operator
\[
    \hat{P} = -\partial_{\hat{\rho}}^2 + H(\hat{\rho}),
\]
which acts on smooth functions on \(\ff\). We call this the \emph{model operator at \(\ff\)}. It is a family of operators parameterized by the base point $y\in\pa\Omega$, but since it does not actually have any $y$-dependence, we shall occasionally tacitly regard it as an operator on functions of $\hat\rho$ only.

\section{Construction of quasimodes on the resolved space}
\label{SConstructionQuasimodes}

We will construct quasimodes as elements of the following space of functions:

\begin{definition}[Quasimode space]
\label{DefQMSpace}
    We write \(\qms\) for the space of functions \(u\) on \(\tilde{M}\) with the following properties:
    \begin{enumerate}
        \item \(u\) is continuous;
        \item \(u\) is supported on \(I \cup \beta^*(N_\delta\pa\Omega\times [0,1))\);
        \item \(u|_{I^\circ}\in \cC^\infty(I)\) and \(u|_{E^\circ}\in \cC^\infty(E)\), i.e., $u|_{I^\circ}$ has a continuous extension to $I$ which defines a smooth function on $I$, similarly for $u|_{E^\circ}$.
    \end{enumerate}
    We moreover write \(\qms_{\ff} \coloneqq R_{\ff} \qms\) for the space of restrictions of elements of $\cX$ to $\ff$. For an element \(u\in\qms\), we write \(u_h\) for its restriction to the \(h\)-level set.
\end{definition}

Thus, elements of $\qms$ are smooth in the interior and exterior and match across the interface $B$, but the interior and exterior normal derivatives may differ.

Error terms arising in the quasimode construction will feature $\delta$-distributions at $\pa\Omega$. (For example, upon twice differentiating the extension by $0$ of a Dirichlet eigenfunction $u$ in $\rho$, one obtains $\delta(\rho)$ times the normal derivative of $u$.) More precisely, we denote by \(\delta_B\) the \(\delta\)-distribution on \(B\subset \tilde{M}\), i.e.,
\[
    \iprod{\delta_B, \varphi} = \int_0^1\int_{\pa\Omega} \varphi(0, y, h)\odif{y}\odif{h}\qquad
    \forall\varphi\in \CIc(\tilde{M}).
\]
This means that in coordinates \((\hat\rho, y, h)\) near \(\ff\) we have \(\delta_B(\hat\rho, y, h) =
\delta_0(\hat\rho)\). Carefully note that this differs from the \(\delta\)-distribution \(\delta_0(\rho)\) on a level set of \(h\) in view of \(h^{-1}\delta_0(\hat\rho)\). Thus, \(\delta_B\) is \emph{not} the lift to
\(\tilde{M}\) of the distribution \(\delta_{\pa\Omega\times[0,1)}\) on \(M\); rather, it is the lift of
\(h\delta_{\pa\Omega\times [0,1)}\).

\begin{definition}[Error space]
\label{DefESpace}
    We write \(\es\) for the space of distributions \(v\in \cD'(\tilde{M})\) (the dual space of $\CIc(\tilde{M})$) which can be written as a sum \(v=u_s + f\delta_B\) where \(u_s\) is a function with
    \(u_s|_{I^\circ}\in \cC^\infty(I)\) and \(u_s|_{E^\circ}\in \cC^\infty(E)\), and \(f\in
    \CI(B)\). We write $\cY_E$ for the subspace of $\cY$ consisting only of distributions with support in $E$.
\end{definition}

Thus, elements of $\es$ are smooth in the interior and exterior but need not be continuous across the interface, and in fact may feature a $\delta$-distributional singularity there.

 There is a restriction map, which we shall again call \(R_{\ff}\), that restricts elements of
 \(\es\) to \(\ff\). The image \(\es_{\ff} = R_{\ff}\es\) is then the set of
 distributions \(u + g\delta_0\) where \(u\) is a function with \(u|_{\pm \hat{\rho} > 0} \in
 \cC^\infty(\ff\setminus\set{\pm \hat{\rho}<0})\) (i.e., smoothness from the left and the right), \(g\in \CI(\pa\Omega)\), and \(\delta_0\)
 is the delta distribution on \(\ff\) at $\{\hat\rho=0\}\subset\ff=\ol{\R_{\hat{\rho}}}\times\pa\Omega$. We similarly write
 \[
   \es_{E,\ff} = R_\ff\es_E
 \]
 for the space of distributions which are sums of $\delta$ distributions at $\hat\rho=0$ and smooth functions on $E$, and which vanish on $I^\circ$.

In order to deal with weights, we introduce notation such as
\[
  \rho_\ff^k\cX := \{ \rho_\ff^k u\colon u\in\cX \},\quad
  h^j\cY := \{ h^j v\colon v\in\cY\},
\]
similarly for weights which are functions of $\rho_\Omega$ and $\rho_{\Omega^\complement}$; of particular importance for us will be weights of the form $\rho_\Omega^k\rho_{\Omega^\complement}^N$ for $k,N\in\N_0$. Note that multiplication by such weights maps $\cX\to\cX$ and $\cY\to\cY$.

We choose to define $\qms$ and $\es$ in this way so that their elements do not feature any $h$-weights at the boundary hypersurfaces of $\tilde M$; this makes the bookkeeping later on more transparent. The small price to pay is that $P$ does not map $\cX\to\cY$; instead:

\begin{lemma}[Mapping properties of $P$]
\label{LemmaPMap}
  We have $P\colon\qms\to\rho_\ff^{-2}\rho_{\Omega^\complement}^{-2}\es$. Moreover,
  \[
    R_{\ff}\bigl(h^2 P u\bigr) = \hat P(R_\ff u),\quad u\in\qms.
  \]
  More generally, $P\colon\rho_\Omega^a\rho_\ff^b\rho_{\Omega^\complement}^c\cX\to\rho_\Omega^a\rho_\ff^{b-2}\rho_{\Omega^\complement}^{c-2}\cY$, and $R_\ff(h^{2-b}P u)=\hat P(R_\ff(h^{-b}u))$.
\end{lemma}
\begin{proof}
  We directly consider the weighted version. By dividing by $h^b$ and replacing $(a,c)$ by $(a-b,c-b)$, we can reduce to the case $b=0$.

  For $\rho<-\delta$, resp.\ $\rho>\delta$ (where we can take $\rho_\ff=1$ and $\rho_\Omega=h$, $\rho_{\Omega^\complement}=1$, resp.\ $\rho_\Omega=1$, $\rho_{\Omega^\complement}=h$), the operator $P$, resp.\ $h^2 P$ has smooth coefficients as a differential operator on $\R^d\times[0,1)$. It remains to study $P$ near $\ff$. Near $\ff^\circ$ (where we can take $\rho_\Omega=\rho_{\Omega^\complement}=1$), we pass to the coordinates $\hat\rho=\rho/h\in\R$, $y\in\pa\Omega$, $h\in[0,1)$ in which $h^2 P$ was computed in~\eqref{EqSchrResolved}. The action of $\pa_{\hat\rho}^2$ on continuous functions that are smooth down to $\hat\rho=0$ from the left and right produces a multiple of $\delta(\hat\rho)$. The remaining terms of $h^2 P$, when acting on $u$, produce continuous functions that are smooth from the left and right.

  It remains to consider a neighborhood of the boundary of $\ff$. We only show the computations near the corner $\ff\cap\wt\Omega^\complement$ where we use the coordinates $\rho\geq 0$, $y\in\pa\Omega$, $\hat h=h/\rho\geq 0$, and we can take $\rho_\ff=\rho$, $\rho_{\Omega^\complement}=\hat h$. Then $\rho_\ff^2\rho_{\Omega^\complement}^2=h^2$, and it remains to note that $h\pa_\rho$ (in $h,\rho$ coordinates) equals $\hat h(\rho\pa_\rho-\hat h\pa_{\hat h})$ (in $\hat h,\rho$ coordinates), and thus $h^2\pa_\rho^2$ maps the space $\rho_{\Omega^\complement}^c\CI([0,1)_\rho\times\pa\Omega\times[0,1)_{\hat h})$ into itself; similarly for all other terms of $h^2 P$.
\end{proof}

Recalling the extension map~\eqref{EqExtMap}, we also note:

\begin{lemma}[Behavior under \(F_\ff^\chi\)]\label{LemmaExtSpaces}
    Let \(\cS\in \set{\qms, \es,\es_E}\). If \(f\in \cS_\ff\), then \(F_\ff^\chi f\in \cS\). The same statements hold for weighted spaces $\rho_\Omega^a\rho_{\Omega^\complement}^b\cS$.
\end{lemma}
\begin{proof}
  This follows directly from the definition of $F_\ff^\chi$.
\end{proof}

\subsection{Solving the model problem on \texorpdfstring{$\ff$}{the front face}}
\label{SsModelff}

Recall that the model operator on \(\ff\) is given by $\hat{P} \coloneqq -\partial_{\hat{\rho}}^2 + H(\hat{\rho})$. We first note that the function
\[
  G(\hat\rho) := \begin{cases} 1, & \hat\rho\leq 0, \\ e^{-\hat\rho}, & \hat\rho>0, \end{cases}
\]
satisfies $\hat P G=\delta(\hat\rho)$. (Carefully note, however, that $\hat P$ does not have constant coefficients, so convolution with $G$ does not yield an inverse of $\hat P$.) The following result solves the model problem for the kinds of right-hand sides that will arise in the quasimode construction:

\begin{lemma}[Solution of the model problem on $\ff$]
\label{LemmaModelSol}
  Let $\hat f\in\rho_{\Omega^\complement}^\infty\cY_{E,\ff}$. Then there exists a solution $\hat u\in\rho_{\Omega^\complement}^\infty\cX_\ff$ of $\hat P\hat u=\hat f$, and $\hat u(\hat\rho,y)$ is constant for $\hat\rho<0$ for every fixed $y\in\pa\Omega$. If $\hat f(\hat\rho,y)=\hat g(y)\delta(\hat\rho)$, then $\hat u(\hat\rho,y)=\hat g(y)G(\hat\rho)$.
\end{lemma}

Since the kernel of $\hat P$ is spanned by $\chfun_{(-\infty,0)}+\cosh(\hat\rho)\chfun_{(0,\infty)}$ and $\hat\rho\chfun_{(-\infty,0)}+\sinh(\hat\rho)\chfun_{(0,\infty)}$ (times smooth functions of $y$), it is easy to see that $\hat u$ must be the \emph{unique} solution of $\hat P\hat u=\hat f$ that remains bounded as $|\hat\rho|\to\infty$.

\begin{proof}[Proof of Lemma~\usref{LemmaModelSol}]
  We drop the $y$-dependence from the notation. The strategy is to find a particular solution of $\hat P\hat u_+=\hat f$ for $\hat\rho>0$ and then add suitable kernel elements on both sides of $\hat\rho=0$ to produce the desired $\delta$-distributional contribution at $\hat\rho=0$. Write thus $\hat f(\hat\rho)=\hat v(\hat\rho)+c\delta_0(\hat\rho)$ where $\hat v|_{\hat\rho>0}\in\CI([0,\infty])$ and $\hat v|_{\hat\rho<0}=0$.

  Extend $\hat v|_{\hat\rho>0}$ to a Schwartz function $\hat v'\in\sS(\R_{\hat\rho})$. Then we can solve $(-\pa_{\hat\rho}^2+1)\hat u_+'=\hat v'$ for $\hat u_+'\in\sS(\R_{\hat\rho})$ using the (inverse) Fourier transform $\cF$ via $\hat u_+'=\cF^{-1}(\xi^2+1)^{-1}\cF\hat v'$. Therefore, $\hat u_+:=\hat u_+'|_{[0,\infty)}$ is smooth and rapidly decaying as $\hat\rho\to\infty$, and hence defines an element $\hat u_+\in(1+\hat\rho)^{-\infty}\CI([0,\infty]_{\hat\rho})=\rho_{\Omega^\complement}^\infty\CI([0,\infty]_{\hat\rho})$.

  Let now $\hat u':=\chfun_{(0,\infty)}\hat u_+\in\rho_\Omega\rho_{\Omega^\complement}^\infty\cX_\ff$. Since this is smooth from the left and from the right at $0$, we have
  \[
    \hat P\hat u' = \hat v + a\delta(\hat\rho) + b\delta'(\hat\rho)
  \]
  for some constants $a,b\in\R$. Now $\hat P H(-\hat\rho)=\delta'(\hat\rho)$, so
  \[
    \hat P(\hat u'-b H(-\hat\rho)) = \hat v + a\delta(\hat\rho) = \hat f + (a-c)\delta(\hat\rho).
  \]
  Finally, since $\hat P G=\delta(\hat\rho)$, we find
  \[
    \hat P\bigl( \hat u'-b H(-\hat\rho) - (a-c)G \bigr) = \hat f.
  \]
  The sought-after solution is thus $\hat u:=\hat u'-b H(-\hat\rho) - (a-c)G$. (Note here that $H(-\hat\rho),G\in\rho_{\Omega^\complement}^\infty\cX_\ff$ indeed, and they are constant for $\hat\rho<0$.)
\end{proof}

We next prove a result that controls the extension of $\hat u$ to $\tilde M$. Let $\cE$ denote the extension by $0$, mapping bounded functions on $\Omega$ to bounded functions on $\R^d$.

\begin{lemma}[Inversion at $\ff$]
\label{LemmaExt}
  Let $\hat u\in\rho_{\Omega^\complement}^\infty\cX_\ff$ be such that $\hat u(\hat\rho,y)$ is constant for $\hat\rho<0$ for each $y\in\pa\Omega$. Then
  \[
    P(F_\ff^\chi\hat u) - h^{-2}F_\ff^\chi(\hat P\hat u) \in \beta^* \cE\bigl(\CI(\bar\Omega\times[0,1))\bigr) + \rho_\ff^{-1}\rho_{\Omega^\complement}^\infty\cY_E.
  \]
\end{lemma}
\begin{proof}
  Since the function $(h,\rho,y)\mapsto\hat u(\rho/h,y)$ is smooth on
  $N_\delta\pa\Omega\times[0,1)_h$, we have $[P,\chi]\hat u\in\beta^*\CI(\R^d\times[0,1))$ (and this vanishes near $\pa\Omega$), which we can write as a sum of a smooth function on $\bar\Omega\times[0,1)$ (vanishing outside $\bar\Omega$) and a smooth function on $\ol{\Omega^\complement}\times[0,1)$ (vanishing outside $\ol{\Omega^\complement}$ and vanishing to infinite order at $h=0$); the first summand lies in $\beta^*\cE\CI(\bar\Omega\times[0,1))$, the second a fortiori in $\rho_\ff^{-1}\rho_{\Omega^\complement}^\infty\es_E$.

  Next, since $P$ annihilates constants on $I$, $\chi P\hat u$ is supported in $E$, and thus $\chi
  P\hat u\in\rho_\ff^{-2}\rho_{\Omega^\complement}^\infty\cY_E$ by Lemma~\ref{LemmaPMap}. By the
  same lemma, the restriction of $h^2\chi P\hat u$ to $\ff$ is given by $\hat P\hat u$. In view of
  $\hat P\hat u\in\rho_{\Omega^\complement}^\infty\cY_{E,\ff}$, we conclude that we not only have
  $P(F_\ff^\chi\hat u)-h^{-2}F_\ff^\chi(\hat P\hat
  u)\in\beta^*\cE\CI(\bar\Omega\times[0,1))+\rho_\ff^{-2}\rho_{\Omega^\complement}^\infty\cY_E$, but
  in fact the power of $\rho_\ff$ of the second summand is improved to $-1$. (We use here that
  $F_\ff^\chi\colon\rho_{\Omega^\complement}^\infty\cY_{E,\ff}\to\rho_{\Omega^\complement}^\infty\cY_E$,
  see Lemma~\ref{LemmaExtSpaces}.)
\end{proof}

\subsection{Iterative construction of quasimodes for simple eigenvalues}

The iteration procedure will proceed in two steps for each order of
\(h\). Suppose we have generated a quasimode of order \(h^k\). First, we eliminate the order \(h^k\)
error on \(\widetilde\Omega\) by a standard perturbation argument based on inverting $-\dirlap{\Omega}-\lambda_0$ (modulo the cokernel) where we write $\lambda_0$ for the simple Dirichlet eigenvalue we are starting with; getting rid of the cokernel yields the order $h^k$ correction to the eigenvalue. Adding to the current quasimode the extension by $0$ of this correction term on $\Omega$, the remaining error after this first step is singular at \(B\). We then use the inverse of the model operator $\hat P$ from the previous section to eliminate this error (which is of order $h^{k-1}$) at $\ff$.

The elimination of error terms at $\widetilde\Omega$ will be based on:

\begin{lemma}[Grushin problem on \(\Omega\)]\label{LemGrushinProblemQM}
    Let \(\lambda_0\) be a simple eigenvalue of \(-\dirlap{\Omega}\) with corresponding \(L^2\)-normalized
    eigenfunction \(u_0\in \CI_0(\bar\Omega)\). Then the operator
    \[
        \begin{pmatrix}
            -\dirlap{\Omega} - \lambda_0 & u_0\\
            \iprod{\cdot, u_0} & 0
        \end{pmatrix}\colon \CI_0(\bar\Omega)\oplus \C \to \CI(\bar\Omega)\oplus \C
    \]
    is bijective. In other words, for all \(f\in \CI(\bar\Omega)\) and \(c\in \C\) the system of equations
    \[
      \left\{
      \begin{alignedat}{1}
            (-\dirlap{\Omega} - \lambda_0)u + \gamma u_0 &= f, \\
            \iprod{u, u_0} &= c
      \end{alignedat}
      \right.
    \]
    has a unique solution \((u, \gamma)\in \CI_0(\bar\Omega)\oplus \C\).
\end{lemma}
\begin{proof}
    Fix \(f\in \CI(\bar\Omega)\) and \(c\in\C\).
    By the usual elliptic theory, there exists an orthogonal decomposition (with respect to the
    \(L^2\)-inner product) of \(\CI(\bar\Omega)\) as
    \[
        \CI(\bar\Omega) = \ker_{\CI_0(\bar\Omega)}\left(-\dirlap{\Omega} - \lambda_0\right) \oplus
        \ran_{\CI_0(\bar\Omega)}\left(-\dirlap{\Omega} - \lambda_0\right).
    \]
    Hence, we can uniquely write
    \[
        f = b u_0 + (-\dirlap{\Omega} - \lambda_0) g
    \]
    for some \(g\in \CI_0(\bar\Omega)\) and \(b\in \C\). Put \(u = g + cu_0\) and \(\gamma = b\).
\end{proof}

\begin{prop}[Existence and regularity of quasimodes]\label{PropQM}
    Let \(\lambda_0\) be a simple eigenvalue of \(-\dirlap{\Omega}\) with corresponding
    \(L^2\)-normalized eigenfunction \(u_0\in \CI_0(\bar\Omega)\). Then there exist sequences
    \(\tilde{u}_k\in \rho_{\Omega^\complement}^\infty\qms\) and \(\tilde\lambda_k\in \CI([0,1))\), $k\in\N_0$, with the following properties:
    \begin{enumerate}
        \item\label{ItQM1} $\tilde\lambda_0(0)=\lambda_0$, $\tilde u_0|_{\wt\Omega}=u_0$;
        \item\label{ItQM2} $(P - \tilde\lambda_k)\tilde{u}_k \in h^{k+1} \beta^*\cE\bigl(\CI(\bar\Omega\times[0,1))\bigr) + \rho_\ff^k\rho_{\Omega^\complement}^\infty\cY_E$ where $\cE$ denotes extension by $0$;
        \item\label{ItQM3} for $k\geq 1$, we have $\tilde\lambda_k-\tilde\lambda_{k-1}\in h^k\CI([0,1)_h)$ and $\tilde u_k-\tilde u_{k-1}\in \rho_\Omega^k\rho_\ff^{k+1}\rho_{\Omega^\complement}^\infty\cX$.
    \end{enumerate}
\end{prop}
\begin{proof}
  We cannot directly work with the spaces $\cX$ and $\cY$ since they are too imprecise for the iterative procedure below.\footnote{For example, the first error in~\eqref{EqErr1st} lies in $\rho_\Omega^\infty\rho_\ff^{-1}\rho_{\Omega^\complement}^\infty\cY$ and can be solved away using an element of $h\cdot\rho_{\Omega^\complement}^\infty\cX$ via Lemma~\ref{LemmaModelSol}. Only using the mapping properties of Lemma~\ref{LemmaPMap}, the remaining error would be an element $f\in\rho_\Omega\rho_{\Omega^\complement}^\infty\cY$. Using this information, the leading order error at $\wt\Omega$, i.e., the restriction of $h^{-1}f$ to $\wt\Omega$, is only known to lie in $\rho_\ff^{-1}\CI(\bar\Omega)$, which is thus singular at $\pa\Omega$; solving this away requires the introduction of logarithmic terms. The upshot is that one can still construct $\cO(h^\infty)$-quasimodes when working exclusively with the simpler $\cX,\cY$ spaces, at the expense of having to allow for logarithmic terms in the approximate eigenvalues and eigenfunctions.} Instead, for the quasimode construction we shall work with the smaller space of functions on $\tilde M$ which on $I$ are pullbacks of smooth functions on $\bar\Omega\times[0,1)$ while on $E$ they are smooth on $\tilde M$ and vanish to infinite order at $\wt\Omega^\complement$.

  \pfstep{Initial step: $k=0$.} We first compute
  \begin{equation}
  \label{EqErr1st}
    (P-\lambda_0)(\beta^*(\cE u_0)) = (\partial_\nu u_0)\delta(\rho) = h^{-1}(\pa_\nu u_0)\delta_B.
  \end{equation}
  Indeed, the only nonzero contribution arises from $\pa_\rho^2$ in~\eqref{EqTApproxDelta}. We solve this away by appealing to Lemma~\ref{LemmaModelSol} with $\hat g=\pa_\nu u_0$. Let thus $\hat u(\hat\rho,y)=-\pa_\nu u_0(y)G(\hat\rho)$. We then have
  \begin{align*}
    &(P-\lambda_0)\Bigl(\beta^*(\cE u_0) + h F_\ff^\chi\hat u \Bigr) \\
    &\qquad = h^{-1}(\pa_\nu u_0)\delta(\hat\rho) - h\lambda_0 F_\ff^\chi\hat u + h\bigl( P(F_\ff^\chi\hat u)-h^{-2}F_\ff^\chi(\hat P\hat u)\bigr) + h^{-1}F_\ff^\chi(\hat P\hat u).
  \end{align*}
  Since $\hat P\hat u=-(\pa_\nu u_0)\delta(\hat\rho)$, the first and last term cancel. The third term lies in $h\beta^*\cE\CI(\bar\Omega\times[0,1))+\rho_{\Omega^\complement}^\infty\cY_E$ by Lemma~\ref{LemmaExt}. In the second term, note that $F_\ff^\chi\hat u|_I$ is an $h$-independent smooth function on $\bar\Omega$. The exterior part $F_\ff^\chi\hat u|_E$ lies in $\rho_{\Omega^\complement}^\infty\cY_E$. Upon setting
  \[
    \tilde u_0 := \beta^*(\cE u_0) + h F_\ff^\chi\hat u,
  \]
  we have therefore arranged property~\eqref{ItQM2}. Note that property~\eqref{ItQM1} is valid since $F_\ff^\chi\hat u$ is bounded near $\wt\Omega$, and hence $h F_\ff^\chi\hat u$ vanishes at $\wt\Omega$.

  \pfstep{Iteration step: $k\geq 1$. Part I: improvement at $\wt\Omega$.} The error terms we wish to solve away are the leading order terms at $\wt\Omega$ and $\ff$ of
  \[
    \tilde f_{k-1} := h^{-k}(P-\tilde\lambda_{k-1})\tilde u_{k-1} \in \beta^*\cE\bigl(\CI(\bar\Omega\times[0,1))\bigr) + \rho_\ff^{-1}\rho_{\Omega^\complement}^\infty\cY_E.
  \]
  We start with the ansatz
  \[
    \tilde u_{k,0} := \tilde u_{k-1} + h^k\beta^*\cE w,\quad
    \tilde\lambda_k := \tilde\lambda_{k-1} + h^k\mu,
  \]
  where we need to determine $w\in\CI_0(\bar\Omega)$ and $\mu\in\R$. We first work over $\Omega\times[0,1)_h$. For $h>0$, we compute
    \begin{equation}\label{EqOmegaError}
        \begin{split}
        h^{-k}(P_h - \tilde{\lambda}_k)(\tilde{u}_{k, 0})_h &= h^{-k}(-\laplace -
            \tilde\lambda_{k-1})(\tilde{u}_{k-1})_h + (-\laplace -
            \lambda_0)w - \mu u_0\\
            &\quad - \mu((\tilde{u}_{k-1})_h - u_0)- (\tilde\lambda_{k-1} - \lambda_0 + \mu h^k)w.
        \end{split}
    \end{equation}
    We would like the right-hand side to be of order $h$.

    Since \(\tilde\lambda_{k-1} - \lambda_0 = \cO(h)\), we have \((\tilde\lambda_{k-1} - \lambda_0 + \mu h^k)w\in h\CI(\Omega\times[0,1))\). Moreover, we have \(\tilde{u}_{k-1} - u_0\in h\CI\), so \(\mu (\tilde{u}_{k-1} - u_0)\in h\cC^\infty\) as well. Therefore, setting \(f\coloneqq \tilde f_{k-1}|_{\wt\Omega}\in\CI(\bar\Omega)\), the equation~\eqref{EqOmegaError} becomes
    \[
        \begin{split}
            h^{-k}(P - \tilde\lambda_k)\tilde{u}_{k, 0} &= f + (-\laplace - \lambda_0)w - \mu u_0 + h\CI(\Omega\times[0,1)).
        \end{split}
    \]
    Thus, we must choose $w,\mu$ such that they satisfy
    \begin{equation}\label{EqGrushinIterationStep}
        \begin{pmatrix}
            -\dirlap{\Omega} - \lambda_0 & u_0\\
            \iprod{\cdot, u_0} & 0
            \end{pmatrix}\begin{pmatrix}
            w\\ -\mu
            \end{pmatrix} = \begin{pmatrix}
            f\\ 0
        \end{pmatrix}.
    \end{equation}
    By Lemma~\ref{LemGrushinProblemQM} such \(w\) and \(\mu\) exist. For these $w,\mu$ then (and thus $\tilde u_{k,0}$ and $\tilde\lambda_k$), we now return to doing computations on all of $\tilde M$. We then have
    \begin{equation}
    \label{EqErrStep1}
      h^{-k}(P-\tilde\lambda_k)\tilde u_{k,0} = r_{\rm i} + r_{\rm e},\quad r_{\rm i}\in h\beta^*\cE\bigl(\CI(\bar\Omega\times[0,1))\bigr),\ r_{\rm e}\in\rho_\ff^{-1}\rho_{\Omega^\complement}^\infty\cY_E.
    \end{equation}
    Indeed, the only term not accounted for in the calculation on $\Omega\times[0,1)$ arises when the term $-h^{-k}\pa_\rho^2$ of $h^{-k}P$ differentiates $h^k\beta^*\cE w$ and produces a $\delta$-distribution at $\hat\rho=0$; this $\delta$-distribution is equal to $\beta^*(\pa_\nu w)\delta(\rho)=h^{-1}\beta^*((\pa_\nu w)\delta_B)$ and can thus be put into the term $r_{\rm e}$.

    We have now improved the error term as $h\to 0$ in $I$.

  \pfstep{Iteration step: $k\geq 1$. Part II: improvement at $\ff$.} Next, we eliminate the error $r_{\rm e}$ (supported in $E$) in~\eqref{EqErrStep1} to leading order at $\ff$. To do this, set
    \[
      \hat f := R_\ff(h r_{\rm e}) \in \rho_{\Omega^\complement}^\infty\cY_{E,\ff}
    \]
    and use Lemma~\ref{LemmaModelSol} to find
    \begin{equation}
    \label{EqStep2hatu}
      \hat u \in \rho_{\Omega^\complement}^\infty\cX_\ff,\ \text{constant for $\hat\rho<0$ for each $y\in\pa\Omega$, such that}\ \hat P\hat u=-\hat f.
    \end{equation}
    We then compute
    \begin{align*}
      h^{-k}(P-\tilde\lambda_k)(\tilde u_{k,0}+h^{k+1}F_\ff^\chi\hat u) &= r_{\rm i} + h^{-1}\bigl(h r_{\rm e} + F_\ff^\chi(\hat P\hat u)\bigr) \\
        &\qquad + h\bigl(P(F_\ff^\chi\hat u) - h^{-2}F_\ff^\chi(\hat P\hat u)\bigr) - h\tilde\lambda_k F_\ff^\chi\hat u.
    \end{align*}
    Since the distribution $h r_{\rm e}+F_\ff^\chi(\hat P\hat u)=h r_{\rm e}-F_\ff^\chi\hat f\in\rho_{\Omega^\complement}^\infty\cY_{E}$ vanishes at $\ff$ by definition of $\hat f$, it lies in $\rho_\ff\rho_{\Omega^\complement}^\infty\cY_{E}=h\rho_{\Omega^\complement}^\infty\cY_{E}$, and therefore the second term lies in $\rho_{\Omega^\complement}^\infty\cY_{E}$. The third term on the right lies in $h\beta^*\cE(\CI(\bar\Omega\times[0,1)))+\rho_{\Omega^\complement}^\infty\cY_E$ by Lemma~\ref{LemmaExt}. For the last term, we use that $F_\ff^\chi\hat u$ is the sum of a term in $\beta^*\cE(\CI(\bar\Omega\times[0,1)))$ (arising by restriction to $I^\circ$) and a term that is smooth in $E$ and vanishes to infinite order at $\Omega^\complement$; thus $F_\ff^\chi\hat u\in\beta^*\cE(\CI(\bar\Omega\times[0,1)))+\rho_{\Omega^\complement}^\infty\cY_E$, and multiplication by $h\tilde\lambda_k$ produces an extra power of $h$. Setting
    \[
      \tilde u_k := \tilde u_{k,0} + h^{k+1}F_\ff^\chi\hat u,
    \]
    we have thus arranged for $h^{-k}(P-\tilde\lambda_k)\tilde u_k\in h\beta^*\cE(\CI(\bar\Omega\times[0,1)))+\rho_{\Omega^\complement}^\infty\cY_E$, and thus property~\eqref{ItQM2} holds for the value $k$.

    Since
    \begin{equation}
    \label{EqStepComplete}
      \tilde u_k-\tilde u_{k-1} = h^k\beta^*\cE w + h^{k+1}F_\ff^\chi\hat u
    \end{equation}
    with $w\in\CI_0(\bar\Omega)$ and $\hat u$ as in~\eqref{EqStep2hatu}, we also conclude the validity of property~\eqref{ItQM3}.
\end{proof}

\begin{rmk}[Choices]
  In \eqref{EqGrushinIterationStep} we chose the constant \(c\) in Lemma~\ref{LemGrushinProblemQM} to be \(0\). This choice is arbitrary, and as we shall show now, it does not matter. Indeed, take \(c\in\R\) arbitrary and suppose we solve
    \[
        \begin{pmatrix}
            -\dirlap{\Omega} - \lambda_0 & u_0\\
            \iprod{\cdot, u_0} & 0
            \end{pmatrix}\begin{pmatrix}
            v\\ -\mu
            \end{pmatrix} = \begin{pmatrix}
            f\\ c
        \end{pmatrix}.
    \]
    A solution of this system of equations is given by \(v = w + c u_0\) and \(\mu\), where \(w\) and \(\mu\) are a solution of \eqref{EqGrushinIterationStep}. For example then, instead of having the first order expansion \(\tilde{u}_0 + hw + \cO(h^2)\) on $\Omega\times[0,1)$, we now instead have
    \[
        \tilde{u}_0 + h(w + cu_0) + \cO(h^2) = (1 + ch)(\tilde{u}_0 + hw + \cO(h^2)).
    \]
    Different choices of the constants $c$ in each iteration step thus amount to taking the solution constructed above and multiplying it by a polynomial in $h$ with leading order term $1$.
\end{rmk}

\begin{rmk}[Precision]
\label{}
  The above proof gives a bit more: the formula~\eqref{EqStepComplete} shows that $\tilde u_k-\tilde u_{k-1}$ is the sum of an element of $h^k\cE\bigl(\CI(\bar\Omega\times[0,1))\bigr)$ and an element of $\rho_\ff^{k+1}\rho_{\Omega^\complement}^\infty\cX$ with support in $E$. We could elect to keep track of this information and thus (using the arguments that are still to follow) produce quasimodes and subsequently true eigenfunctions of $P_h$ which can be written as the sum of an element of $\cE\bigl(\CI(\bar\Omega\times[0,1)))$ and an element of $\rho_\ff\rho_{\Omega^\complement}^\infty\cX$ with support in $E$. We leave the details to the interested reader.
\end{rmk}

\subsection{Proof of Theorem~\ref{thm:exquasimodes}}

Now that we have constructed quasimodes of arbitrary order, we shall combine them to obtain a
quasimode of infinite order. This utilizes an argument based on the proof of Borel's lemma, see, e.g., \cite[Theorem~1.2.6]{Ho1}.

Let \(n\in\N\) and suppose \(\lambda_n^{\rm D}\) is a simple eigenvalue of \(-\dirlap{\Omega}\) with
corresponding \(L^2\)-normalized eigenfunction \(u_n\in \CI_0(\bar\Omega)\). Denote by \(\tilde{v}_k\in \rho_{\Omega^\complement}^\infty\qms\) and
\(\tilde\mu_k\in \CI([0,1))\) the quasimodes constructed in Proposition~\ref{PropQM}. Let
\(\varphi\in \CIc([0,\infty))\) be such that $\varphi=1$ on $[0,1]$ and $\varphi=0$ on $[2,\infty)$. Let \((\alpha_k)_{k\in\N}\) be a sequence, yet to be determined, tending to \(0\) as
\(n\to\infty\). Let
\[
    \eta_k\coloneqq (\tilde{v}_{k} - \tilde{v}_{k-1})\varphi(h\alpha_k^{-1}) \in \varphi(h\alpha_k^{-1})\cdot\rho_\Omega^k\rho_\ff^{k+1}\rho_{\Omega^\complement}^\infty\cX,
\]
and set
\begin{equation}
\label{EqTildeun}
    \tilde{u}_n\coloneqq \tilde{v}_0 + \sum_{k=1}^\infty\eta_k.
\end{equation}
Note that the sum is finite for every fixed $h>0$, and moreover \(\tilde{u}_h\) is continuous on $\R^d$ and smooth outside of \(\pa\Omega\). For each \(k\in\N\), pick \(\alpha_k\) small enough so that
\[
    \norm{\partial_h^j\partial_x^\alpha\eta_k|_{I^\circ}}_\infty + \norm{\partial_h^j\partial_x^\alpha\eta_k|_{E^\circ}} \le 2^{-k}h^{k/2}\quad\forall\,j\in\N_0,\ \alpha\in \N_0^d\text{ with
    }\abs{j}+\abs{\alpha}\le \frac{k}{2}-1.
\]
The bound by $h^{k/2}$ is automatic since $h\pa_h$ and $h\pa_{x^1},\ldots,h\pa_{x^d}$ lift to a smooth vector fields on $\tilde M$. The extra prefactor $2^{-k}$ can be obtained by choosing $\alpha_k$ small enough. Given any $k_0\in\N$, it then follows that $\sum_{k=2 k_0+2}^\infty\eta_k$ converges in $h^{k_0}\cC^{k_0}(\R^d\times[0,1))$. On the other hand, the finite sum $\tilde v_0+\sum_{k=1}^{2 k_0+1}\eta_k$ trivially defines an element of $\rho_{\Omega^\complement}^\infty\cX$. We also recall that each $\tilde v_0$ and $\eta_k$ has $x$-support in a fixed compact subset of $\R^d$.

We can similarly define the smooth function
\[
    \tilde\lambda_n \coloneqq \tilde\mu_0 + \sum_{k=1}^\infty(\tilde\mu_k -
    \tilde\mu_{k-1})\varphi(h\beta_k^{-1})
\]
with the \((\beta_k)_{k\in\N}\) chosen so that the series converges in $\CI([0,1)_h)$. For later use, we summarize our conclusions so far as
\begin{equation}
\label{EqTildeunMem}
  \tilde u_n \in \rho_{\Omega^\complement}^\infty\cX,\quad
  \tilde\lambda_n \in \CI([0,1)_h),\ \tilde\lambda_n(0)=\lambda_0.
\end{equation}

By construction of \(\tilde{u}_n\), points~\eqref{ItIQM1} and \eqref{ItIQM2} of Theorem~\ref{thm:exquasimodes} are satisfied. Consider point~\eqref{ItIQM3}. Since the $L^2$-norm of the tail of the series~\eqref{EqTildeun} for $k\geq 4$ is bounded by $\cO(h)$ (since this tail converges in $h\cC^1(\R^d\times[0,1))$), it suffices to observe that
\[
  \bigl| \|(\tilde u_0)_h\|_{L^2(\R^d)} - \|u_0\|_{L^2}\bigr| \leq \|\tilde u_0-u_0\|_{L^2},
\]
where $|\tilde u_0-u_0|\lesssim h$ is supported in a compact set and thus has $L^2$-norm $\lesssim h$; and each $\eta_k$ is similarly bounded by $|\eta_k|\lesssim h^k$ by part~\eqref{ItQM3} of Proposition~\ref{PropQM}, and hence has $L^2$-norm $\lesssim h^k\leq h$.

Turning to point~\eqref{ItIQM4} of Theorem~\ref{thm:exquasimodes}, we claim that
\begin{equation}
\label{EqIQM4Pf}
  (P-\tilde\lambda_n)\tilde u_n \in h^m\cY\quad\forall\,m\in\N.
\end{equation}
This membership holds if we only keep finitely (but sufficiently) many terms in the series definitions of $\tilde u_n$ and $\tilde\lambda_n$, as follows from Proposition~\ref{PropQM}\eqref{ItQM2}. Since a sufficiently late tail of the series for $\tilde u_n$ converges in $h^{m+2}\cY$, we obtain~\eqref{EqIQM4Pf}, and thus $(P-\tilde\lambda_n)\tilde u_n\in h^\infty\cY$. Let now \(N\in \N\) and \(f = (P -
\tilde\lambda_n)\tilde{u}_n\in h^N\es\) with the decomposition \(f = u_s + g\delta_B\), where \(u_s\) is
a compactly supported function on each level set of \(h\) and smooth on \(I\) and \(E\), and \(g\) is a
smooth function on \(B\) so that both \(u_s\) and \(g\) are pointwise of size \(\cO(h^N)\) as \(h\to 0\). Then
\[
    \norm{(P - \tilde\lambda_n)\tilde{u}_n}_{H^{-1}(\R^d)} \le \norm{u_s}_{H^{-1}(\R^d)} +
    h\norm{g}_{L^2(\R^d)} \lesssim h^N
\]
for \(h\) small enough. (Here we use $\delta_B=h\delta_{\pa\Omega}$ and $\norm{g\,\delta_{\pa\Omega}}_{H^{-1}(\R^d)}\lesssim h^N$, which follows from
\[
  \abs{\iprod{g\,\delta_{\pa\Omega}, f}} \lesssim \norm{g f|_{\pa\Omega}}_{L^2(\pa\Omega)}\lesssim
  \norm{g}_{L^\infty(\pa\Omega)}\norm{f|_{\pa\Omega}}_{L^2(\pa\Omega)} \lesssim h^N\norm{f}_{H^1(\R^d)}
\]
by duality; we use the trace theorem in the last inequality.) Since \(N\) was arbitrary, this shows the last point of Theorem~\ref{thm:exquasimodes} and thus finishes its proof. (We recall from~\S\ref{SI} that Lemma~\ref{LemmaQM} implies that $\spec(P_h)$ contains a point $\cO(h^\infty)$-close to $\tilde\lambda_n(h)$.)

\subsection{Explicit first order correction for eigenvalues: proof of Proposition~\ref{PropIFirstOrderExp}}
\label{SsExplicitFirstOrder}

We show how to explicitly compute the first order correction to the eigenvalue obtained via the procedure in the proof of Proposition~\ref{PropQM}. We make an ansatz for a correction inside the domain \(\Omega\), which vanishes at the boundary. Let \(\tilde{u}_0\) and \(\tilde\lambda_0=\lambda_0\) be as in the initial step in the proof of Proposition~\ref{PropQM}, so $\tilde{u}_0=\beta^*\cE u_0-h F_\ff^\chi((\pa_\nu u_0)G)$. Put then
\[
    \tilde{u}_{1,0}\coloneqq \tilde{u}_0 + h\beta^* E_0 w \in \qms
    \quad\text{and}\quad
    \tilde{\lambda}_1(h) = \tilde{\lambda}_0 + h\lambda_1 = \lambda_0 + h\lambda_1,
\]
where \(w\in \CI_0(\bar\Omega)\) and \(\lambda_1\in\R\) are yet to be chosen. Define \(g=F_\ff^\chi((\pa_\nu u_0)G)\), so on $\Omega$ and in normal coordinates, \(g=\partial_\nu u_0(y)\chi(\rho)\) on \(N_\delta\pa\Omega\cap\Omega\) and \(g=0\) on the rest of \(\Omega\). As in~\eqref{EqOmegaError}, we compute on $\Omega$ that
\[
    \begin{aligned}
        (P_h - \tilde{\lambda}_1(h))\tilde{u}_{1,0} =\
        &h[- (-\laplace - \lambda_0)g + (-\laplace - \lambda_0)w - \lambda_1 u_0] + h^2\lambda_1(g - w).
    \end{aligned}
\]
As in the iteration step in the proof of Proposition~\ref{PropQM}, we can determine \(w\) and \(\lambda_1\) by requiring the \(\cO(h)\)-term to vanish, i.e.,
\begin{equation}\label{EqFirstOrderCorr}
    (-\dirlap{\Omega} - \lambda_0) w - \lambda_1 u_0 = (-\laplace - \lambda_0)g.
\end{equation}
Fix some \(c\in\C\). Then a pair \((w, \lambda_1)\in \CI_0(\bar\Omega)\oplus \C\) solves
\eqref{EqFirstOrderCorr} if
\[
    \begin{pmatrix}
        -\dirlap{\Omega} - \lambda_0 & u_0\\
        \iprod{\cdot, u_0} & 0
        \end{pmatrix}\begin{pmatrix}
        w \\ -\lambda_1
        \end{pmatrix}
        =
        \begin{pmatrix}
            f\\ c
        \end{pmatrix},
\]
where \(f = (-\laplace - \lambda_0)g\in \CI(\bar\Omega)\). From Lemma~\ref{LemGrushinProblemQM} it follows that \((w, \lambda_1)\) exist and are unique. Moreover, by inspecting the proof of
Lemma~\ref{LemGrushinProblemQM} one sees that \(\lambda_1 u_0 + f\) must be orthogonal to \(u_0\) in
the \(L^2\)-inner product. Thus,
\begin{equation}\label{EqQFirstOrderComp}
    \begin{split}
        0 &= \lambda_1 + \iprod{(-\laplace - \lambda_0)g, u_0}_{L^2(\Omega)}\\
        &= \lambda_1 + \iprod{(-\laplace - \lambda_0)g, u_0}_{L^2(\Omega)} -
        \iprod{g, (-\laplace - \lambda_0)u_0}_{L^2(\Omega)}\\
        &= \lambda_1 + \norm{\partial_{\nu} u_0}_{L^2(\pa\Omega)}^2,
    \end{split}
\end{equation}
where we used that \(u_0\) is an eigenfunction and Green's identity as well as the definition of $g$. This proves Proposition~\ref{PropIFirstOrderExp}.

\section{Correcting quasimodes to eigenfunctions}
\label{STCorrectingQuasimodes}

Recall from~\eqref{EqTildeunMem} and~\eqref{EqIQM4Pf} and in the notation of Definitions~\ref{DefQMSpace} and \ref{DefESpace} that the proof of Theorem~\ref{thm:exquasimodes} produces
\[
  u \in \rho_{\Omega^\complement}^\infty\cX,\quad
  \lambda \in \CI([0,1)_h),
\]
such that $(P-\lambda)u\in h^\infty\cY$, and $\lambda(0)=\lambda_0$ is a simple Dirichlet eigenvalue of $\Omega$; here $u_{\wt\Omega}=u_0$ is the corresponding $L^2$-normalized Dirichlet eigenfunction. Our goal is to correct the quasimode \((u_h,\lambda_h)\) by \(\cO(h^\infty)\) correction terms to true eigenfunctions and eigenvalues of \(P_h\). 

Unlike the quasimodes $u_h$, the eigenfunctions of \(P_h\) cannot compactly supported by unique continuation for second order elliptic PDE, so we must seek the correction terms in a space larger than \(\qms\) (which required vanishing in \(\rho>\delta\)). Not aiming for maximum precision (e.g., one may conjecture the exponential decay of eigenfunctions as $\rho/h\to\infty$), the following space will be convenient:

\begin{definition}[Function space]
\label{DefCorr}
  We write \(\qms'\) for the space of functions \(v\) on \(\tilde M\) with the following properties:
  \begin{enumerate}
  \item \(v\) is continuous;
  \item \(v|_{I^\circ}\in\cC^\infty(I)\), \(v|_{E^\circ}\in\cC^\infty(E)\);
  \item for \(\rho\geq\delta\) and for all \(\alpha\in\N_0^d\) and \(N\in\N\), we have
    \[
        |\partial_x^\alpha v(x)| \leq C_{\alpha,N}\jbr{x}^{-N}.
    \]
  \end{enumerate}
\end{definition}

The correction terms will lie in $h^\infty\cX'$, and the eigenfunctions themselves will lie in $\cX'$:

\begin{thm}[Correction]
\label{ThmT}
  There exist \(h_0>0\) and \(v\in h^\infty\qms'\), \(\mu\in h^\infty\CI([0,1))\) such that, for \(0<h<h_0\),
  \begin{equation}
  \label{EqTEq}
    (P-(\lambda+\mu))(u+v)=0.
  \end{equation}
  That is, for all \(h\in(0,h_0)\), the function \(u_h+v_h\) is an eigenfunction of \(P_h\) with eigenvalue \(\lambda_h+\mu_h\).
\end{thm}

\begin{cor}[True eigenfunctions and eigenvalues]
\label{CorT}
  For small \(h>0\), denote by \(\lambda_h\) the unique eigenvalue of \(P_h\) which is \(\cO(h)\)-close to the simple Dirichlet eigenvalue \(\lambda_0\) of $\Omega$. Then:
  \begin{enumerate}
  \item there exists \(\lambda\in\CI([0,h_0))\) such that $\lambda_h=\lambda(h)$ for $h\in(0,h_0)$;
  \item for a suitable choice \(u_h\in\ker(P_h-\lambda_h)\cap L^2(\R^d)\) of $L^2$-normalized eigenfunctions, we have \(u\in\qms'\) where \(u=u_h\) on the \(h\)-level set $\{h\}\times\R^d\subset\tilde M$ for \(h\in(0,h_0)\).
  \end{enumerate}
\end{cor}

Perhaps somewhat surprisingly, equation~\eqref{EqTEq} is a \emph{nonlinear} equation in the unknowns \(\mu,v\); moving its linear part to the left-hand side, it reads
\begin{equation}
\label{EqTEq2}
  \begin{pmatrix} P-\lambda & -u \end{pmatrix} \begin{pmatrix} v \\ \mu \end{pmatrix} = -f + \mu v,\quad f:=(P-\lambda)u\in h^\infty\cY.
\end{equation}
The solution of this equation will not be unique: after all, given one solution \(v\), pick any
function \(\kappa\in h^\infty\CI([0,h_0))\); then also \((1+\kappa)(u+v)=u+(\kappa u+(1+\kappa)v)\) is
an eigenfunction of \(P-(\lambda+\mu)\). To fix this indeterminacy (which roughly amounts to adding to
\(v\) a multiple of \(u_0\), plus further lower order corrections), we augment~\eqref{EqTEq2} with the
condition \(\iprod{v,u_0^\sharp}_{L^2(\Omega)}=0\), where we fix any\footnote{The function \(u_0\)
satisfies the second, but not the first condition. The first condition will lead to convenient
simplifications in our arguments below.}
\[
    u_0^\sharp\in\CIc(\Omega),\quad \iprod{u_0,u_0^\sharp}_{L^2(\Omega)}=1.
\]
Altogether, we thus wish to solve
\begin{equation}
\label{EqTFin}
  P^\aug(v,\mu)=\begin{pmatrix}-f+\mu v \\ 0\end{pmatrix},\quad P^\aug \coloneqq \begin{pmatrix} P-\lambda & -u \\
  \iprod{\cdot,u_0^\sharp}_{L^2} & 0 \end{pmatrix}.
\end{equation}
We emphasize that this is a \emph{family} of \emph{nonlinear} equations, one for each \(h>0\). Our strategy for solving~\eqref{EqTFin} for small $h>0$ is as follows.
\begin{enumerate}
\item We shall apply the contraction mapping principle on an appropriate Sobolev space \emph{with \(h\)-dependent} norm, at once for all small \(h\), to get a solution \(v,\mu\) with finite regularity and \(h\)-decay as \(h\searrow 0\).
\item The infinite order vanishing of \(v,\mu\) as \(h\searrow 0\) will easily follow from that of the error term \(f\).
\item An elliptic bootstrap improves the regularity of \(v\), and differentiation in the parameter \(h\) gives the smoothness of $v,\mu$ in \(h\).
\end{enumerate}

The main work consists in proving uniform estimates for (the inverse of) the linear operator \(P^\aug\); this is the content of~\S\ref{SsTUnif}. The above three steps can afterwards be completed relatively quickly; see~\S\ref{SsTPf}.

\subsection{Uniform estimates for the augmented operator}
\label{SsTUnif}

In order to determine the correct (\(h\)-dependent) norms for the uniform analysis of \(P^\aug\), consider first \(P_h\) in the three regimes of interest.
\begin{enumerate}
\item\label{ItTUnifInt} Near \(\widetilde\Omega^\circ\), \(P_h=-\laplace\) has quadratic form \(\int
    |\nabla u|^2\odif{x}\), suggesting the use of standard derivatives \(\partial_x\) to measure regularity.
\item\label{ItTUnifExt} Near \((\widetilde\Omega^\complement)^\circ\), \(h^2 P_h=-h^2\laplace+1\) has
    quadratic form \(\int |u|^2+|h\nabla u|^2\odif{x}\), suggesting the use of semiclassical
    derivatives \(h\partial_x\) to measure regularity.
\item\label{ItTUnifff} Near \({\ff}^\circ\) and recalling \(\hat\rho=\frac{\rho}{h}\), we have (schematically) \(h^2 P_h=-\partial_{\hat\rho}^2+(h\partial_y)^2+H(\hat\rho)\); here \(\partial_y\) refers to derivatives tangential to \(\partial\Omega\). (See~\eqref{EqSchrResolved} for the full expression.)
  \begin{itemize}
  \item For \(\hat\rho\geq 0\), the associated quadratic form is coercive in \(u\),
      \(\partial_{\hat\rho}u=h\partial_\rho u\), \(h\partial_y u\).
  \item For \(\hat\rho\leq 0\) on the other hand, the quadratic form only controls
      \(\partial_{\hat\rho}u\) and \(h\partial_y u\). Getting control on \(u\) itself requires integrating
      in \(\hat\rho\), which (via the Hardy inequality) loses a power of \(\hat\rho\). We rephrase this
      by measuring regularity with respect to \(\jbr{\hat\rho}\partial_{\hat\rho}\),
      \(h\jbr{\hat\rho}\partial_y\) relative to an \(L^2\)-space weighted by powers of \(\jbr{\hat\rho}\).
  \end{itemize}
\end{enumerate}

In the transition from \({\ff}\) to \(\widetilde\Omega^\complement\), the derivatives \(h\partial_\rho\) and \(h\partial_y\) from~\eqref{ItTUnifff} match the semiclassical derivatives in~\eqref{ItTUnifExt}. On the other hand, the derivatives \(\partial_x\) suggested by~\eqref{ItTUnifInt} do not match up with the derivatives on \({\ff}\) in \(\hat\rho\leq 0\) from~\eqref{ItTUnifff}. Rather, upon changing coordinates from \(\hat\rho,y\) to \(\rho=h\hat\rho,y\), we have
\[
    \jbr{\hat\rho}\partial_{\hat\rho}=\Bigl(1+\frac{\rho^2}{h^2}\Bigr)^{\frac12}h\partial_\rho = (h^2+\rho^2)^{\frac12}\partial_\rho,\quad
    h\jbr{\hat\rho}\partial_y=(h^2+\rho^2)^{\frac12}\partial_y,
\]
which restrict to \(h=0\) as \(\rho\partial_\rho\), \(\rho\partial_y\). This suggests measuring regularity near \(\widetilde\Omega\) using these rescaled vector fields; they are very natural from the perspective of scaling \(\laplace\) near boundary points of \(\Omega\), and are indeed the vector fields naturally arising in interior Schauder estimates. In the standard parlance of geometric singular analysis, they are a basis of the space of \emph{0-vector fields} introduced by Mazzeo--Melrose \cite{MazzeoMelroseHyp}.

\begin{rmk}[Quasimode construction vs.\ uniform estimates]
\label{RmkTUnifQvsT}
  Unlike in the quasimode construction, where we could afford to drop the term \((h\partial_y)^2\) at \({\ff}\), we must keep it now. The reason is that our proof of uniform estimates for \(P_h\) below relies on its uniform ellipticity (with respect to appropriate basis vector fields); estimates for the operator \(-\partial_{\hat\rho}^2+H(\hat\rho)\) on the other hand can, of course, never recover \(2\) (weighted) full derivatives, since inverting this operator does not gain any tangential, i.e., $y$-, regularity at all.
\end{rmk}

Globally then, we shall test for regularity using the vector fields
\(\rho_{\ff}\rho_{\Omega^\complement}\partial_z\): they restrict to \(\widetilde\Omega\) and
\(\widetilde\Omega^\complement\) as \(\rho\partial_z\) and \(h\partial_z\), respectively, and to
\({\ff}\) as positive multiples of \(h\partial_z\) when \(\hat\rho\geq 0\) and
\(\jbr{\hat\rho}\partial_{\hat\rho}\), \(h\jbr{\hat\rho}\partial_y\) when \(\hat\rho\leq 0\). We are
thus led to define:

\begin{definition}[Function spaces and norms]
\label{DefTSpace}
  Let \(\alpha,\beta,\gamma\in\R\) and \(k\in\N_0\).
  \begin{enumerate}
  \item{\rm (\(\Omega\).)} For \(u\in\CIc(\Omega)\), we define\footnote{We write \(H_{\rm
      z}^k(\Omega)\), with `\({\rm z}\)' for `zero', to avoid confusion of this space with the closure
  \(H_0^k(\Omega)\) of \(\CIc(\Omega)\) with respect to the standard \(H^k\)-norm.}
    \begin{equation}
    \label{EqTSpacez}
    \norm{u}_{\rho^\alpha H_{\rm z}^k(\Omega)}^2 \coloneqq \sum_{|\zeta|\leq k} \norm{\rho^{-\alpha} (\rho\partial)^\zeta u}_{L^2(\Omega)}^2.
    \end{equation}
    The space \(\rho^\alpha H_{\rm z}^k(\Omega)\) is the completion of \(\CIc(\Omega)\) under this
    norm. We denote the dual space of \(\rho^\alpha H_{\rm z}^1(\Omega)\) by \(\rho^{-\alpha} H_{\rm
    z}^{-1}(\Omega)\). (It consists of all distributions on \(\Omega\) of the form \(\sum_{|\zeta|\leq
    1}\rho^{-\alpha}(\rho\partial)^\zeta u_\zeta\) where \(u_\zeta\in L^2(\Omega)\).)
  \item{\rm (\(\Omega^\complement\).)} We similarly define \(H_{0,h}^k(\Omega^\complement)\) as the completion of \(\CIc((\Omega^\complement)^\circ)\) with respect to the norm
    \[
      \norm{u}_{H_{0,h}^k(\Omega^\complement)}^2 \coloneqq \sum_{|\zeta|\leq k} \norm{ (h\partial)^\zeta u }_{L^2(\Omega^\complement)}^2.
    \]
    We write $H^k_h(\Omega^\complement)$ for the completion of $\CIc(\Omega^\complement)$ (i.e., without the requirement of vanishing at $\pa\Omega$) with respect to the same norm, now denoted $\|\cdot\|_{H_h^k(\Omega^\complement)}$. The space \(H^{-1}_{0,h}(\Omega^\complement)\) is the \(L^2\)-dual of \(H^1_{0,h}(\Omega)\).
  \item{\rm (\({\ff}\).)} Define \(L^2(\R\times\partial\Omega)\) using the product of the Lebesgue measure on \(\R\) and the volume measure on \(\partial\Omega\). Write \(\hat\rho_<\), \(\hat\rho_>\) for smooth positive functions of \(\hat\rho\) such that
    \begin{equation}
    \label{EqTHatRho}
      \hat\rho_< = \begin{cases} |\hat\rho|, & \hat\rho\leq -1, \\ 1, & \hat\rho\geq 0, \end{cases}\qquad
      \hat\rho_> = \begin{cases} 1, & \hat\rho\leq 0, \\ \hat\rho, & \hat\rho\geq 1. \end{cases}
    \end{equation}
    For functions \(u\in\CIc(\R\times\partial\Omega)\), we then define
    \begin{equation}
    \label{EqTHff}
      \norm{u}_{\hat\rho_<^{-\alpha}\hat\rho_>^{-\gamma}H_{{\ff},h}^k}^2 \coloneqq \sum_{j+m\leq k} \norm{ \hat\rho_<^\alpha\hat\rho_>^\gamma (\hat\rho_<\partial_{\hat\rho})^j (\hat\rho_<h\nabla_{\pa\Omega})^m u }_{L^2(\R\times\partial\Omega)}^2
    \end{equation}
    The space \(\hat\rho_<^\alpha\hat\rho_>^\gamma H_{{\ff},h}^{-1}\) is the \(L^2\)-dual of \(\hat\rho_<^{-\alpha}\hat\rho_>^{-\gamma} H_{{\ff},h}^1\).
  \item{\rm (Combination.)} We let
    \[
      \norm{u}_{\rho_\Omega^\alpha\rho_{\ff}^\beta\rho_{\Omega^\complement}^\gamma X_h^k}^2 \coloneqq \sum_{|\zeta|\leq k} \norm{ \rho_\Omega^{-\alpha}\rho_{\ff}^{-\beta}\rho_{\Omega^\complement}^{-\gamma} (\rho_{\ff}\rho_{\Omega^\complement}\partial)^\zeta u}_{L^2(\R^d)}^2.
    \]
    The space \(\rho_\Omega^{-\alpha}\rho_{\ff}^{-\beta}\rho_{\Omega^\complement}^{-\gamma}X_h^{-1}\) is the \(L^2\)-dual of \(\rho_\Omega^\alpha\rho_{\ff}^\beta\rho_{\Omega^\complement}^\gamma X_h^1\).
  \end{enumerate}
\end{definition}

Given the preceding discussion, we have simple relationships between the (norms of the) spaces in Definition~\ref{DefTSpace}. To state them, let us fix cutoff functions \(\chi_\Omega,\chi_{\ff},\chi_{\Omega^\complement}\in\CI(M)\) by
\begin{equation}
\label{EqTCutoffs}
  \chi_\Omega = \chi_0(\rho_\Omega),\quad
  \chi_{\ff} = \chi_0(\rho_{\ff}),\quad
  \chi_{\Omega^\complement} = \chi_0(\rho_{\Omega^\complement});
\end{equation}
here \(\chi_0\in\CI([0,\infty))\) equals \(1\) on \([0,\frac12]\) and \(0\) on \([\frac34,\infty)\), and the defining functions are as in~\eqref{EqIDefFn}. This means that \(\chi_\Omega\) equals \(1\) near \(\widetilde\Omega\), and transitions to \(0\) in the interval \(-h\lesssim\rho\leq -h\) (i.e., \(-1\lesssim\hat\rho\leq -1\)), analogously for \(\chi_{\Omega^\complement}\); and \(\chi_{\ff}\) equals \(1\) for \(\rho\) near \(0\), and \(0\) for \(\rho\) sufficiently far from \(0\). We emphasize that the functions~\eqref{EqTCutoffs} depend on \(h\). See Figure~\ref{fig:suppbndry2}.

\begin{figure}
    \centering
    \begin{tikzpicture}[scale=1.5]
        \draw[thick, -{Latex[length=2mm]}] (-2.0, 0.0) -- node[below, midway] {\(\wt\Omega\)} (-0.5, 0.0) arc
            [start angle=180, end angle=0, radius=0.5] --  node[below, midway]
            {\(\wt\Omega^\complement\)} (2.2, 0.0) node[anchor=south west] {\(\rho\)};
        \draw[thick, -{Latex[length=2mm]}] (0.0, 0.5) -- (0.0, 2.0) node[anchor=south west] {\(h\)};
        \fill[opacity=0.2] ({90 + atan(1)}:{-2 / cos(90 + atan(1))}) -- ({90 + atan(1)}:0.5) arc
            [start angle={90 + atan(1)}, end angle={90 + atan(2)}, radius=0.5] -- ({90 +
            atan(2)}:{-2 / cos(90 + atan(2))}) -- cycle;

        \fill[opacity=0.3] (-2.0, 0) -- (-0.5, 0) arc [start angle=180, end
            angle={90 + atan(2)}, radius=0.5] -- ({90 + atan(2)}:{-2 / cos(90 + atan(2))}) --
            cycle;

        \node[] at (-1.25, 0.25) {\(\supp\chi_\Omega\)};

        \fill[opacity=0.2] ({90 - atan(1)}:{2 / cos(90 - atan(1))}) -- ({90 - atan(1)}:0.5) arc
            [start angle={90 - atan(1)}, end angle={90 - atan(2)}, radius=0.5] -- ({90 -
            atan(2)}:{2 / cos(90 - atan(2))}) -- cycle;

        \fill[opacity=0.3] (2.0, 0) -- (0.5, 0) arc [start angle=0, end
            angle={90 - atan(2)}, radius=0.5] -- ({90 - atan(2)}:{2 / cos(90 - atan(2))}) --
            cycle;

        \node[] at (1.35, 0.25) {\(\supp\chi_{\Omega^\complement}\)};

    \end{tikzpicture}
    \caption{The supports of \(\chi_\Omega\) and \(\chi_{\Omega^\complement}\) are the shaded
        regions on the left- and the right-hand side, respectively. The darker regions show where
    the cutoffs are equal to \(1\), while the lighter shaded regions show the transition regions.}\label{fig:suppbndry2}
\end{figure}
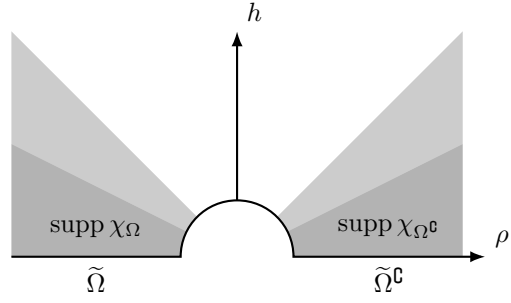

\begin{lemma}[Norm equivalences]
\label{LemmaTNormEq}
  Let \(\alpha,\beta,\gamma\in\R\), \(k\in\N_0\). Then we have the uniform norm equivalences\footnote{In~\eqref{EqTNormEqOmega}, we omit weights at \(\widetilde\Omega^\complement\) since \(\rho_{\Omega^\complement}=1\) on \(\supp\chi_\Omega\); analogously for~\eqref{EqTNormEqOmegaC} and weights at \(\widetilde\Omega\).}
  \begin{align}
  \label{EqTNormEqOmega}
    \norm{\chi_\Omega u}_{\rho_{\ff}^\alpha X_h^k} &\sim \norm{\chi_\Omega u}_{\rho^\alpha H^k_{\rm z}(\Omega)}, \\
  \label{EqTNormEqOmegaC}
    \norm{\chi_{\Omega^\complement}u}_{\rho_{\ff}^\beta\rho_{\Omega^\complement}^\beta X_h^k} &\sim h^{-\beta}\norm{\chi_{\Omega^\complement} u}_{H_{0,h}^k(\Omega^\complement)}, \\
  \label{EqTNormEqff}
    \norm{\chi_{\ff} u}_{\rho_\Omega^\alpha\rho_{\ff}^\beta\rho_{\Omega^\complement}^\gamma X_h^k} &\sim h^{-\beta+\frac12} \norm{\chi_{\ff} u}_{\hat\rho_<^{-\alpha+\beta}\hat\rho_>^{-\gamma+\beta}H_{{\ff},h}^k}.
  \end{align}
\end{lemma}
\begin{proof}
  On \(\supp\chi_\Omega\), we have \(\rho_{\ff}=a\rho\) and thus also \(\rho_{\ff}\partial=a\rho\partial\) where \(a\coloneqq\frac{\rho_{\ff}}{\rho}\) and \(a^{-1}\) are smooth functions on \(\supp\chi_\Omega\). (This follows from the fact that \(\rho\) and \(\rho_{\ff}\) are both defining functions of \({\ff}\) on \(\supp\chi_\Omega\).) This implies~\eqref{EqTNormEqOmega}. One similarly obtains~\eqref{EqTNormEqOmegaC} from \(\rho_{\ff}\rho_{\Omega^\complement}=b h\) where \(b,b^{-1}\) are smooth functions on \(\supp\chi_{\Omega^\complement}\).

  Finally, consider~\eqref{EqTNormEqff}. Since \(\rho_\Omega\rho_{\ff}\rho_{\Omega^\complement}=c h\)
  with \(c,c^{-1}\in\CI(M)\), it suffices to prove~\eqref{EqTNormEqff} for \(\beta=0\). On
  \(\supp\chi_{\ff}\), we may moreover take \(\rho_\Omega=\hat\rho_<^{-1}\) and
  \(\rho_{\Omega^\complement}=\hat\rho_>^{-1}\) for the weights; thus we may further
  reduce to the case \(\alpha=\gamma=0\). On \(\supp\chi_{\ff}\), we may instead of
  \(\rho_{\ff}\rho_{\Omega^\complement}\partial\) consider derivatives along
  \(\rho_\Omega^{-1}h\partial\), or equivalently \(\hat\rho_<h\partial\), i.e., \(\hat\rho_<\partial_{\hat\rho}\) and \(\hat\rho_<h\nabla_{\pa\Omega}\). It thus
  remains to consider the case \(k=0\), which follows from the fact that the Lebesgue measure is a
  bounded positive multiple of \(\odif{\rho,y}\) where \(\odif{y}\) is the surface measure on
  \(\partial\Omega\), and
  \[
      \norm{u}_{L^2(\R_\rho\times\partial\Omega)}^2 = \iint |u(\rho,\omega)|^2\odif{\rho, y} = \iint
      h|u(h\hat\rho,\omega)|^2\odif{\hat\rho, y} =
      h\norm{u}_{L^2(\R_{\hat\rho}\times\partial\Omega)}^2,\quad \hat\rho=\frac{\rho}{h}.
  \]
  This completes the proof.
\end{proof}

\subsubsection{Grushin problem at \texorpdfstring{\(\widetilde\Omega\)}{Omega}}

Formally restricting the operator \(P^\aug\) in~\eqref{EqTFin} to \(\widetilde\Omega\) yields the operator
\begin{equation}
\label{EqTModelOmega}
L^\aug_\Omega \coloneqq \begin{pmatrix} -\laplace-\lambda_0 & -u_0 \\ \iprod{\cdot,u_0^\sharp}_{L^2(\Omega)} & 0 \end{pmatrix}.
\end{equation}
We shall prove:

\begin{prop}[Inversion of \(L_\Omega^\aug\)]
\label{PropTOmega}
  The operator
  \[
    L^\aug_\Omega \colon \rho H_{\rm z}^1(\Omega) \oplus \C \to \rho^{-1}H_{\rm z}^{-1}(\Omega) \oplus \C
  \]
  is invertible.
\end{prop}

We first clarify the space \(\rho H_{\rm z}^1(\Omega)\):
\begin{lemma}[Function space equality]
\label{LemmaTH0}
  \(H_0^1(\Omega)=\rho H_{\rm z}^1(\Omega)\), and this space is moreover the closure of \(\CIc(\Omega)\) with respect to \(\norm{\nabla u}_{L^2(\Omega)}\).
\end{lemma}
\begin{proof}
  Since \(\rho^{-1}(\rho\partial)=\partial\), it suffices to control the \(\zeta=0\) term of~\eqref{EqTSpacez}, which is delicate only near the boundary. There, we use the elementary inequality
  \begin{alignat*}{2}
      &\int_0^\infty x^{-2}|u(x)|^2\odif{x} = \biggl\| \int_0^1 u'(t\cdot)\,\dd t\biggr\|_{L^2([0,\infty)_x)}^2 \\
      &\qquad\leq \biggl(\int_0^1 \|u'(t\cdot)\|_{L^2([0,\infty)_x)}\,\dd t\biggr)^2 = \biggl(\int_0^1 t^{-1/2} \|u'\|_{L^2}\,\dd t\biggr)^2 &&= 4\int_0^\infty |u'(x)|^2\odif{x} \\
      &\qquad&&=4\int_0^\infty x^{-2}|x u'(x)|^2\odif{x},
  \end{alignat*}
  valid for \(u\in\CIc((0,\infty))\). The final statement follows from the Poincar\'e inequality.
\end{proof}

Thus, the space \(H^{-1}(\Omega)=H_0^1(\Omega)^*\) is equal to \(\rho^{-1}H_{\rm z}^{-1}(\Omega)\).

\begin{proof}[Proof of Proposition~\usref{PropTOmega}]
    For \(u\in\CIc(\Omega)\), we have \(\iprod{-\laplace u,u}_{L^2(\Omega)}=\int_\Omega |\nabla
    u|^2\odif{x}\). By Lemma~\ref{LemmaTH0}, this is equivalent to (i.e., bounded from above and
    below by universal constants times) the squared \(\rho H^1_{\rm z}(\Omega)\)-norm of \(u\). By
    Cauchy--Schwarz, we therefore have
  \[
    \norm{u}_{\rho H^1_{\rm z}(\Omega)} \lesssim \norm{\laplace u}_{\rho^{-1}H^{-1}_{\rm z}(\Omega)} \lesssim \norm{(-\laplace-\lambda_0)u}_{\rho^{-1}H_{\rm z}^{-1}(\Omega)} + \norm{u}_{\rho^{-1}H_{\rm z}^{-1}(\Omega)}.
  \]
  Note that the space \(\rho H_{\rm z}^1(\Omega)\cap\ker(-\laplace-\lambda_0)\) consists of the
  Dirichlet eigenfunctions of \(-\laplace\) with eigenvalue \(\lambda_0\), and hence it is 1-dimensional and
  spanned by \(u_0\). Since the inclusion \(\rho H^1_{\rm z}(\Omega)\to\rho^{-1}H_{\rm z}^{-1}(\Omega)\)
  is compact (as it factors through the compact inclusion $\rho H^1_{\rm z}(\Omega)=H^1_0(\Omega)\hra L^2(\Omega)$), the final term on the right can be dropped (upon increasing the constant further) when
  \(u\) lies in a fixed subspace of \(\rho H_{\rm z}^1(\Omega)\) that is complementary to \(\mathspan\{u_0\}\). Since
  the functional \(\iprod{\cdot,u_0^\sharp}\) is nonzero on \(\mathspan\{u_0\}\), we conclude
  \begin{equation}
  \label{EqTOmega1}
    \norm{u}_{\rho H_{\rm z}^1(\Omega)} \lesssim \norm{(-\laplace-\lambda_0)u}_{\rho^{-1}H_{\rm
    z}^{-1}(\Omega)} + |\iprod{u,u_0^\sharp}|.
  \end{equation}
  Moreover, the \(L^2\)-orthogonal complement of \((-\laplace-\lambda_0)(\rho H_{\rm z}^1(\Omega))\) is spanned by \(u_0\). Therefore,
  \begin{align*}
    \norm{(u,c)}_{\rho H^1_{\rm z}(\Omega)\oplus\C} &\sim \norm{u}_{\rho H^1_{\rm z}(\Omega)} + |c| \\
      &\lesssim \norm{(-\laplace-\lambda_0)u - c u_0}_{\rho^{-1}H_{\rm z}^{-1}(\Omega)} +
      |\iprod{u,u_0^\sharp}| \sim \norm{L_\Omega^\aug(u,c)}_{\rho^{-1}H^{-1}_{\rm z}(\Omega)\oplus\C}.
  \end{align*}

  Surjectivity is proved as follows. Given \(f\in\rho^{-1}H^{-1}_{\rm z}(\Omega)\), \(w\in\C\), let
  \(c\in\C\) denote the unique constant such that \(\iprod{f+c u_0,u_0}=0\) and solve
  \((-\laplace-\lambda_0)v=f+c u_0\) for \(v\in\rho H^1_{\rm z}(\Omega)\). Then \(L_\Omega^{\rm
  aug}(u,c)=(f,w)\) for \(u=v+c'u_0\) where \(c'\coloneqq w-\iprod{v,u_0^\sharp}\).
\end{proof}

For the proof of higher regularity, we record the following (interior Schauder) estimate:

\begin{lemma}[Higher regularity]
\label{LemmaTOmegaHi}
  Let \(k\geq 1\). If \(u\in\rho H_{\rm z}^1(\Omega)\) solves \((-\laplace-\lambda)u=f\in\rho^{-1}H_{\rm z}^{k-2}(\Omega)\), then \(u\in\rho H_{\rm z}^k(\Omega)\) with
  \[
    \norm{u}_{\rho H_{\rm z}^k(\Omega)} \lesssim \norm{f}_{\rho^{-1}H_{\rm z}^{k-2}(\Omega)} + \norm{u}_{\rho H_{\rm z}^1(\Omega)}.
  \]
  The implicit constant can be taken to be uniform when \(\lambda\) varies in a fixed compact subset of \(\C\).
\end{lemma}
\begin{proof}
  Away from \(\partial\Omega\), this is standard elliptic regularity. Near \(\partial\Omega\), we note that \(\rho^2(-\laplace-\lambda_0)\), in size \(\rho_0\)-cubes centered around points at distance \(2\rho_0\) from \(\partial\Omega\), is, in the coordinates \(X=\rho/\rho_0\) and \(Y=y/\rho_0\) (with \(y\) denoting local coordinates on \(\partial\Omega\)), uniformly elliptic in unit cubes in \(X,Y\). Thus, elliptic regularity in such unit cubes gives the desired estimate.
\end{proof}

\subsubsection{Estimates in \texorpdfstring{\(\widetilde\Omega^\complement\)}{the complement of Omega}}

We next solve the Dirichlet problem for
\begin{equation}
\label{EqTModelOmegaC}
  L_{\Omega^\complement,h} \coloneqq -h^2\laplace+1
\end{equation}
in \(\Omega^\complement\):

\begin{prop}[Inversion of \(L_{\Omega^\complement,h}\)]
\label{PropTOmegaC}
  The operator
  \[
    L_{\Omega^\complement,h} \colon H^1_{0,h}(\Omega^\complement) \to H^{-1}_{0,h}(\Omega^\complement)
  \]
  is invertible, and its inverse is uniformly bounded, i.e.,
  \begin{equation}
  \label{EqTOmegaC}
    \norm{u}_{H^1_{0,h}(\Omega^\complement)} \lesssim \norm{L_{\Omega^\complement,h}u}_{H^{-1}_{0,h}(\Omega^\complement)}.
  \end{equation}
\end{prop}
\begin{proof}
  For \(u\in\CIc(\Omega^\complement)\), we have
  \[
      \iprod{L_{\Omega^\complement,h}u,u}_{L^2} = \norm{h\nabla u}^2 + \norm{u}^2 =
      \norm{u}_{H_{0,h}^1(\Omega^\complement)}^2,
  \]
  and therefore \(L_{\Omega^\complement,h}\colon H_{0,h}^1(\Omega^\complement)\to
  H_{0,h}^{-1}(\Omega^\complement)\) is invertible, with inverse having operator norm \(1\) (in
  particular, uniformly bounded in \(h\)).
\end{proof}

We record the following analogue of Lemma~\ref{LemmaTOmegaHi}:

\begin{lemma}[Higher regularity]
\label{LemmaTOmegaCHi}
  For \(k\geq 1\), the operator \(L_{\Omega^\complement,h}\colon H_h^k(\Omega^\complement)\cap H_{0,h}^1(\Omega^\complement)\to H_h^{k-2}(\Omega^\complement)\) is invertible, with uniformly bounded inverse. That is,
  \[
    \norm{u}_{H_h^k(\Omega^\complement)} \lesssim \norm{L_{\Omega^\complement,h}u}_{H_h^{k-2}(\Omega^\complement)}.
  \]
\end{lemma}
\begin{proof}
  This follows as in the standard proof of higher regularity near the boundary \(\partial\Omega^\complement\) based on finite tangential differences, now rescaled by \(h\).
\end{proof}

\subsubsection{Estimates on \texorpdfstring{\({\ff}\)}{the front face}}

The model operator we study here is
\begin{equation}
\label{EqTModelff}
  L_{{\ff},h} \coloneqq -\partial_{\hat\rho}^2 - h^2\laplace_{\partial\Omega} + H(\hat\rho),
\end{equation}
where $\laplace_{\pa\Omega}$ is as in \eqref{EqTApproxDelta}. This differs from the model operator $\hat P$ studied in~\S\ref{SsModelff} by the presence of the boundary Laplacian, which will allow us to control ($h$-rescaled) tangential derivatives here.

\begin{lemma}[Inversion of \(L_{{\ff},h}\): particular weights]
\label{LemmaTff}
  The operator
  \begin{equation}
  \label{EqTff}
    L_{{\ff},h} \colon \hat\rho_< H_{{\ff},h}^1 \to \hat\rho_<^{-1} H_{{\ff},h}^{-1}
  \end{equation}
  is invertible; its operator norm as well as that of its inverse are uniformly bounded.
\end{lemma}
\begin{proof}
  The uniform boundedness of \(L_{{\ff},h}\) follows directly from the definition of the function spaces.

  For \(u\in\CIc(\R\times\partial\Omega)\), we have
  \[
    Q \coloneqq \norm{\partial_{\hat\rho}u}_{L^2}^2 + \norm{h\nabla_{\partial\Omega}u}_{L^2}^2 + \norm{
    H(\hat\rho)u}_{L^2}^2 = \iprod{L_{{\ff},h}u,u} \leq \norm{u}_{\hat\rho_< H^1_{{\ff},h}}
    \norm{L_{{\ff},h}u}_{\hat\rho_<^{-1} H^{-1}_{{\ff},h}}.
  \]
  It thus suffices to show the estimate \(Q\gtrsim\norm{u}_{\hat\rho_< H_{{\ff},h}^1}^2\). Since \(\partial_{\hat\rho}=\hat\rho_<^{-1}\cdot\hat\rho_<\partial_{\hat\rho}\) and \(h\nabla_{\partial\Omega}=\hat\rho_<^{-1}\cdot\hat\rho_< h\nabla_{\partial\Omega}\), we only need to bound \(\norm{\hat\rho_<^{-1}u}_{L^2}^2\lesssim Q\). Now, \(Q\) directly controls \(u\) in \(L^2(\{\hat\rho\geq 0\})\), and via the control of \(\norm{\partial_{\hat\rho}u}_{L^2}^2\) also in \(L^2(\{\hat\rho\geq -2\})\), say. Let \(\chi\in\CI(\R)\) be equal to \(1\) on \((-\infty,-2]\) and \(0\) on \([-1,\infty)\). Dropping the boundary variables, note then that
  \[
      \int_{-\infty}^{-1} |\hat\rho|^{-2}|\chi(\hat\rho)u(\hat\rho)|^2\odif{x}\leq
      4\int_{-\infty}^{-1} |\hat\rho|^{-2}|\hat\rho\partial_{\hat\rho}(\chi u)|^2\odif{\hat\rho}
      \lesssim \int_{-\infty}^{-1} |\partial_{\hat\rho}u|^2\odif{\hat\rho} + \int_{-2}^{-1}
      |u|^2\odif{\hat\rho}.
  \]
  This provides the required weighted control of \(u\) in \(\hat\rho<0\).
\end{proof}

We do not state a higher regularity statement analogous to Lemmas~\ref{LemmaTOmegaHi} and
\ref{LemmaTOmegaCHi} here. The reason is that the error term \(f\) in~\eqref{EqTFin} and~\eqref{EqTEq2} arising from our quasimode construction is permitted to have a \(\delta(\rho/h)\)-singularity (with a coefficient which is a
smooth function on \(\partial\Omega\)); it thus lies in an \(H^{-1}\)-type space but not in \(H^0\). We
shall instead prove higher one-sided regularity in \(\pm\hat\rho\geq 0\) rather directly
in~\S\ref{SsTPf}; see Lemma~\ref{LemmaTPfImpr}.

It will in fact be important to have some flexibility in the weights in~\eqref{EqTff}:
\begin{prop}[Inversion of \(L_{{\ff},h}\)]
\label{PropTff}
  There exists \(\eps_0>0\) such that the following holds for all \(\alpha,\gamma\in\R\) with \(|\alpha|,|\gamma|<\eps_0\): the operator
  \[
    L_{{\ff},h} \colon \hat\rho_<^{1+\alpha}\hat\rho_>^\gamma H_{{\ff},h}^1 \to \hat\rho_<^{-1+\alpha}\hat\rho_>^\gamma H_{{\ff},h}^{-1}
  \]
  is invertible, with uniformly bounded inverse.
\end{prop}
\begin{proof}
  Set
  \[
    L_{{\ff},h}^{\alpha,\gamma} \coloneqq \hat\rho_<^{-\alpha}\hat\rho_>^{-\gamma}L_{{\ff},h}\hat\rho_<^\alpha\hat\rho_>^\gamma.
  \]
  Let
  \(C\coloneqq\sup_{h\in(0,1)}\norm{L_{{\ff},h}^{-1}}_{\cL(\hat\rho_<^{-1}H_{{\ff},h}^{-1},\hat\rho_<
  H_{{\ff},h}^1)}<\infty\). It suffices to show that there exists \(\eps_0>0\) such that
  \begin{equation}
  \label{EqTffDiff}
    \norm{L_{{\ff},h}-L_{{\ff},h}^{\alpha,\gamma}}_{\cL(\hat\rho_<H_{{\ff},h}^1,\hat\rho_<^{-1}H_{{\ff},h}^{-1})} \leq \frac{1}{2 C}
  \end{equation}
  for all \(h\in(0,1)\) whenever \(|\alpha|,|\gamma|<\eps_0\) since we can then invert
  \[
    L_{{\ff},h}^{\alpha,\gamma} = L_{{\ff},h} \bigl( I + L_{{\ff},h}^{-1}(L_{{\ff},h}^{\alpha,\gamma}-L_{{\ff},h})\bigr) \colon \hat\rho_< H_{{\ff},h}^1 \to \hat\rho_> H_{{\ff},h}^{-1}
  \]
  using a Neumann series, and the norm of the inverse is uniformly bounded by \(2 C\).

  We proceed to establish~\eqref{EqTffDiff} by proving the slightly stronger statement
  \begin{equation}
  \label{EqTffDiff2}
    \sup_{h\in(0,1)}\norm{L_{{\ff},h}-L_{{\ff},h}^{\alpha,\gamma}}_{\cL(\hat\rho_<
    H_{{\ff},h}^1,\hat\rho_<^{-1}H_{{\ff},h}^0)} \xra{\alpha,\gamma\to 0} 0.
  \end{equation}
  Since \(\laplace_{\partial\Omega}\) and \(H(\hat\rho)\) are unchanged by the conjugation by \(\hat\rho_<^{-\alpha}\hat\rho_>^{-\gamma}\), we may replace \(L_{{\ff},h}\) by \(-\partial_{\hat\rho}^2\) for the purpose of showing~\eqref{EqTffDiff2}, and we thus need to control the operator norm of
  \begin{align*}
    \hat\rho_<^{-\alpha}\hat\rho_>^{-\gamma}[\partial_{\hat\rho}^2,\hat\rho_<^\alpha\hat\rho_>^\gamma] &= 2\bigl(\alpha(\partial_{\hat\rho}\hat\rho_<)\hat\rho_<^{-1} + \gamma(\partial_{\hat\rho}\hat\rho_>)\hat\rho_>^{-1}\bigr)\partial_{\hat\rho} \\
      &\quad + \alpha(\alpha-1)(\partial_{\hat\rho}\hat\rho_<)^2\hat\rho_<^{-2} + \gamma(\gamma-1)(\partial_{\hat\rho}\hat\rho_>)^2\hat\rho_>^{-2} + 2\alpha\gamma(\partial_{\hat\rho}\hat\rho_<)(\partial_{\hat\rho}\hat\rho_>)\hat\rho_<^{-1}\hat\rho_>^{-1} \\
      &\quad + \alpha(\partial_{\hat\rho}^2\hat\rho_<)\hat\rho_<^{-1} + \gamma(\partial_{\hat\rho}^2\hat\rho_>)\hat\rho_>^{-1}.
  \end{align*}
  Now, \(\partial_{\hat\rho}\colon\hat\rho_< H_{{\ff},h}^1\to H_{{\ff},h}^0\) is uniformly bounded.
  Moreover, \(\partial_{\hat\rho}\hat\rho_<\) and \(\partial_{\hat\rho}\hat\rho_<\) are bounded, and
  \(\partial_{\hat\rho}\hat\rho_>=0\) for \(\hat\rho<-1\); therefore
  \((\partial_{\hat\rho}\hat\rho_<)\hat\rho_<^{-1}\) and
  \((\partial_{\hat\rho}\hat\rho_>)\hat\rho_>^{-1}\) are uniformly bounded maps
  \(H_{{\ff},h}^0\to\hat\rho_<^{-1}H_{{\ff},h}^0\). Therefore, the first line on the right is
  uniformly bounded by \(|\alpha|+|\gamma|\). Arguing similarly for the second and third lines
  (using also that \(\partial_{\hat\rho}^2\hat\rho_<\) and \(\partial_{\hat\rho}^2\hat\rho_>\) are
  compactly supported) gives the uniform (in \(h\)) bound
  \[
    \norm{L_{{\ff},h}-L_{{\ff},h}^{\alpha,\gamma}}_{\cL(\hat\rho_< H_{{\ff},h}^1,\hat\rho_<^{-1} H_{{\ff},h}^0)} \lesssim |\alpha|+|\gamma|,
  \]
  proving~\eqref{EqTffDiff2}.
\end{proof}

With more effort, one can show that the conclusion of Proposition~\ref{PropTff} remains valid for all \(\alpha\in(-\frac12,\frac12)\) and \(\gamma\in\R\), with the bound on the operator norm on the inverse being uniform in \(h\) (for fixed \(\alpha,\gamma\)).

\subsubsection{Combination}

We now return to the study of the operator
\[
    P_h^\aug = \begin{pmatrix} P_h-\lambda_h & -u_h \\ \iprod{\cdot,u_0^\sharp}_{L^2} & 0 \end{pmatrix}
\]
introduced in~\eqref{EqTFin}, where we now make the \(h\)-dependence explicit in the notation. We shall prove:

\begin{prop}[Inversion of \(P^\aug\)]
\label{PropTFull}
  There exists \(h_0>0\) such that for all \(h\in(0,h_0]\), the operator \(P_h^\aug\colon H^1(\R^d)\oplus\C\to H^{-1}(\R^d)\oplus\C\) is invertible; moreover, regarded as a map
  \begin{equation}
  \label{EqTFullMap}
    P_h^\aug \colon \rho_{\ff}\rho_{\Omega^\complement}X^1_h \to \rho_{\ff}^{-1}\rho_{\Omega^\complement}^{-1}X^{-1}_h,
  \end{equation}
  its inverse is uniformly bounded for \(h\in(0,h_0]\).
\end{prop}

We first investigate how well the model operators \(L_\Omega^\aug\) (from~\eqref{EqTModelOmega}), resp.\ \(L_{\Omega^\complement,h}\) (from~\eqref{EqTModelOmegaC}) and \(L_{{\ff},h}\) (from~\eqref{EqTModelff}) approximate \(P^\aug_h\), resp.\ (rescaled versions of) \(P_h-\lambda_h\), respectively.

\begin{lemma}[Approximations]
\label{LemmaTApprox}
  Recall the cutoff functions \(\chi_\Omega,\chi_{\ff},\chi_{\Omega^\complement}\) from~\eqref{EqTCutoffs}. Let \(\tilde\chi_0\in\CI([0,\infty))\) be equal to \(1\) on \([0,\frac34]\) and \(0\) on \([1,\infty)\), and set \(\tilde\chi_\Omega=\tilde\chi_0(\rho_\Omega)\) (so $\tilde\chi_\Omega=1$ on $\supp\chi_\Omega$). Let \(\alpha,\gamma\in\R\). Then
  \begin{align}
  \label{EqTApproxOmega}
    \norm*{\begin{pmatrix} \tilde\chi_\Omega & 0 \\ 0 & 1 \end{pmatrix}(P_h^\aug-L_\Omega^\aug)\begin{pmatrix} \chi_\Omega u \\ c \end{pmatrix}}_{\rho^{-1}H_{\rm z}^{-1}(\Omega)\oplus\C} &\lesssim h(\norm{\chi_\Omega u}_{\rho^{-1}H_{\rm z}^{-1}(\Omega)}+|c|), \\
  \label{EqTApproxOmegaC}
    \norm{\bigl(h^2(P_h-\lambda_h)-L_{\Omega^\complement,h}\bigr)(\chi_{\Omega^\complement}u)}_{H_{0,h}^{-1}(\Omega^\complement)}
    &\lesssim h^2\norm{\chi_{\Omega^\complement}u}_{H_{0,h}^{-1}(\Omega^\complement)}, \\
  \label{EqTApproxff}
    \norm*{\bigl(h^2(P_h-\lambda_h) - L_{{\ff},h}\bigr)(\chi_{\ff}
    u)}_{\hat\rho_<^{-1-\alpha}\hat\rho_<^{-\gamma}H_{{\ff},h}^{-1}} &\lesssim
    \norm{\rho_{\ff}\chi_{\ff} u}_{\hat\rho_<^{1-\alpha}\hat\rho_<^{-\gamma}H_{{\ff},h}^1}
  \end{align}
  for all \(u\in H^1(\R^d)\).
\end{lemma}
\begin{proof}
  We have
  \[
    \begin{pmatrix} \tilde\chi_\Omega & 0 \\ 0 & 1 \end{pmatrix}(P_h^\aug-L_\Omega^\aug)\begin{pmatrix} \chi_\Omega u \\ c \end{pmatrix} = \begin{pmatrix} -(\lambda_h-\lambda_0)\chi_\Omega u & -\tilde\chi_\Omega (u_h-u_0)c \\ 0 & 0 \end{pmatrix}.
  \]
  Note then that \(|\lambda_h-\lambda_0|\lesssim h\) and \(\tilde\chi_\Omega(u_h-u_0)\in\rho_\Omega\CI(M)\). On \(\supp\tilde\chi_\Omega\), the function \(\rho\) is a defining function of \({\ff}\), and hence \(\rho\rho_\Omega\sim h\) there; therefore, \(\norm{\tilde\chi_\Omega(u_h-u_0)}_{\rho^{-1}H_{\rm z}^{-1}(\Omega)}\lesssim\norm{\rho\tilde\chi_\Omega(u_h-u_0)}_{H_{\rm z}^0(\Omega)}\lesssim h\). This implies~\eqref{EqTApproxOmega}.

  The estimate~\eqref{EqTApproxOmegaC} follows from \((h^2(P_h-\lambda_h)-L_{\Omega^\complement,h})\chi_{\Omega^\complement}=-h^2\lambda_h\chi_{\Omega^\complement}\) and \(|h^2\lambda_h|\lesssim h^2\).

  To prove~\eqref{EqTApproxff}, we recall the splitting of the Laplacian from \eqref{EqTApproxDelta}. Passing to \(\hat\rho=\frac{\rho}{h}\), so \(h\partial_\rho=\partial_{\hat\rho}\), we thus obtain, on $\supp\chi_\ff$,
  \begin{align*}
    h^2(-\laplace+h^{-2}\chfun_{\Omega^\complement}) - L_{{\ff},h} &= -h^2\laplace + (\partial_{\hat\rho}^2 + h^2\laplace_{\partial\Omega}) \\
      &= -h a(\rho,y)\partial_{\hat\rho} - \rho\tilde g^{j k}(\rho,y)h\partial_{y^j}h\partial_{y^k} - h\rho\tilde b^l(\rho,y)h\partial_{y^l}
  \end{align*}
  for smooth \(\tilde g^{j k},\tilde b^l\). Consider the first term, which we rewrite as \(-h\hat\rho_< a(\rho,y)\cdot\hat\rho_<^{-1}\partial_{\hat\rho}\): since
  \[
    \hat\rho_<^{-1}\partial_{\hat\rho} \colon \hat\rho_<^{1-\alpha}\hat\rho_>^{-\gamma}H_{{\ff},h}^1 \to \hat\rho_<^{-1-\alpha}\hat\rho_>^{-\gamma}H_{{\ff},h}^0
  \]
  is uniformly bounded and multiplication by smooth functions of \((\rho,y)\) is uniformly bounded on every weighted \(H_{{\ff},h}^k\) space, we obtain
  \[
    \norm{ h a \partial_{\hat\rho} (\chi_{\ff} u) }_{\hat\rho_<^{-1-\alpha}\hat\rho_>^{-\gamma}H_{{\ff},h}^0} \lesssim \norm{h\hat\rho_<\chi_{\ff} u}_{\hat\rho_<^{1-\alpha}\hat\rho_>^{-\gamma}H_{{\ff},h}^1} \lesssim \norm{\rho_{\ff}\chi_{\ff} u}_{\hat\rho_<^{1-\alpha}\hat\rho_>^{-\gamma}H_{{\ff},h}^1};
  \]
  we use here that \(h\hat\rho_<\), as a function on \(M\), is a smooth multiple of \(\rho_{\ff}\) (in fact, of \(\rho_{\ff}\rho_{\Omega^\complement}\)). We write the second term schematically as a smooth function times \(\rho h^2\partial_y^2=\rho\cdot\hat\rho_<^{-2}\cdot\hat\rho_<^2(h\partial_y)^2\); the uniform boundedness of \(\hat\rho_<^{-2}\cdot\hat\rho_<^2(h\partial_y)^2\colon\hat\rho_<^{1-\alpha}\hat\rho_>^{-\gamma}H_{{\ff},h}^1\to\hat\rho_<^{-1-\alpha}\hat\rho_>^{-\gamma}H_{{\ff},h}^{-1}\) and the fact that \(\rho\) is a smooth multiple of \(\rho_{\ff}\) on \(\supp\chi_{\ff}\) give the bound
  \[
    \norm{\rho\tilde g^{j k}h\partial_{y^j}h\partial_{y^k}u}_{\hat\rho_<^{-1-\alpha}\hat\rho_<^{-\gamma}H_{{\ff},h}^{-1}} \lesssim \norm{\rho_{\ff}\chi_{\ff} u}_{\hat\rho_<^{1-\alpha}\hat\rho_>^{-\gamma}H_{{\ff},h}^1}.
  \]
  The third term obeys an analogous estimate since it schematically equals \(h\hat\rho_<\rho\cdot\hat\rho_<^{-2}\,\hat\rho_< h\partial_y\), with \(h\hat\rho_<\rho\) a smooth multiple of \(\rho_{\ff}\).
\end{proof}

We can now give:

\begin{proof}[Proof of Proposition~\usref{PropTFull}]
  Let \(\eps\coloneqq\eps_0/2\) in the notation of Proposition~\ref{PropTff}. We begin by estimating
  \begin{equation}
  \label{EqTFull0}
  \begin{split}
    \norm{u}_{\rho_{\ff}\rho_{\Omega^\complement}X_h^1} &\lesssim \norm{\chi_\Omega u}_{\rho_{\ff}\rho_{\Omega^\complement}X_h^1} + \norm{(1-\chi_\Omega)u}_{\rho_{\ff}\rho_{\Omega^\complement}X_h^1} \\
      &\lesssim \norm{\chi_\Omega u}_{\rho H_{\rm z}^1(\Omega)} + \norm{(1-\chi_\Omega)u}_{\rho_\Omega^{-\eps}\rho_{\ff}\rho_{\Omega^\complement}X_h^1} \\
      &\lesssim \norm{\chi_\Omega u}_{\rho H_{\rm z}^1(\Omega)} + \norm{u}_{\rho_\Omega^{-\eps}\rho_{\ff}\rho_{\Omega^\complement}X_h^1};
  \end{split}
  \end{equation}
  in the passage to the second line we used~\eqref{EqTNormEqOmega} for the first term and the fact that \(\rho_\Omega\) is bounded away from \(0\) on \(\supp(1-\chi_\Omega)\) for the second term (which allows us to insert any fixed \(\rho_\Omega\)-weight, and we choose $\rho_\Omega^{-\eps}$ here).

  \pfstep{Inversion of the \(\widetilde\Omega\)-model.} We apply Proposition~\ref{PropTOmega} to treat the first term, plus $|c|$; using also Lemma~\ref{LemmaTApprox}, we get
  \begin{align}
  \label{EqTFullOmega}
    \norm{\chi_\Omega u}_{\rho H_{\rm z}^1(\Omega)} + |c| &\lesssim \norm{L_\Omega^\aug(\chi_\Omega u,c)}_{\rho^{-1}H_{\rm z}^{-1}(\Omega)\oplus\C} \\
  \label{EqTFullOmega1}
      &\leq \norm*{\begin{pmatrix} \tilde\chi_\Omega & 0 \\ 0 & 1 \end{pmatrix} P_h^\aug\begin{pmatrix} \chi_\Omega u \\ c \end{pmatrix}}_{\rho^{-1}H_{\rm z}^{-1}(\Omega)\oplus\C} \\
  \label{EqTFullOmega2}
      &\qquad + \norm*{\begin{pmatrix} \tilde\chi_\Omega & 0 \\ 0 & 1 \end{pmatrix}(P_h^\aug-L_\Omega^\aug)\begin{pmatrix}\chi_\Omega u \\ c \end{pmatrix}}_{\rho^{-1}H^{-1}_{\rm z}(\Omega)\oplus\C} \\
  \label{EqTFullOmega3}
      &\qquad + \norm*{\begin{pmatrix} 1-\tilde\chi_\Omega & 0 \\ 0 & 0 \end{pmatrix} L_\Omega^\aug\begin{pmatrix} \chi_\Omega u \\ c \end{pmatrix}}_{\rho^{-1}H_{\rm z}^{-1}(\Omega)\oplus\C}.
  \end{align}
  The first term on the right is
  \begin{align*}
    &\lesssim
    \norm*{\begin{pmatrix}\tilde\chi_\Omega&0\\0&1\end{pmatrix}P_h^\aug\begin{pmatrix}u\\c\end{pmatrix}}_{\rho^{-1}H_{\rm
    z}^{-1}(\Omega)\oplus\C} + \norm*{ \begin{pmatrix}\tilde\chi_\Omega&0\\0&1\end{pmatrix}\left[P_h^\aug,\begin{pmatrix} \chi_\Omega & 0 \\ 0 & 1 \end{pmatrix}\right]\begin{pmatrix}u\\c\end{pmatrix}}_{\rho^{-1}H_{\rm z}^{-1}(\Omega)\oplus\C} \\
    &\lesssim \norm*{P^\aug_h(u,c)}_{\rho_{\ff}^{-1}\rho_{\Omega^\complement}^{-1}X_h^{-1}} + \norm*{ \begin{pmatrix} -[\laplace,\chi_\Omega] & -\tilde\chi_\Omega(1-\chi_\Omega)u_h \\ 0 & 0 \end{pmatrix}\begin{pmatrix} u \\ c \end{pmatrix}}_{\rho^{-1}H_{\rm z}^{-1}(\Omega)\oplus\C};
  \end{align*}
  here we use that \(\chi_\Omega u_0^\sharp=u_0^\sharp\) for all sufficiently small \(h\). (This is where the compact support of \(u_0^\sharp\) in \(\Omega\) is useful.) Now, \([\laplace,\chi_\Omega]\) is uniformly bounded from \(\rho H_{\rm z}^1(\Omega)\) (or even \(\rho H_{\rm z}^0(\Omega)\)) to \(\rho^{-1} H_{\rm z}^{-1}(\Omega)\). Furthermore, \(\tilde\chi_\Omega(1-\chi_\Omega)\) vanishes near \(\widetilde\Omega\), thus \(\rho\tilde\chi_\Omega(1-\chi_\Omega)\) is a smooth multiple of \(h\), and hence
  \begin{equation}
  \label{EqTFullOmega1Pf}
    \norm{c\tilde\chi_\Omega(1-\chi_\Omega)u_h}_{\rho^{-1}H_{\rm z}^{-1}(\Omega)} \sim |c| \norm{ \rho\tilde\chi_\Omega(1-\chi_\Omega)u_h}_{H_{\rm z}^{-1}(\Omega)} \lesssim h|c|.
  \end{equation}
  Altogether, we have then shown that
  \[
    \eqref{EqTFullOmega1} \lesssim \norm{P^\aug_h(u,c)}_{\rho_{\ff}^{-1}\rho_{\Omega^\complement}^{-1}X_h^{-1}} + \norm{u}_{\rho_\Omega^{-\eps}\rho_{\ff}\rho_{\Omega^\complement}X_h^1} + h|c|;
  \]
  the second term here arises from \([\laplace,\chi_\Omega]\) whose coefficients have supports disjoint from \(\widetilde\Omega\) and \(\widetilde\Omega^\complement\), which allows us to insert arbitrary \(\rho_\Omega\)- and \(\rho_{\Omega^\complement}\)-weights. (We choose \(-\eps\) and \(1\) here for consistency with the right-hand side of~\eqref{EqTFull0}.)

  Next, we can apply~\eqref{EqTApproxOmega} to bound (a fortiori)
  \[
    \eqref{EqTFullOmega2} \lesssim h\norm{u}_{\rho_{\ff}\rho_{\Omega^\complement}X_h^1} + h|c|.
  \]

  Finally, to estimate~\eqref{EqTFullOmega3}, we note that
  \[
    \begin{pmatrix} 1-\tilde\chi_\Omega & 0 \\ 0 & 0 \end{pmatrix} L_\Omega^\aug \begin{pmatrix} \chi_\Omega & 0 \\ 0 & 1 \end{pmatrix} = \begin{pmatrix} 0 & -(1-\tilde\chi_\Omega)u_0 \\ 0 & 0 \end{pmatrix}.
  \]
  Similarly to~\eqref{EqTFullOmega1Pf}, the action of this on \((u,c)\) is bounded in \(\rho^{-1}H_{\rm z}^{-1}(\Omega)\oplus\C\) by \(h|c|\). Plugging the estimates thus obtained into~\eqref{EqTFull0} and absorbing the terms of size \(\cO(h)\) into the left-hand side, we have now proved
  \begin{equation}
  \label{EqTFull1}
    \norm{u}_{\rho_{\ff}\rho_{\Omega^\complement}X_h^1} + |c| \lesssim \norm{P_h^\aug(u,c)}_{\rho_{\ff}^{-1}\rho_{\Omega^\complement}^{-1}X_h^{-1}\oplus\C} + \norm{u}_{\rho_\Omega^{-\eps}\rho_{\ff}\rho_{\Omega^\complement}X_h^1}.
  \end{equation}

  \pfstep{Inversion of the \(\widetilde\Omega^\complement\)-model.} We now localize the second term on the right in~\eqref{EqTFull1} near \(\widetilde\Omega^\complement\), so using Lemma~\ref{LemmaTNormEq}
  \begin{align*}
    \norm{u}_{\rho_\Omega^{-\eps}\rho_{\ff}\rho_{\Omega^\complement}X_h^1} &\leq \norm{\chi_{\Omega^\complement}u}_{\rho_{\ff}\rho_{\Omega^\complement}X_h^1} + \norm{(1-\chi_{\Omega^\complement})u}_{\rho_\Omega^{-\eps}\rho_{\ff}\rho_{\Omega^\complement}X_h^1} \\
      &\lesssim h^{-1}\norm{\chi_{\Omega^\complement}u}_{H_{0,h}^1(\Omega^\complement)} + \norm{u}_{\rho_\Omega^{-\eps}\rho_{\ff}\rho_{\Omega^\complement}^{1-\eps}X_h^1}.
  \end{align*}
  We use here that \(\rho_\Omega\gtrsim 1\) on \(\supp\chi_{\Omega^\complement}\) and \(\rho_{\Omega^\complement}\gtrsim 1\) on \(\supp(1-\chi_{\Omega^\complement})\), as well as the vanishing of $\chi_{\Omega^\complement}u$ near $\pa\Omega$ for every $h>0$. Using Proposition~\ref{PropTOmegaC}, we estimate
  \begin{align*}
    h^{-1}\norm{\chi_{\Omega^\complement}u}_{H_{0,h}^1(\Omega^\complement)} &\lesssim h^{-1}\norm{L_{\Omega^\complement,h}(\chi_{\Omega^\complement}u)}_{H_{0,h}^{-1}(\Omega^\complement)} \\
      &\lesssim h\norm{\chi_{\Omega^\complement}(P_h-\lambda_h)u}_{H_{0,h}^{-1}(\Omega^\complement)} + h^{-1}\norm{[h^2(P_h-\lambda_h),\chi_{\Omega^\complement}]u}_{H_{0,h}^{-1}(\Omega^\complement)} \\
      &\qquad + h^{-1}\norm{(h^2(P_h-\lambda_h)-L_{\Omega^\complement,h})(\chi_{\Omega^\complement}u)}_{H_{0,h}^{-1}(\Omega^\complement)}.
  \end{align*}
  The first term is \(\lesssim\norm{\chi_{\Omega^\complement}(P_h-\lambda_h)u}_{\rho_{\ff}^{-1}\rho_{\Omega^\complement}^{-1}X_h^{-1}}\). In the second term, note that \([h^2(P_h-\lambda_h),\chi_{\Omega^\complement}]\) is a sum of terms of the form \(a h\partial\) and \(b\) where \(a,b\) are bounded (together with all derivatives along \(\rho_{\ff}\partial\)) and vanish near \(\widetilde\Omega,\widetilde\Omega^\complement\); thus this term is \(\lesssim\norm{u}_{\rho_\Omega^{-\eps}\rho_{\ff}\rho_{\Omega^\complement}^{1-\eps}X_h^1}\) (a fortiori). The third term can be bounded using~\eqref{EqTApproxOmegaC} and is thus \(\lesssim h\norm{\chi_{\Omega^\complement}u}_{H_{0,h}^{-1}(\Omega^\complement)}\lesssim h^2\norm{u}_{\rho_\Omega^{-\eps}\rho_{\ff}\rho_{\Omega^\complement}X_h^1}\).

  Plugging these estimates into~\eqref{EqTFull1}, we have now proved
  \[
    \norm{u}_{\rho_{\ff}\rho_{\Omega^\complement}X_h^1} + |c| \lesssim \norm{P_h^\aug(u,c)}_{\rho_{\ff}^{-1}\rho_{\Omega^\complement}^{-1}X_h^{-1}\oplus\C} + \norm{\chi_{\Omega^\complement}(P_h-\lambda_h)u}_{\rho_{\ff}^{-1}\rho_{\Omega^\complement}^{-1}X_h^{-1}} + \norm{u}_{\rho_\Omega^{-\eps}\rho_{\ff}\rho_{\Omega^\complement}^{1-\eps}X_h^1}.
  \]
  It remains to relate the second term here to \(P_h^\aug(u,c)\); to this end, note that
  \[
    \begin{pmatrix} \chi_{\Omega^\complement}(P_h-\lambda_h)u \\ 0 \end{pmatrix} - \begin{pmatrix} \chi_{\Omega^\complement} & 0 \\ 0 & 0 \end{pmatrix} P_h^\aug \begin{pmatrix} u \\ c \end{pmatrix} = \begin{pmatrix} c\chi_{\Omega^\complement}u_h \\ 0 \end{pmatrix}.
  \]
  Since \(\chi_{\Omega^\complement}u_h\) is smooth on \(M\), vanishes near \(\widetilde\Omega\),
  vanishes (to infinite order) at \(\widetilde\Omega^\complement\), and has uniformly compact
  support, its \(L^2(\R^d)\)-norm is \(\lesssim h^{\frac12}\),\footnote{The relevant computation is $\int_{-1}^1 \la\rho/h\ra^{-N}\,\dd\rho\leq h\int_{-\infty}^\infty \la\hat\rho\ra^{-N}\,\dd\hat\rho\lesssim h$ for $N>1$.} and thus a fortiori its \(\rho_{\ff}^{-1}\rho_{\Omega^\complement}^{-1}X_h^0\)-norm (and thus also its \(\rho_{\ff}^{-1}\rho_{\Omega^\complement}^{-1}X_h^{-1}\)-norm) has the same bound. Absorbing the resulting \(h^{\frac12}|c|\) term, we thus altogether get
  \begin{equation}
  \label{EqTFull2}
    \norm{u}_{\rho_{\ff}\rho_{\Omega^\complement}X_h^1} + |c| \lesssim \norm{P_h^\aug(u,c)}_{\rho_{\ff}^{-1}\rho_{\Omega^\complement}^{-1}X_h^{-1}\oplus\C} + \norm{u}_{\rho_\Omega^{-\eps}\rho_{\ff}\rho_{\Omega^\complement}^{1-\eps}X_h^1}.
  \end{equation}

  \pfstep{Inversion of the \({\ff}\)-model.} We finally localize the error term in~\eqref{EqTFull2} near \({\ff}\), so applying~\eqref{EqTNormEqff} with \(\alpha=-\eps\), \(\beta=1\), and \(\gamma=1-\eps\), we start with
  \[
    \norm{u}_{\rho_\Omega^{-\eps}\rho_{\ff}\rho_{\Omega^\complement}^{1-\eps}X_h^1} \lesssim h^{-\frac12}\norm{\chi_{\ff} u}_{\hat\rho_<^{1+\eps}\hat\rho_>^\eps H_{{\ff},h}^1} + \norm{u}_{\rho_\Omega^{-\eps}\rho_{\ff}^{1-\eps}\rho_{\Omega^\complement}^{1-\eps}X_h^1}.
  \]
  We bound the first term using Proposition~\ref{PropTff} by
  \begin{align*}
    h^{-\frac12}\norm{\chi_{\ff} u}_{\hat\rho_<^{1+\eps}\hat\rho_>^\eps H_{{\ff},h}^1} &\lesssim h^{-\frac12}\norm{L_{{\ff},h}(\chi_{\ff} u)}_{\hat\rho_<^{-1+\eps}\hat\rho_>^\eps H_{{\ff},h}^{-1}} \\
      &\lesssim h^{\frac32}\norm{\chi_{\ff}(P_h-\lambda_h)u}_{\hat\rho_<^{-1+\eps}\hat\rho_>^\eps H_{{\ff},h}^{-1}} + h^{\frac32}\norm{ [P_h-\lambda_h,\chi_{\ff}]u}_{\hat\rho_<^{-1+\eps}\hat\rho_>^\eps H_{{\ff},h}^{-1}} \\
      &\qquad + h^{-\frac12} \norm{ (h^2(P_h-\lambda_h)-L_{{\ff},h})(\chi_{\ff} u)}_{\hat\rho_<^{-1+\eps}\hat\rho_>^\eps H_{{\ff},h}^{-1}}.
  \end{align*}
  Using~\eqref{EqTNormEqff}, the first term is \(\lesssim\norm{\chi_{\ff}(P_h-\lambda_h)u}_{\rho_\Omega^{-\eps}\rho_{\ff}^{-1}\rho_{\Omega^\complement}^{-1-\eps}X_h^{-1}}\). Similarly, the second term is
  \[
    \lesssim \norm{[P_h-\lambda_h,\chi_{\ff}]u}_{\rho_\Omega^{-\eps}\rho_{\ff}^{-1}\rho_{\Omega^\complement}^{-1-\eps}X_h^{-1}};
  \]
  since the coefficients of \([P_h-\lambda_h,\chi_{\ff}]\) are supported in \(|\rho|\gtrsim 1\) (so away from $\ff$), smooth in \(x\), and uniformly bounded in \(h\), this is bounded (a fortiori) by \(\norm{u}_{\rho_\Omega^{-\eps}\rho_{\ff}^{1-\eps}\rho_{\Omega^\complement}^{1-\eps}X_h^1}\). The third term, finally, can be estimated using~\eqref{EqTApproxff} by
  \[
    h^{-\frac12}\norm{\rho_{\ff}\chi_{\ff} u}_{\hat\rho_<^{1+\eps}\hat\rho_>^\eps H_{{\ff},h}^1} \sim \norm{\rho_{\ff}\chi_{\ff} u}_{\rho_\Omega^{-\eps}\rho_{\ff}\rho_{\Omega^\complement}^{1-\eps}X_h^1} \lesssim \norm{u}_{\rho_\Omega^{-\eps}\rho_{\ff}^{1-\eps}\rho_{\Omega^\complement}^{1-\eps}X_h^1}.
  \]
  (The \(\rho_{\ff}\)-weight on the right can be taken to be \(0\) even.) Plugging these estimates into~\eqref{EqTFull2} yields
  \begin{equation}
  \label{EqTFull3}
  \begin{split}
    \norm{u}_{\rho_{\ff}\rho_{\Omega^\complement}X_h^1} + |c| &\lesssim \norm{P_h^\aug(u,c)}_{\rho_{\ff}^{-1}\rho_{\Omega^\complement}^{-1}X_h^{-1}\oplus\C} \\
      &\qquad + \norm{\chi_{\ff}(P_h-\lambda_h)u}_{\rho_\Omega^{-\eps}\rho_{\ff}^{-1}\rho_{\Omega^\complement}^{-1-\eps}X_h^{-1}}  + \norm{u}_{\rho_\Omega^{-\eps}\rho_{\ff}^{1-\eps}\rho_{\Omega^\complement}^{1-\eps}X_h^1}.
  \end{split}
  \end{equation}

  Similarly to the previous step, we need to relate the second term on the right to \(P_h^\aug(u,c)\). Note then that the error term caused by this replacement involves the operator
  \[
    \begin{pmatrix} \chi_{\ff}(P_h-\lambda_h)u \\ 0 \end{pmatrix} - \begin{pmatrix} \chi_{\ff} & 0 \\ 0 & 0 \end{pmatrix} P_h^\aug \begin{pmatrix} u \\ c \end{pmatrix} = \begin{pmatrix} c\chi_{\ff} u_h \\ 0 \end{pmatrix};
  \]
  since \(\rho_\Omega\rho_{\ff}\rho_{\Omega^\complement}\) is a smooth multiple of \(h\), its norm is
  \[
    \norm{c\chi_{\ff} u_h}_{\rho_\Omega^{-\eps}\rho_{\ff}^{-1}\rho_{\Omega^\complement}^{-1-\eps}X_h^{-1}} \lesssim h^\eps|c| \norm{u_h}_{L^2} \lesssim h^\eps|c|.
  \]
  Absorbing this into the left-hand side of~\eqref{EqTFull3} gives, for all sufficiently small \(h>0\),
  \begin{equation}
  \label{EqTFull4}
    \norm{u}_{\rho_{\ff}\rho_{\Omega^\complement}X_h^1} + |c| \lesssim \norm{P_h^\aug(u,c)}_{\rho_{\ff}^{-1}\rho_{\Omega^\complement}^{-1}X_h^{-1}\oplus\C} + \norm{u}_{\rho_\Omega^{-\eps}\rho_{\ff}^{1-\eps}\rho_{\Omega^\complement}^{1-\eps}X_h^1}.
  \end{equation}

  \pfstep{Conclusion.} The error term in~\eqref{EqTFull4} is \(\lesssim h^\eps\norm{u}_{\rho_{\ff}\rho_{\Omega^\complement}X_h^1}\) and can thus be absorbed into the left-hand side for small \(h>0\); this establishes the existence of \(h_0>0\) such that we have the uniform estimate
  \begin{equation}
  \label{EqTFullEst}
    \norm{(u,c)}_{\rho_{\ff}\rho_{\Omega^\complement}X_h^1\oplus\C} \lesssim \norm{P_h^\aug(u,c)}_{\rho_{\ff}^{-1}\rho_{\Omega^\complement}^{-1}X_h^{-1}\oplus\C},\quad 0<h\leq h_0.
  \end{equation}
  This proves the injectivity (with uniform estimates) of the map~\eqref{EqTFullMap}.

  It remains to show its surjectivity; we sketch the proof. To start, fix a partition of unity \(\psi_\Omega+\psi_{\ff}+\psi_{\Omega^\complement}=1\) such that \(\supp\psi_\bullet\subset\{\chi_\bullet=1\}\) for \(\bullet=\Omega,{\ff},\Omega^\complement\). Given \(f\in H^{-1}(\R^d)\) and \(w\in\C\), set then
  \[
    (u_\approx,c_\approx) \coloneqq \begin{pmatrix} \chi_\Omega & 0 \\ 0 & 1 \end{pmatrix}
    (L_\Omega^\aug)^{-1}\begin{pmatrix} \psi_\Omega f \\ w \end{pmatrix} + \begin{pmatrix}
    \chi_{\ff} h^2 L_{{\ff},h}^{-1}(\psi_{\ff} f) \\ 0 \end{pmatrix} + \begin{pmatrix}
\chi_{\Omega^\complement} h^2 L_{\Omega^\complement,h}^{-1}(\psi_{\Omega^\complement}f) \\ 0
\end{pmatrix}.
  \]
  That is, we attempt to invert \(P_h^\aug\) by inverting each of its model operators in the respective regimes of validity. The resulting map \(Q_h\colon(f,w)\mapsto(u_\approx,c_\approx)\) is uniformly bounded from \(\rho_{\ff}^{-1}\rho_{\Omega^\complement}^{-1}X_h^{-1}\) to \(\rho_{\ff}\rho_{\Omega^\complement}X_h^1\), and \(P_h^\aug Q_h=I+R_h\) where the operator norm of \(R_h\) on \(\rho_{\ff}^{-1}\rho_{\Omega^\complement}^{-1}X_h^{-1}\) is \(\lesssim h^\eps\) (as follows from estimates completely analogous to those that led to~\eqref{EqTFull4}), and thus \(\leq\frac12\) for sufficiently small \(h\). This gives the right inverse \(Q_h(I+R_h)^{-1}\) for \(P_h^\aug\).

  (A more elegant approach to surjectivity would be to work not with \(P_h^\aug\) but with the formally self-adjoint operator
  \[
      \tilde P_h^\aug \coloneqq \begin{pmatrix} P_h-\lambda_h & -u_h \\ -\iprod{\cdot,u_h}_{L^2} & 0 \end{pmatrix}.
  \]
  For this operator we also have the estimate~\eqref{EqTFullEst} by essentially the same proof. The only difference is that the fact that \(\chi_\Omega u_h\neq u_h\) causes additional error terms which however vanish as \(h\searrow 0\). The surjectivity then follows from \(\tilde P_h^\aug=(\tilde P_h^\aug)^*\). We leave the details to the interested reader.)
\end{proof}

\subsection{Proof of Theorem~\ref{ThmT}}
\label{SsTPf}

Fix \(h_0>0\) from Proposition~\ref{PropTFull}. We now return to the problem of solving~\eqref{EqTFin}.

\pfstep{Step~1. Solution in low regularity spaces.} The key input is:
\begin{equation}
\label{EqfBound}
  \norm{f_h}_{\rho_{\ff}^{-1}\rho_{\Omega^\complement}^{-1}X_h^{-1}} \leq C_N h^N\quad\forall\,N.
\end{equation}
This follows easily by cutting \(f_h\) into three pieces with support near \(\widetilde\Omega\), \(\widetilde\Omega^\complement\), and \(\hat\rho=0\), and bounding the norms in the three regimes using Lemma~\ref{LemmaTNormEq}.

\begin{lemma}[Contraction]
\label{LemmaTPfContr}
  For \(h\in(0,h_0]\), define
  \[
    S_h \colon \begin{pmatrix} v \\ \mu \end{pmatrix} \mapsto (P_h^\aug)^{-1}\begin{pmatrix} -f_h + \mu v \\ 0 \end{pmatrix}
  \]
  Then there exists \(h_1\in(0,h_0]\) such that for all \(h\in(0,h_1]\), the map
  \[
    S_h \colon \rho_{\ff}\rho_{\Omega^\complement}X_h^1\oplus\C \to \rho_{\ff}\rho_{\Omega^\complement}X_h^1\oplus\C
  \]
  maps the \(h\)-ball into itself and is a contraction.
\end{lemma}
\begin{proof}
  We have
  \[
    \norm{S_h(v,\mu)}_{\rho_{\ff}\rho_{\Omega^\complement}X_h^1\oplus\C} \leq C\norm{-f_h+\mu v}_{\rho_{\ff}^{-1}\rho_{\Omega^\complement}^{-1}X_h^{-1}} \leq C\norm{f_h}_{\rho_{\ff}^{-1}\rho_{\Omega^\complement}^{-1}X_h^{-1}} + C'|\mu|\norm{v}_{\rho_{\ff}\rho_{\Omega^\complement}X_h^1}.
  \]
  For first term on the right, we use the bound by $C_2 h^2$ from~\eqref{EqfBound}. When \(|\mu|+\norm{v}_{\rho_{\ff}\rho_{\Omega^\complement}X_h^1}\leq h\), we thus obtain
  \[
    \norm{S_h(v,\mu)}_{\rho_{\ff}\rho_{\Omega^\complement}X_h^1\oplus\C} \leq C C_2 h^2 + C'h^2 \leq h
  \]
  for all \(h\leq h_2\coloneqq(C C_2+C')^{-1}\). Furthermore,
  \begin{align*}
    \norm{S_h(v,\mu)-S_h(v',\mu')}_{\rho_{\ff}\rho_{\Omega^\complement}X_h^1\oplus\C} &\leq C\norm{\mu v-\mu'v'}_{\rho_{\ff}^{-1}\rho_{\Omega^\complement}^{-1}X_h^{-1}} \\
      &\leq C''\bigl( |\mu| \norm{ v-v' }_{\rho_{\ff}\rho_{\Omega^\complement}X_h^1} + |\mu-\mu'| \norm{v}_{\rho_{\ff}\rho_{\Omega^\complement}X_h^1}\bigr) \\
      &\leq 2 C''h \norm{(v,\mu)-(v',\mu')}_{\rho_{\ff}\rho_{\Omega^\complement}X_h^1\oplus\C} \\
      &\leq \frac12 \norm{(v,\mu)-(v',\mu')}_{\rho_{\ff}\rho_{\Omega^\complement}X_h^1\oplus\C}
  \end{align*}
  provided \(h\leq h_3\coloneqq(4 C'')^{-1}\). The conclusion thus holds for \(h_1\coloneqq\min(h_2,h_3)\).
\end{proof}

Applying the Banach fixed point theorem to the map \(S_h\) for each \(h\in(0,h_1]\) produces $v_h\in H^1(\R^d)$ and $\mu_h\in\R$ with
\begin{equation}
\label{EqTPfBd1}
  \norm{v_h}_{\rho_{\ff}\rho_{\Omega^\complement}X_h^1}+|\mu_h|\leq h
\end{equation}
such that
\begin{equation}
\label{EqTPfEq}
  P_h^\aug(v_h,\mu_h)=(-f_h+\mu_h v_h,0),\quad \text{that is,}\quad \begin{cases}
  \bigl(-\laplace+h^{-2}\chfun_{\Omega^\complement} - (\lambda_h+\mu_h)\bigr)(u_h+v_h) = 0, \\
    \iprod{v_h,u_0^\sharp}_{L^2(\Omega)} = 0. \end{cases}
\end{equation}

As discussed after~\eqref{EqTFin}, the remaining task is to improve on~\eqref{EqTPfBd1}; we do this by bootstrapping.

\pfstep{Step~2. Improving decay in \(h\).} Given~\eqref{EqTPfBd1} and \(f\in h^\infty\es\), we have
\[
  \norm{-f_h+\mu_h v_h}_{\rho_{\ff}^{-1}\rho_{\Omega^\complement}^{-1}X_h^{-1}} \leq C\bigl(\norm{f_h}_{\rho_{\ff}^{-1}\rho_{\Omega^\complement}^{-1}X_h^{-1}} + h^2\bigr) \leq C'h^2.
\]
Proposition~\ref{PropTFull} thus gives the improvement \(\norm{(v_h,\mu_h)}_{\rho_{\ff}\rho_{\Omega^\complement}X_h^1}\leq C''h^2\) over~\eqref{EqTPfBd1}. Iterating this gives
\begin{equation}
\label{EqTPfBd2}
  \norm{v_h}_{\rho_{\ff}\rho_{\Omega^\complement}X_h^1}+|\mu_h|\leq C_N h^N\quad\forall\,N.
\end{equation}

\pfstep{Step~3.1. Higher regularity away from \(\hat\rho=0\).} We use simple elliptic estimates. Concretely, using the notation \(\tilde\chi_\Omega\) from Lemma~\ref{LemmaTApprox}, we localize equation \((-\laplace-\lambda_h)v_h=-f_h+\mu_h(u_h+v_h)\), valid for \(\rho<0\), to
\[
  (-\laplace-\lambda_h)(\chi_\Omega v_h) = \chi_\Omega\bigl(-f_h+\mu_h(u_h+v_h)\bigr) + [-\laplace-\lambda_h,\chi_\Omega](\tilde\chi_\Omega v_h).
\]
Since \([-\laplace-\lambda_h,\chi_\Omega]\colon\rho H_{\rm z}^1(\Omega)\to\rho^{-1}H_{\rm z}^0(\Omega)\) is uniformly bounded, Lemma~\ref{LemmaTOmegaHi} gives, for any \(N\),
\[
  \norm{\chi_\Omega v_h}_{\rho H_{\rm z}^2(\Omega)} \leq \bigl(C_N h^N + \norm{\tilde\chi_\Omega v_h}_{\rho^{-1}H_{\rm z}^1(\Omega)}\bigr) + \norm{\chi_\Omega v_h}_{\rho H_{\rm z}^1(\Omega)}.
\]
The right-hand side is \(\leq C'_N h^N\) for all \(N\). Iterating this argument and relabeling the cutoff functions gives
\[
  \norm{\tilde\chi_\Omega v_h}_{\rho H_{\rm z}^k(\Omega)} \leq C_{N,k}h^N\quad\forall\,N,k.
\]
A completely analogous argument, now using \(P_h-\lambda_h\) and Lemma~\ref{LemmaTOmegaCHi}, gives higher regularity in the exterior region; to wit,
\begin{equation}
\label{EqTPfExt}
  \norm{\tilde\chi_{\Omega^\complement}v_h}_{H_h^k(\Omega^\complement)} \leq C_{N,k}h^N\quad\forall\,N,k.
\end{equation}

\pfstep{Step~3.2. Higher regularity near \(\hat\rho=0\).} It remains to control \(\psi_{\ff} v\) where \(\psi_{\ff}\coloneqq1-\chi_\Omega-\chi_{\Omega^\complement}\) vanishes near \(\widetilde\Omega\cup\widetilde\Omega^\complement\) and thus localizes near a compact subset of \({\ff}^\circ\). Consider then
\begin{equation}
\label{EqTPfff}
  h^2(P_h-\lambda_h)(\psi_{\ff} v_h) = h^2\psi_{\ff} \bigl( -f_h + \mu_h(u_h+v_h)\bigr) + [h^2(P_h-\lambda_h),\psi_{\ff}]v_h.
\end{equation}

Given the \(\delta\)-singularity of \(f_h\) along \(\hat\rho=0\), we certainly do not have \(H^2\)-membership of \(\psi_{\ff} v_h\). However, we can prove \emph{tangential} regularity relative to $H_{\ff,h}^1$ as well as higher \emph{one-sided} regularity in \(\pm\hat\rho\geq 0\). For the precise statement, define the norm
\[
  \norm{w}_{H_{{\ff},h,\pm}^k}^2 \coloneqq \sum_{j+m\leq k} \norm{\partial_{\hat\rho}^j(h\nabla_{\pa\Omega})^m w}_{L^2(\pm(0,\infty)\times\partial\Omega)}^2.
\]
We omit weights in \(\hat\rho_<\), \(\hat\rho_>\) since we only consider this norm for $w$ with with compact support in \(\hat\rho\) (such as $w=\psi_\ff v_h$).

\begin{lemma}[Improved regularity of \(\psi_{\ff} v_h\)]
\label{LemmaTPfImpr}
  For all \(k,N\in\N\), we have
  \begin{equation}
  \label{EqTPfImpr}
      \sum_{j=0}^{k-1} \norm{(h\nabla_{\partial\Omega})^j(\psi_{\ff} v_h)}_{H_{{\ff},h}^1} + \sum_\pm \norm{\psi_{\ff} v_h}_{H_{{\ff},h,\pm}^k} \leq C_{N,k}h^N.
  \end{equation}
\end{lemma}
\begin{proof}
  The case $k=1$ follows from~\eqref{EqTPfBd2}. Consider thus the case $k=2$. Since the coefficients of \([P_h-\lambda_h,\psi_{\ff}]\) are contained in \(\supp\tilde\chi_\Omega\cup\supp\tilde\chi_{\Omega^\complement}\), we have
  \[
    \norm{q_h}_{H_{{\ff},h}^k} \leq C_{N,k}h^N\quad\forall\,N,k,\qquad q_h\coloneqq[h^2(P_h-\lambda_h),\psi_{\ff}]v_h.
  \]
  Denoting by \(\psi_\partial=\psi_\partial(y)\) a cutoff to a coordinate chart on \(\partial\Omega\), say with \(|y|<1\), we localize~\eqref{EqTPfff} further: setting
  \[
    w_h\coloneqq\psi_\partial\psi_{\ff} v_h,
  \]
  we have
  \begin{equation}
  \label{EqTPfImprw}
    h^2(P_h-\lambda_h)w_h = -h^2\psi_{\ff}\psi_\partial f_h+h^2\psi_{\ff}\psi_\partial\mu_h(u_h+v_h) + \psi_\partial q_h+[h^2(P_h-\lambda_h),\psi_\partial](\psi_{\ff} v_h)
  \end{equation}
  All terms on the right, with the exception of the first, are bounded in \(H_{{\ff},h}^0\) by \(C_N h^N\). Consider now the local coordinate expression~\eqref{EqTApproxDelta}, so
  \[
    h^2(P_h-\lambda_h) = -\partial_{\hat\rho}^2 - g^{j k}(\rho,y)h\partial_{y^j}h\partial_{y^k} - h a(\rho,y)\partial_{\hat\rho} - h b^l(\rho,y)h\partial_{y^l} - h^2\lambda_h + H(\hat\rho).
  \]
  Upon acting on \(w_h\), let us move all terms except for the first two to the right-hand side; this gives
  \begin{equation}
  \label{EqTPfff2}
    (\partial_{\hat\rho}^2+g^{j k}(\rho,y)h\partial_{y^j}h\partial_{y^k})w_h = h^2\psi_{\ff}\psi_\partial f_h + \cO_{H_{{\ff},h}^0}(h^\infty).
  \end{equation}
  The usual proof of tangential regularity (approximating \(h\partial_{y^l}\) by finite differences, now rescaled by \(h\)) applies since differentiation in \(y\) preserves membership in the space \(h^\infty\es\) in which \(f\) lies; since upon application of \(h\partial_{y^l}\) (or the approximating finite differences) the right-hand side lies in \(\cO_{H_{{\ff},h}^{-1}}(h^\infty)\), we therefore obtain
  \begin{equation}
  \label{EqTPfImprTan}
    \norm{h\partial_{y^l}w_h}_{H_{{\ff},h}^1} \leq C_N h^N\quad\forall\,N.
  \end{equation}
  Since therefore \(h\partial_{y^j}h\partial_{y^k}w_h=\cO_{H_{{\ff},h}^0}(h^\infty)\), we deduce from~\eqref{EqTPfff2} that
  \[
    \partial_{\hat\rho}^2 w_h = h^2\psi_{\ff}\psi_\partial f_h + \cO_{H_{{\ff},h}^0}(h^\infty).
  \]
  This implies one-sided Sobolev bounds: since \(\norm{\psi_\ff\psi_\partial f_h}_{H_{{\ff},h,\pm}^k}\leq C_{N,k}h^N\) for all \(N,k\) (though at present we only need this for \(k=0\)), we obtain
  \begin{equation}
  \label{EqTPfImprOne}
    \norm{w_h}_{H_{{\ff},h,\pm}^2} \leq C_N h^N\quad\forall\,N.
  \end{equation}
  Summing the bounds~\eqref{EqTPfBd2}, \eqref{EqTPfImprTan}, and \eqref{EqTPfImprOne} over a cover of \(\partial\Omega\) by coordinate charts yields the estimate~\eqref{EqTPfImpr} for \(k=2\).

  Higher regularity follows inductively (up to shrinking the support of cutoff function \(\psi_{\ff}\)
  slightly, which we omit from our discussion) by the same token; thus, we shall be brief. Consider \(k\geq 3\), and consider again
  the equation~\eqref{EqTPfImprw}. Apply \(D\coloneqq h\partial_{y^{l_1}}\cdots
  h\partial_{y^{l_{k-1}}}\) to it. (More precisely, \(h\partial_{y^{l_1}}\) needs to be replaced by a
  finite difference quotient, but we omit this part of the argument.) Using the inductive hypothesis
  (i.e., \eqref{EqTPfImpr} for \(k-1\) in place of \(k\)), one then finds that \(D w_h\) solves
  \((\partial_{\hat\rho}^2+g^{j k}(\rho,y)h\partial_{y^j}h\partial_{y^k})(D
  w_h)=\cO_{H_{{\ff},h}^{-1}}(h^\infty)\). This gives
  \[
    \norm{(h\nabla_{\partial\Omega})^{k-1}(\psi_{\ff} v_h)}_{H_{{\ff},h}^1}\leq C_N h^N\quad\forall\,N.
  \]
  Using the PDE~\eqref{EqTPfImprw}, we obtain \(\partial_{\hat\rho}^2(D w_h)=\cO_{H_{{\ff},h}^0}(h^\infty)\), and therefore
  \[
    \norm{(h\nabla_{\partial\Omega})^{k-1}(\psi_{\ff} v_h)}_{H_{{\ff},h,\pm}^2} \leq C_N h^N\quad\forall\,N.
  \]
  Using~\eqref{EqTPfImprw}, we can continue trading tangential (\(h\nabla_{\partial\Omega}\)) by normal (\(\partial_{\hat\rho}\)) derivatives and thus obtain~\eqref{EqTPfImpr} as stated.
\end{proof}

We now apply Sobolev embedding to~\eqref{EqTPfImpr}: away from \(\hat\rho=0\), we have pointwise bounds \(|\partial^\alpha v_h|\leq C_{\alpha,N}h^N\) for all \(\alpha\in\N_0^d\). Near \(\hat\rho=0\) on the other hand, the infinite tangential regularity captured in~\eqref{EqTPfImpr} implies that Sobolev embedding in the \(\hat\rho\)-variable gives the continuity of \(v_h\) across \(\hat\rho=0\), while the infinite one-sided regularity gives smoothness from the left and the right. Thus:

\begin{cor}[Smoothness of \(v_h\)]
\label{CorTPfCI}
  \(v_h\) is continuous, and \(v_h|_{\rho<0}\in\cC^\infty(\bar\Omega)\) and \(v_h|_{\rho>0}\in\cC^\infty(\overline{\Omega^\complement})\) for all \(h\in(0,h_1]\); moreover, \(|v_h|\leq C_N h^N\) for all \(N\), and \(v_h|_{\pm\rho>0}\) and all its coordinate derivatives are bounded by \(C_N h^N\) for all \(N\).
\end{cor}

We also recall the uniform \(L^2\)-integrability of \(v_h\) (with \(h^N\) bounds for all \(N\)) in the exterior region from~\eqref{EqTPfExt}. We can use this and the equation satisfied by \(v_h\) there to improve the exterior decay of \(v_h\):

\begin{lemma}[Exterior decay of \(v_h\)]
\label{LemmaTPfExtDec}
  For \(\rho\geq\delta\) and for all \(\alpha\in\N_0^d\), we have
  \[
      |\partial^\alpha v_h(x)| \leq C_{\alpha,N}h^N\jbr{x}^{-N}.
  \]
\end{lemma}
\begin{proof}
  Let \(\psi\in\CI(\R^d)\) be equal to \(0\) for \(\rho\leq\delta/2\) and \(1\) for \(\rho\geq\delta\). Then the equation \((-h^2\laplace+1-h^2(\lambda_h+\mu_h))(u_h+v_h)=0\) satisfied by \(v_h\) in \(\rho>0\) implies
  \[
    \bigl(-h^2\laplace+1-h^2(\lambda_h+\mu_h)\bigr)(\psi v_h) = q_h \coloneqq [-h^2\laplace,\psi]v_h - \psi\bigl(-h^2\Delta+1-h^2(\lambda_h+\mu_h)\bigr)u_h.
  \]
  Note that \(\supp q_h\subset\supp\dd\psi\cup\supp u_h\) is contained in a fixed compact set for all $h$. Corollary~\ref{CorTPfCI} implies \(h^N\)
  bounds for \(q_h\) and all of its derivatives. Therefore, the Fourier transform\footnote{It would be
  more natural to utilize the semiclassical Fourier transform here. Due to the \(\cO(h^\infty)\)
size of all functions involved, working with the standard Fourier transform ultimately makes no
difference here.} \(\cF q_h\) of \(q_h\) is Schwartz, with all seminorms bounded by \(h^N\) for all \(N\);
and thus the same is true for \(\cF(\psi v_h)(\xi)=(h^2|\xi|^2+1-h^2(\lambda_h+\mu_h))^{-1}(\cF q_h)(\xi)\). This implies
\(|\partial^\alpha v_h|\leq C_{\alpha,N}h^N\jbr{x}^{-N}\) for all \(\alpha,N\).
\end{proof}

\pfstep{Step~3. Regularity in \(h\).} In the argument thus far, which involved only regularity considerations in the \(x\)-variable, we have only needed the bounds~\eqref{EqTPfBd1}; and we did not use any \(h\)-regularity of \(u_h,\lambda_h,f_h\) either. Recall however from the discussion following~\eqref{EqTEq2} that \(v_h\) and \(\mu_h\) are not unique if all we require is the validity of~\eqref{EqTEq2}. To prove regularity in \(h\), it is therefore necessary to use the \emph{full} equation~\eqref{EqTPfEq}. Formally differentiating~\eqref{EqTPfEq} along \(h\partial_h\) yields the system
\[
  P_h^\aug(\dot v_h,\dot\mu_h) = (-\dot f_h+\dot\mu_h v_h+\mu_h\dot v_h,0) - [h\partial_h,P_h^\aug](v_h,\mu_h)
\]
where \(\dot v_h=h\partial_h v\) etc. This is equivalent to
\begin{equation}
\label{EqTPfExth}
  \begin{pmatrix} -\laplace + h^{-2}\chfun_{\Omega^\complement} - (\lambda_h+\mu_h) & -(u_h+v_h) \\
      \iprod{\cdot,u_0^\sharp} & 0 \end{pmatrix}\begin{pmatrix} \dot v_h \\ \dot\mu_h \end{pmatrix}
      = \begin{pmatrix} -\dot f_h+(2 h^{-2}\chfun_{\Omega^\complement}+\dot\lambda_h)v_h + \mu_h\dot u_h
      \\ 0 \end{pmatrix}.
\end{equation}
In view of the infinite order of vanishing of \(\mu_h\) and \(v_h\) as \(h\searrow 0\), our analysis (starting with Proposition~\ref{PropTFull}) applies to the above equation as well; note also that \(\dot u_h\in h^\infty\qms\) and \(\dot f_h\in h^\infty\es\). Thus, \(\dot v_h\) satisfies the conclusions of Corollary~\ref{CorTPfCI} and Lemma~\ref{LemmaTPfExtDec}, and \(|\dot\mu_h|\leq C_N h^N\) for all \(N\). This argument can be made rigorous by considering the equation satisfied by the finite difference quotients \((v_{(1+\eta)h}-v_h)/\eta\) and \((\mu_{(1+\eta)h}-\mu_h)/\eta\) and letting \(\eta\searrow 0\).

In order to proceed, note again that the equation~\eqref{EqTPfExth} has the same structure as~\eqref{EqTPfEq}; since for the solution of~\eqref{EqTPfEq} we were just able to show one order of \(h\partial_h\)-regularity, the same arguments then apply also to~\eqref{EqTPfExth}. In other words, for the solution of~\eqref{EqTPfExth} we also have the \(\cO(h^\infty)\) bounds of Corollary~\ref{CorTPfCI} and Lemma~\ref{LemmaTPfExtDec} also for \(h\partial_h\dot v_h\) and \(h\partial_h\dot\mu_h\); and so on.

In view of the \(\cO(h^\infty)\) bounds on \(v_h,\mu_h\), infinite regularity with respect to \(h\partial_h\) is equivalent to infinite regularity with respect to \(\partial_h\). We therefore conclude that the function \(v\), defined by \(v_h\) on \(h\)-level sets of \(M\) for \(h\in(0,h_1]\), and the function $\mu$, defined by \(\mu(h)=\mu_h\), satisfy
\[
  v_h\in\cC^0(M),\ v_h|_{I^\circ}\in h^\infty\cC^\infty(I),\ v_h|_{E^\circ}\in h^\infty\cC^\infty(E),\quad
  \mu\in h^\infty\CI([0,h_1]).
\]
In other words, \(v_h\in h^\infty\qms'\) and \(\mu\in h^\infty\CI([0,h_1))\). This completes the proof of Theorem~\ref{ThmT}.

\appendix
\section{Some remarks on eigenvalues with multiplicity}\label{SRemarksMultiplicity}

In~\S\ref{SConstructionQuasimodes}, we proved the existence of quasimodes of infinite order for \emph{simple} eigenvalues. The treatment of eigenvalues of $-\dirlap{\Omega}$ with higher multiplicity is more delicate since a multiple eigenvalue could split into several distinct ones for \(h > 0\). If an eigenvalue $\lambda_0$ splits into at least two branches for \(h > 0\), this means that as we approach the final eigenspace at \(h = 0\), one might expect the eigenfunctions of \(P_h\) to converge to very particular subspaces of the $\lambda_0$-eigenspace of \(-\dirlap{\Omega}\). Conversely, this means that a quasimode or eigenfunction which restricts to $\wt\Omega$ as an eigenfunction \(u_0\) corresponding to $\lambda_0$   can only exist when $u_0$ lies in one of these subspaces. However, the splitting could, a priori, occur at any order of \(h\), and one would thus need to determine the coefficients and corrections up to this order to determine which linear combination one needed to choose for \(u_0\) in the initial step. We shall illustrate the analysis of the first order term in the following by using multiplicity 2 as a representative case.

Suppose \(\lambda_0\) is an eigenvalue of \(-\dirlap{\Omega}\) with multiplicity \(2\) and let
\(\psi_1,\psi_2\in \CI_0(\bar\Omega)\) be two corresponding \(L^2\)-normalized eigenfunctions which form
an orthonormal basis of the eigenspace associated to \(\lambda_0\). We start with the ansatz
\[
    u_0 = z_1\psi_1 + z_2\psi_2
\]
for some \(z_1, z_2\in \R\) with \(z_1^2 + z_2^2 = 1\) that are to be determined. As in \eqref{EqFirstOrderCorr} we obtain the following equation for the
first order correction pair \((w, \lambda_1)\)
\begin{equation}\label{EqMFirstOrder}
    (-\dirlap{\Omega} - \lambda_0)w = \lambda_1 u_0 + f
\end{equation}
where the function
\[
    f = (-\laplace - \lambda_0)F_\ff^\chi(\partial_\nu u_0) \in \CI(\bar\Omega)
\]
itself depends linearly on the eigenfunction \(u_0\) we choose. Let us introduce the operator \(T \coloneqq (-\laplace - \lambda_0)F_\ff^\chi(\partial_\nu \cdot)\), then \(f =
z_1T\psi_1 + z_2T\psi_2\).

A necessary and sufficient condition for \eqref{EqMFirstOrder} to have a solution is that the right-hand side be orthogonal to the kernel of \(-\dirlap{\Omega} - \lambda_0\). This amounts to
\[
    \iprod{f, \psi_i} + \lambda_1\iprod{u_0, \psi_i} = 0,\quad i =1,2.
\]
Expanding both \(u_0\) and \(f\) in terms of \(\psi_0, \psi_1\), this is equivalent to the following system of equations for \(z_1, z_2\) and \(\lambda_1\):
\begin{equation}\label{EqMPertSystem}
    \begin{aligned}
        z_1^2 + z_2^2 &= 1\\
        \iprod{T\psi_1, \psi_1}z_1 + \iprod{T\psi_2, \psi_1}z_2 + \lambda_1z_1 &= 0\\
        \iprod{T\psi_1, \psi_2}z_1 + \iprod{T\psi_2, \psi_2}z_2 + \lambda_1z_2 &= 0,
    \end{aligned}
\end{equation}
By essentially the same computation as in \eqref{EqQFirstOrderComp},
\[
    \iprod{T\psi_i, \psi_j} = \iprod{\partial_\nu \psi_i, \partial_\nu
    \psi_j}_{L^2(\pa\Omega)},\quad i,j\in \set{1,2}.
\]
Writing \(A\) for the symmetric \(2\times2\)-matrix with $(i,j)$-entry \(\iprod{\partial_\nu
\psi_i, \partial_\nu \psi_j}_{L^2(\pa\Omega)}\) for \(i, j\in\set{1,2}\), we see that
\eqref{EqMPertSystem} can be rewritten as an eigenvalue equation
\[
    A\begin{pmatrix}
        z_1\\z_2
    \end{pmatrix} = -\lambda_1\begin{pmatrix}
        z_1\\z_2
    \end{pmatrix},
\]
where \(-\lambda_1\) is an eigenvalue of the matrix \(A\) and \(z_1, z_2\) are the components of an
associated normalized eigenvector. This means that we can determine \(\lambda_1\) by finding the roots
of the characteristic polynomial of \(A\), which is a polynomial of degree \(2\); the components
\(z_1, z_2\) can then readily be determined by finding a normalized vector in the kernel of \(A +
\lambda_1I\). More generally, if we work with a multiplicity \(m\) eigenvalue of \(-\dirlap{\Omega}\), the
same computations will yield that \(\lambda_1\) is determined by finding a root of the characteristic
polynomial of degree \(m\) of a symmetric matrix.

In principle, it is possible for the eigenspace of \(-\lambda_1\) with respect to \(A\) to be \(2\)-dimensional so that every normalized vector \((z_1, z_2)\in\R^2\) solves~\eqref{EqMPertSystem}. On the level of the operator \(P_h\), this means that the eigenvalues of \(P_h\) converging to \(\lambda_0\) as \(h\to 0\) can split at most at order \(\cO(h^2)\) (if at all); and one can construct $\cO(h^2)$ quasimodes starting with any choice of $u_0$. To determine the coefficients \(z_1, z_2\) for which $\cO(h^3)$ quasimodes exist, one would then have to move on to analyze the \(h^2\)-order analogue of \eqref{EqMFirstOrder}, and so on.

Suppose that \(-\lambda_1\) is a double eigenvalue of \(A\) and write \(A_{ij} = \iprod{\pa_\nu
\psi_i, \pa_\nu\psi_j}_{L^2(\pa\Omega)}\) for \(i, j =1,2\). Then we must have the equality of polynomials
\[
    (x - A_{11})(x - A_{22}) - A_{12}^2 = (x + \lambda_1)^2.
\]
Expanding and comparing coefficients yields
\[
    \lambda_1 = -\frac{A_{11} + A_{22}}{2}\quad\text{ and }\quad \lambda_1^2 =
    A_{11}A_{22} - A_{12}^2.
\]
Plugging \(\lambda_1\) from the first into the second equation, one finds
\[
    0\le (A_{11} - A_{22})^2 = -A_{12}^2\le 0.
\]
This is only possible if both sides are \(0\) and thus \(A_{11} = A_{22}\) and \(A_{12} = A_{21} = 0\), i.e., the normal derivatives of the \(\psi_1\) and \(\psi_2\) are \(L^2\)-orthogonal functions on \(\pa\Omega\) and the \(L^2\)-norms of their normal derivatives are the same.

\begin{qu}[Orders of non-trivial splitting]
    For what values of $N\in\N\cup\{+\infty\}$ is it possible for two eigenvalues \(\lambda_j^h\) and \(\lambda_{j+1}^h\) of \(P_h\) to be such that
    \(\lim_{h\to 0}\lambda_j^h = \lim_{h\to 0}\lambda_{j+1}^h\) and $\lambda_{j+1}^h-\lambda_j^h=\cO(h^N)$?
\end{qu}

\section{Eigenvalues for intervals}\label{SAppIntervals}
Here we want to derive a secular equations for the eigenvalues \(\lambda_n^h\) when \(\Omega\) is an
interval in \(\R\). By shifting \(\Omega\) if necessary, we may assume that \(\Omega = (-a, a)\) for some
\(a > 0\). We then want to find \(\lambda \in (0, h^{-2})\) so that there exists a solution \(u\in
\cC^1(\R)\cap L^2(\R)\) to
\begin{equation}\label{EqEigfunc1D}
    \begin{cases}
        -\pa_x^2 u + h^{-2}u = \lambda u &\text{on }(-\infty, -a],\\
        -\pa_x^2 u = \lambda u &\text{on }(-a, a),\\
        -\pa_x^2 u + h^{-2}u = \lambda u &\text{on }[a, \infty).
    \end{cases}
\end{equation}
Taking into account that \(u\in L^2(\R)\), we must have
\begin{equation}\label{EqInterval}
    u(x) = \begin{cases}
        A\exp\left((h^{-2} - \lambda)^{\frac12}x\right) & \text{if }x\in (-\infty, -a],\\
        B\cos(\lambda^{\frac12}x) + C\sin(\lambda^{\frac12}x) & \text{if }(-a, a),\\
        D\exp\left(-(h^{-2} - \lambda)^{\frac12}x\right) & \text{if }x\in [a, \infty),
    \end{cases}
\end{equation}
where the coefficients \(A, B, C, D\in\R\) still need to be determined. Since \(u\in \cC^1(\R)\), we must
have
\begin{equation}\label{EqBoundaryConditions}
    \begin{aligned}
        \lim_{x\searrow -a}u(x) &= \lim_{x\nearrow -a} u(x), &\qquad \lim_{x\searrow a}u(x) &= \lim_{x\nearrow a} u(x), \\
        \lim_{x\searrow -a}u^\prime(x) &= \lim_{x\nearrow -a} u^\prime(x), &\qquad \lim_{x\searrow a}u^\prime(x) &= \lim_{x\nearrow a} u^\prime(x).
    \end{aligned}
\end{equation}
Note that \eqref{EqEigfunc1D} is invariant under the transformations \(u(x)\mapsto u(-x)\) and
\(u(x)\mapsto -u(-x)\). Thus, any solution \(u\) maps to another solution under these transformations.
It hence suffices to look for solutions \(u\) which are invariant under either transformation. If \(u\)
is invariant under the first transformation, it is even, so \(A = D\) and \(C = 0\). Using
\eqref{EqBoundaryConditions}, we then obtain the linear system of equations
\[
    \begin{pmatrix}
        \exp(-(h^{-2}-\lambda)^{\frac12}a) & -\cos(\lambda^{\frac12}a)\\
        -(h^{-2}-\lambda)^{\frac12}\exp(-(h^{-2}-\lambda)^{\frac12}a) & \lambda^{\frac12}\sin(\lambda^{\frac12}a)
    \end{pmatrix}
    \begin{pmatrix}
        A\\B
    \end{pmatrix}
     =
    \begin{pmatrix}
        0\\0
    \end{pmatrix}.
\]
Computing the determinant and simplifying yields the equation
\[
    (h^{-2} - \lambda)^{\frac12}\cos(\lambda^{\frac12}a) - \lambda^{\frac12}\sin(\lambda^{\frac12}a) = 0,
\]
which \(\lambda\in (0,h^{-2})\) has to satisfy in order for it to be an eigenvalue corresponding to an
even eigenfunction. We call the left-hand side of this equation the \emph{even secular function}.

Similarly, if \(u\) is invariant under the second transformation, it is odd and thus \(A = -D\) and \(B
= 0\). Using \eqref{EqBoundaryConditions} yields
\[
    \begin{pmatrix}
        \exp(-(h^{-2}-\lambda)^{\frac12}a) & \sin(\lambda^{\frac12}a)\\
        (h^{-2}-\lambda)^{\frac12}\exp(-(h^{-2}-\lambda)^{\frac12}a) & -\lambda^{\frac12}\cos(\lambda^{\frac12}a)
    \end{pmatrix}
    \begin{pmatrix}
        A\\C
    \end{pmatrix}
     =
    \begin{pmatrix}
        0\\0
    \end{pmatrix}.
\]
Again computing the determinant and simplifying yields the equation
\[
    \lambda^{\frac12}\cos(\lambda^{\frac12}a) + (h^{-2} - \lambda)^{\frac12}\sin(\lambda^{\frac12}a) = 0.
\]
We call the left-hand side of the previous equation the \emph{odd secular
function}.

Thus, in total, the eigenvalues \(\lambda\) can be found by determining the zeros of the \emph{secular
function}
\[
    \begin{aligned}
        S_h(\lambda) &= \left((h^{-2} - \lambda)^{\frac12}\cos(\lambda^{\frac12}a) -
    \lambda^{\frac12}\sin(\lambda^{\frac12}a)\right)\left(\lambda^{\frac12}\cos(\lambda^{\frac12}a) + (h^{-2} -
    \lambda)^{\frac12}\sin(\lambda^{\frac12}a)\right)\\
    &= \lambda^{\frac12}(h^{-2} - \lambda)^{\frac12}\cos(2\lambda^{\frac12}a) + \frac{h^{-2} -
    2\lambda}{2}\sin(2\lambda^{\frac12}a),
    \end{aligned}
\]
which is the product of the even and the odd secular function.

The Dirichlet Laplace eigenvalues on the interval \((-a, a)\) are easily computed to be
\[
    \spec\left(-\dirlap{(-a, a)}\right) = \set*{\frac{\pi^2k^2}{4a^2} : k\in \N};
\]
all of these eigenvalues have multiplicity \(1\).

\section{Eigenvalues for disks}\label{SAppDisks}

We now move one dimension up and provide a family of secular
functions for a disk \(B_a(0) = \set{x\in \R^2 : \norm{x} < a}\subset \R^2\) of radius \(a > 0\). We want
to find \(\lambda\in (0,h^{-2})\) so that there exists a solution \(u\in \cC^1(\R^2)\cap L^2(\R^2)\) to
\begin{equation}\label{EqEigfunc2D}
    \begin{cases}
        -\laplace u = \lambda u & \text{on }B_a(0)\\
        -\laplace u = (\lambda - h^{-2}) u & \text{on }B_a(0)^\complement.
    \end{cases}
\end{equation}
Let \((r, \varphi)\) denote polar coordinates. Since this problem is radially symmetric, we separate
variables in the polar coordinates and write \(u(r, \varphi) = U(r)Y(\varphi)\) for
\(U\colon [0,\infty)\to \R\) and a \(2\pi\)-periodic function \(Y\colon [0, 2\pi] \to \R\). Using the
polar form of the Laplacian, we find that \eqref{EqEigfunc2D} can be rewritten as
\[
    \begin{cases}
        r^2\frac{\pa_r^2U}{U} + r\frac{\pa_r U}{U} + \lambda r^2 = -\frac{\pa_\varphi^2 Y}{Y} &
        \text{for } r\in [0, a),\\
        r^2\frac{\pa_r^2U}{U} + r\frac{\pa_r U}{U} - (h^{-2} - \lambda) r^2 = -\frac{\pa_\varphi^2 Y}{Y} &
        \text{for } r\in [a, \infty).
    \end{cases}
\]
(Since $u\in\cC^1(\R^2)$, we do not need to consider the case where $Y$ is different for $r<a$ and $r\geq a$.) The right-hand sides must equal to some eigenvalue of the operator \(-\pa_\varphi^2\) on the interval \([0, 2\pi]\) with periodic boundary conditions, and thus equal to \(\nu^2\) for some \(\nu\in\Z\). Thus, there exists $\nu\in\N_0$ with
\[
    \begin{cases}
        r^2\pa_r^2U + r\pa_r U + (\lambda r^2 - \nu^2)U = 0 &
        \text{for } r\in [0, a),\\
        r^2\pa_r^2U + r\pa_r U - ((h^{-2} - \lambda) r^2 + \nu^2)U = 0 &
        \text{for } r\in [a, \infty).
    \end{cases}
\]
These are readily seen to be rescaled versions of the Bessel and modified Bessel differential
equation for the interior and the exterior, respectively. Taking into account that \(u\in L^2(\R^2)\),
we can make the ansatz
\[
    U(r) = \begin{cases}
        AJ_\nu(\lambda^{\frac12} r) & \text{for }r\in [0,a), \\
        BK_\nu((h^{-2} - \lambda)^{\frac12} r) & \text{for }r\in [a,\infty),
    \end{cases}
\]
where \(J_\nu\) denotes the \(\nu\)-th order Bessel function of the first kind, \(K_\nu\) denotes the
\(\nu\)-th order modified Bessel function of the second kind and \(A, B\in\R\) are some parameters that
are to be determined. Again invoking our requirement \(u\in\cC^1(\R^2)\), we must have that the two
parts are continuously differentiable at \(r = a\), so setting functions and derivatives equal we
obtain the equations
\[
    \begin{aligned}
        AJ_\nu(\lambda^{\frac12}a) &= BK_\nu((h^{-2} - \lambda)^{\frac12}a),\\
        A\lambda^{\frac12}J_\nu^\prime(\lambda^{\frac12}a) &= B(h^{-2} - \lambda)^{\frac12}K_\nu^\prime((h^{-2}
        - \lambda)^{\frac12}a).
    \end{aligned}
\]
This has a solution if and only if the determinant vanishes, that is,
\[
    \lambda^{\frac12}J_\nu^\prime(\lambda^{\frac12}a)K_\nu((h^{-2} - \lambda)^{\frac12}a)
    - (h^{-2} - \lambda)^{\frac12}J_\nu(\lambda^{\frac12}a)K_\nu^\prime((h^{-2} - \lambda)^{\frac12}a) = 0.
\]
These give us a family of secular functions indexed by \(\nu\in \N_0\).

By a similar calculation, one can show that
\[
    \spec\left(-\dirlap{B_a(0)}\right) = \set*{\frac{j_{l, \nu}^2}{a^2}: j_{l, \nu} \text{ is the
    \(l\)-th positive root of } J_\nu},
\]
where every eigenvalue has multiplicity \(1\) if it corresponds to a zero of \(J_0\) and else
multiplicity \(2\). Moreover, by rotational symmetry considerations, the corresponding eigenvalues of the particle-in-well operator have the same multiplicity structure.

\section{Eigenvalues for \texorpdfstring{$d$}{d}-balls}\label{SAppBalls}

Let now again \(a>0\) and \(d\in \N\) with \(d \ge 3\). For completeness, we shall also give the secular
function for eigenvalues of \(P_h\) for balls \(B_a(0)\subset \R^d\). The Laplacian can then be
represented in spherical coordinates \((r, \omega)\in (0,\infty)\times \mathbb{S}^{d-1}\) as
\[
    \laplace = \pa_r^2 + \frac{d - 1}{r}\pa_r + \frac{1}{r^2}\laplace_{\mathbb{S}^{d-1}},
\]
where \(\laplace_{\mathbb{S}^{d-1}}\) is the Laplace Beltrami operator on the sphere
\(\mathbb{S}^{d-1}\). Separating \(u(r, \omega) = U(r)Y(\omega)\) for some \(U\colon (0,\infty)\to \R\)
and \(Y\colon \mathbb{S}^{d-1}\to \R\), we obtain, as in the two-dimensional case, the equations
\[
    \begin{cases}
        r^2\frac{\pa_r^2U}{U} + (d - 1)r\frac{\pa_r U}{U} + \lambda r^2 =
        \frac{-\laplace_{\mathbb{S}^{d-1}} Y}{Y} &
        \text{for } r\in [0, a)\\
        r^2\frac{\pa_r^2U}{U} + (d - 1)r\frac{\pa_r U}{U} - (h^{-2} - \lambda) r^2 = \frac{-\laplace_{\mathbb{S}^{d-1}} Y}{Y} &
        \text{for } r\in [a, \infty).
    \end{cases}
\]
Again, both equations must be constant and \(\cC^1\) on \(\pa B_a(0)\) and thus both left-hand sides
must equal to the same eigenvalue of \(-\laplace_{\mathbb{S}^{d-1}}\), which are given by
\[
    \spec\left(-\laplace_{\mathbb{S}^{d-1}}\right) = \set*{\nu(\nu + d - 2): l\in \N_0}.
\]
Thus, for a fixed \(\nu\in\N_0\) we want to solve
\[
    \begin{cases}
        r^2\pa_r^2U + (d-1)r\pa_r U + (\lambda r^2 - \nu(\nu + d - 2))U = 0 &
        \text{for } r\in [0, a),\\
        r^2\pa_r^2U + (d-1)r\pa_r U - ((h^{-2} - \lambda) r^2 + \nu(\nu + d - 2))U = 0 &
        \text{for } r\in [a, \infty).
    \end{cases}
\]
One can check that these equations as well as the requirement that \(u\in \cC^1(\R^d)\cap L^2(\R^d)\) are satisfied if
\[
    U(r) = \begin{cases}
        A\mathfrak{J}_{d, \nu, \lambda}(r) & \text{for }r\in [0,a),\\
        B\mathfrak{K}_{d, \nu, \lambda, h}(r) & \text{for }r\in [a,\infty),
    \end{cases}
\]
where
\[
    \mathfrak{J}_{d, \nu, \lambda}(r) = r^{1 - d/2}J_{\nu + d/2 - 1}(\lambda^{\frac12}
    r)\quad\text{and}\quad
\mathfrak{K}_{d, \nu, \lambda, h}(r) = r^{1 - d/2}K_{\nu + d/2 - 1}((h^{-2} - \lambda)^{\frac12} r),
\]
and \(A, B\in\R\) are coefficients that are to be determined. Continuing as in the $2$-dimensional
case, we find the following family of secular functions indexed by \(\nu\in\N_0\):
\[
    \mathfrak{J}_{d, \nu, \lambda}^\prime(a)\mathfrak{K}_{d, \nu, \lambda, h}(a) - \mathfrak{J}_{d,
    \nu, \lambda}(a)\mathfrak{K}_{d, \nu, \lambda, h}^\prime(a) = 0.
\]

Again the eigenvalues of the Dirichlet Laplacian on \(B_a(0)\) and the eigenvalues of the
particle-in-well operators have the same multiplicity structure.

\bibliographystyle{alphaurl}
\bibliography{bibliography}

\end{document}